\documentclass[a4paper,12pt]{amsart}
\usepackage{amsfonts,amsmath,amssymb,euscript,inputenc}
\usepackage{graphicx,psfrag,subfigure,xcolor}

\title[Wave and Klein--Gordon equations on hyperbolic spaces]
{Wave and Klein--Gordon equations\\on hyperbolic spaces}

\author{Jean--Philippe Anker}
 
\address{Universit\'e d'Orl\'eans \& CNRS,
F\'ed\'eration Denis Poisson (FR 2964) \& Laboratoire MAPMO (UMR 6628),
B\^atiment de math\'ematiques -- Route de Chartres,
B.P. 6759 -- 45067 Orl\'eans cedex 2 -- France}
 
\email{anker@univ-orleans.fr}

\author{Vittoria Pierfelice}

\address{Universit\'e d'Orl\'eans \& CNRS,
F\'ed\'eration Denis Poisson (FR 2964) \& Laboratoire MAPMO (UMR 6628),
B\^atiment de math\'ematiques -- Route de Chartres,
B.P. 6759 -- 45067 Orl\'eans cedex 2 -- France}
 
\email{vittoria.pierfelice@univ-orleans.fr}

\date{\today}



\newtheorem{lemma}{Lemma}[section]
\newtheorem{theorem}[lemma]{Theorem}

\newtheorem{corollary}[lemma]{Corollary}
\newtheorem{remark}[lemma]{Remark}
\newtheorem{definition}[lemma]{Definition}

\newcommand\1{1\hskip-1mm\text{\rm I}}
\newcommand{\bc}{\mathbf{c}\hspace{.1mm}}
\newcommand{\C}{\mathbb{C}}
\newcommand\const{\operatorname{const.}}
\newcommand\gammazero{\gamma_{\ssf0}}
\newcommand\gammaone{\gamma_{\ssf1}}
\newcommand\gammaconf{\gamma_{\ssf\text{\rm conf}}}
\newcommand\gammatwo{\gamma_{\ssf2}}

\newcommand\gammathree{\gamma_{\ssf3}}
\newcommand\gammafour{\gamma_{\ssf4}}
\newcommand{\Hn}{\mathbb{H}^n}
\renewcommand\Im{\operatorname{Im}}
\newcommand{\N}{\mathbb{N}}
\newcommand\ptwo{p_{\ssf2}}
\newcommand\qtwo{q_{\ssf2}}
\newcommand\tildeqtwo{\tilde{q}_{\ssf2}}

\newcommand\qthree{q_{\ssf3}}
\newcommand\tildeqthree{\tilde{q}_{\ssf3}}
\renewcommand\Re{\operatorname{Re}}
\newcommand{\R}{\mathbb{R}}
\newcommand{\Rn}{\R^n}
\newcommand{\Sn}{\mathbb{S}^{n-1}}
\newcommand\ssb{\hskip-.25mm}
\newcommand\ssf{\hskip.25mm}
\newcommand\supp{\operatorname{supp}}
\newcommand{\Z}{\mathbb{Z}}

\begin{document}

\begin{abstract}
We consider the Klein--Gordon equation
associated with the Laplace--Beltrami operator \ssf$\Delta$
\ssf on real hyperbolic spaces of dimension \ssf$n\!\ge\!2$\ssf;
as \ssf$\Delta$ \ssf has a spectral gap,
the wave equation is a particular case of our study.
After a careful kernel analysis,
we obtain dispersive and Strichartz estimates
for a large family of admissible couples.
As an application, we prove global well--posedness results
for the corresponding semilinear equation with low regularity data.
\end{abstract}

\maketitle

\section{Introduction}
\label{Introduction}

Dispersive properties of the wave and other evolution equations
have been proved very useful in the study of nonlinear problems.
The theory is well established
for the Euclidean wave equation in dimension \ssf$n\!\ge\!3$\,:
\begin{equation}\label{WaveEuclidean}
\begin{cases}
\;\partial_{\ssf t}^{\ssf 2}u(t,x)-\Delta_{\ssf x}u(t,x)=F(t,x)\ssf,\\
\;u(0,x)=f(x)\ssf,
\;\partial_{\ssf t}|_{t=0}\,u(t,x)=g(x)\ssf.\\
\end{cases}
\end{equation}
The following Strichartz estimates hold for
solutions \ssf$u$ to the Cauchy problem \eqref{WaveEuclidean}\,:
\begin{equation*}
\|\ssf\nabla_{\R\times\Rn}u\ssf\|_{L^p(I;\ssf\dot{H}^{-\sigma,q}(\Rn))}\ssb
\lesssim\ssf\|\ssf f\ssf\|_{\dot{H}^1(\Rn)}\ssb
+\ssf\|\ssf g\ssf\|_{L^2(\Rn)\vphantom{\dot{H}^1}}\ssb
+\ssf\|\ssf F\ssf\|_{L^{\tilde{p}'}(I;\ssf\dot{H}^{\tilde{\sigma},\,\tilde{q}'}(\Rn))}    
\end{equation*}
on any (possibly unbounded) time interval \ssf$I\!\subseteq\ssb\R$\ssf,
under the assumptions that  
\begin{equation*}\textstyle
\sigma\ge\frac{n+1}2\ssf\bigl(\frac12\ssb-\ssb\frac1q\bigr)\,,\quad 
\tilde{\sigma}\ge\frac{n+1}2\ssf\bigl(\frac12\ssb-\ssb\frac1{\tilde{q}}\bigr)\,,
\end{equation*}
and the couples \ssf
$(p,q),(\tilde{p},\tilde{q})\ssb\in\ssb[\ssf2,\infty\ssf]\!\times\![\ssf2,\infty\ssf)$
\ssf satisfy 
\begin{equation}\label{EuclideanAdmissibility}\textstyle
\frac2p\ssb+\ssb\frac{n-1}q\ssb=\ssb\frac{n-1}2\ssf,\quad
\frac2{\tilde{p}}\ssb+\ssb\frac{n-1}{\tilde{q}}\ssb=\ssb\frac{n-1}2\ssf.
\end{equation}
We refer to \cite{GV} and \cite{KT} for more details.

These estimates serve as a tool for several existence results
about the nonlinear wave equation in the Euclidean setting.
The problem of finding minimal regularity conditions on the initial data
ensuring local well--posedness for semilinear wave equations
was addressed in \cite{Kap} and then almost completely answered in \cite{LS, KT}
(see Figure \ref{LocalRegularityEuclidean} in Section \ref{GWP}).
In general local solutions cannot be extended to global ones,
unless further assumptions are made on the nonlinearity or on the initial data.
A successful machinery was developed
towards the global existence of small solutions
to the semilinear wave equation
\begin{equation}\label{SLWEuclidean}
\begin{cases}
\;\partial_{\ssf t}^{\ssf 2} u(t,x)-\Delta_{\ssf x}u(t,x)=F(u)\ssf,\\
\;u(0,x)=f(x)\ssf,
\;\partial_{\ssf t}|_{t=0}\,u(t,x)=g(x)\ssf,\\
\end{cases}
\end{equation}
with nonlinearities
\begin{equation}\label{eq:nonlinin}
F(u)\sim|u|^{\gamma}\quad\text{near 0.}
\end{equation}
The results depend on the space dimension $n$\ssf.
After the pioneer work \cite{J} of John in dimension $n\!=\!3$\ssf,
Strauss conjectured in \cite{Stra}
that the problem \eqref{SLWEuclidean}
is globally well posed  in dimension $n\!\ge\!2$
for small initial data provided
\begin{equation}\label{strauss}
\textstyle
\gamma>\gammazero=\frac12\ssb+\ssb\frac1{n-1}
+\sqrt{\bigl(\frac12\ssb+\ssb\frac1{n-1}\bigr)^2\!+\ssb\frac2{n-1}\,}\ssf.
\end{equation}
On one hand,
the negative part of the conjecture
was established by Sideris \cite{Si},
who proved blow up for generic data
(and nonlinearities satisfying $F(u)\!\gtrsim\!|u|^{\gamma}$)
when $\gamma\!<\!\gammazero\ssf$.
On the other hand,
the positive part of the conjecture
was proved for any dimension in several steps
(see e.g.~\cite{KP}, \cite{GLS}, \cite{DGK},
and \cite{G} for a comprehensive survey).

Analogous results hold for the Klein--Gordon equation
$$
\partial_{\ssf t}^{\ssf 2}u(t,x)-\Delta_{\ssf x}u(t,x)+u(t,x)=F(t,x)\ssf,
$$
though its study has not been carried out as thoroughly as for the wave equation\,;
in particular the sharpness of several well--posedness results is yet unknown
(see \cite{BG}, \cite{GV1}, \cite{MNO}, \cite{N} and the references therein).

In view of the rich Euclidean theory,
it is natural to look at the corresponding equations
on more general manifolds.
Here we consider real hyperbolic spaces \ssf$\Hn$,
which are the most simple examples of
noncompact Riemannian manifolds with negative curvature.
For geometric reasons,
we expect better dispersive properties hence stronger results
than in the Euclidean setting.

Consider the wave equation associated to
the Laplace--Beltrami operator \ssf$\Delta\ssb=\ssb\Delta_{\Hn}$
on \ssf$\Hn$\,:
\begin{equation}\label{nonshiftedWave}
\begin{cases}
\;\partial_{\ssf t}^{\ssf 2}u(t,x)-\Delta_{\ssf x}u(t,x)=F(t,x)\ssf,\\
\;u(0,x)=f(x)\ssf,
\;\partial_t|_{t=0}\,u(t,x)=g(x)\ssf,
\end{cases}
\end{equation}
The operator $-\ssf\Delta$ is positive on $L^2(\Hn)$
and its $L^2$--\ssf spectrum
is the half--line \ssf$[\ssf\rho^{\ssf2},+\infty)\ssf$,
where \ssf$\rho\ssb=\!\frac{n-1}2$.
Thus \eqref{nonshiftedWave} may be considered
as a special case of the following family of Klein--Gordon equations
\begin{equation}\label{KleinGordonC}
\begin{cases}
\;\partial_{\ssf t}^{\ssf 2}u(t,x)-\Delta_{\ssf x}u(t,x)+c\,u(t,x)=F(t,x)\ssf,\\
\;u(0,x)=f(x)\ssf,
\;\partial_t|_{t=0}\,u(t,x)=g(x)\ssf,
\end{cases}\end{equation}
where
\begin{equation}\label{constantc}\textstyle
c\ge\ssb-\ssf\rho^{\ssf2}=\ssb-\ssf\frac{(n-1)^2}4
\end{equation}
is a constant.
In the limit case \ssf$c\ssb=\!-\ssf\rho^{\ssf2}$\ssb,
\eqref{KleinGordonC} is
called the {\it shifted\/} wave equation.

In \cite{P} Pierfelice obtained Strichartz estimates
for the non--shifted wave equation \eqref{nonshiftedWave}
with radial data on a class of Riemannian manifolds containing all hyperbolic spaces.
The wave equation \eqref{nonshiftedWave} was also investigated
on the 3--dimensional hyperbolic space by Metcalfe and Taylor \cite{MT},
who proved dispersive and Strichartz estimates
with applications to small data global well--posedness
for the semilinear wave equation.
In his recent thesis \cite{Ha},
Hassani obtains a first set of results
on noncompact Riemannian symmetric spaces of higher rank.

To our knowledge, the shifted wave equation \eqref{KleinGordonC}
in the limit case \ssf$c\ssb=\!-\ssf\rho^{\ssf2}$
\ssf was first considered by Fontaine \cite{F1, F2}
in low dimensions \ssf$n\!=\!3$ \ssf and \ssf$n\!=\!2$\ssf.
In \cite{Ta} Tataru obtained dispersive estimates for the operators
$\frac{\sin\ssf\left(t\,\sqrt{\Delta\ssf+\ssf\rho^{\ssf2}\ssf}\ssf\right)}
{\sqrt{\Delta\ssf+\ssf\rho^{\ssf2}\ssf}}$
and $\cos\ssf\bigl(\ssf t\ssf\sqrt{\Delta\ssb+\ssb\rho^{\ssf2}\ssf}\ssf\bigr)$
acting on inhomogeneous Sobolev spaces on \ssf$\Hn$
and then transferred them to \ssf$\Rn$
in order to get well--posedness results
for the Euclidean semilinear wave equation
(see also \cite{G}).
Complementary results were obtained by A. Ionescu \cite{I1},
who investigated $L^q\!\to\!L^q$ Sobolev estimates
for the above operators on all hyperbolic spaces.

A more detailed analysis of the shifted wave equation
was carried out in \cite{APV2}.
There Strichartz estimates were obtained
for a wider range of couples than in the Euclidean setting
and consequently
stronger well--posedness results were shown to hold
for the nonlinear equations.
Corresponding results for the Schr\"odinger equation
were obtained in \cite{AP}, \cite{APV1} and \cite{IS}.

In the present paper we study
the family of equations \eqref{KleinGordonC}
in the remaining range \ssf$c\ssb>\!-\ssf\rho^{\ssf2}$
and in dimension \ssf$n\!\ge\!2$\ssf,
which includes the particular case
\ssf$c\!=\!0$ \ssf and \ssf$n\!=\!3$ \ssf
considered in \cite{MT}.
In order to state and describe our results,
it is convenient to rewrite the constant \eqref{constantc} as follows\,:
\begin{equation}\label{constantkappa}
c=\kappa^{2}\ssb-\rho^{\ssf2}
\quad\text{ \;with}\quad\kappa\!>\!0\ssf,
\end{equation}
and to introduce the operator
\begin{equation}\label{operatorD}
D=\ssb\sqrt{-\ssf\Delta\ssb-\ssb\rho^{\ssf2}\!+\ssb\kappa^2\ssf}\,,
\end{equation}
as well as
\begin{equation}\label{operatorDtilde}
\widetilde{D}
=\ssb\sqrt{-\ssf\Delta\ssb-\ssb\rho^{\ssf2}\!+\ssb\tilde\kappa^2\ssf}\,,
\end{equation}
where \ssf$\tilde\kappa\!>\!\rho$ \ssf is another fixed constant.
Thus our family of equations \eqref{KleinGordonC} becomes
\begin{equation}\label{IP}
\begin{cases}
\;\partial_{\ssf t}^{\,2}u(t,x)-D_x^{\ssf2}\ssf u(t,x)=F(t,x)\ssf,\\
\;u(0,x)=f(x)\ssf,
\;\partial_t|_{t=0}\,u(t,x)=g(x)\ssf,
\end{cases}\end{equation}
the wave equation \eqref{nonshiftedWave}
corresponding to the choice \ssf$\kappa\ssb=\ssb\rho$
\ssf and the shifted wave equation
to the limit case \ssf$\kappa\ssb=\ssb0$\ssf.

Let us now describe the content of this paper
and present our main results,
that we state for simplicity in dimension \ssf$n\!\ge\!3$\ssf.
In Section \ref{Notation},
we recall the basis tools of spherical Fourier analysis
on real hyperbolic spaces $\Hn$.
After analyzing carefully the integral kernel of the half wave operator
\begin{equation*}%\label{Wt}
W_t^{\ssf\sigma}=\widetilde{D}^{-\sigma}\ssf e^{\ssf i\ssf t\ssf D}
\end{equation*}
in Section \ref{Kernel},
we prove in Section \ref{Dispersive} the following dispersive estimates,
which combine the small time estimates \cite{APV2} for the shifted wave equation
and the large time estimates \cite{AP} for the Schr\"odinger equation\,:
\begin{equation*}
\bigl\|\,W_t^{\ssf\sigma}\ssf\bigr\|_{L^q\to L^{q'}}
\lesssim\,\begin{cases}
\;|t|^{-(n-1)\ssf(\frac12-\frac1q)}
&\text{if \;}0\!<\!|t|\!<\!1\ssf,\\
\;|t|^{-\frac32}
&\text{if \;}|t|\!\ge\!1\ssf,
\end{cases}\end{equation*}
where \ssf$2\!<\!q\!<\!\infty$ \ssf and
\ssf$\sigma\!\ge\!(n\!+\!1)\ssf(\frac12\!-\!\frac1q)$\ssf.
%These estimates hold in dimension \ssf$n\!\ge\!3$\,;
%in dimension \ssf$n\!=\!2$\ssf,
%the small time bound contains an additional logarithmic factor.
Notice that we don't deal with the limit case
\linebreak
\ssf$q\!=\!\infty$\ssf,
where Metcalfe and Taylor \cite{MT} have obtained
an \ssf$H^1\hspace{-1mm}\to\!BMO$ \ssf estimate
in dimension \ssf$n\!=\!3$\ssf.

In Section \ref{Strichartz}
we deduce the Strichartz estimates
\begin{equation*}%\label{StrichartzEstimate1}
\|\ssf\nabla_{\R\times\Hn}u\ssf\|_{\vphantom{\big|}
L^p(I;\ssf H^{-\sigma,\hspace{.1mm}q}(\Hn))}
\lesssim\,\|\ssf f\ssf\|_{H^1(\Hn)\vphantom{\big|}}\ssb
+\,\|\ssf g\ssf\|_{L^2(\Hn)\vphantom{\big|}}\ssb
+\,\|\ssf F\ssf\|_{\vphantom{\big|}
L^{\tilde{p}'}\!(I;\ssf H^{\tilde{\sigma}\ssb,\tilde{q}'}\!(\Hn))}
\end{equation*}
for solutions \ssf$u$ to \eqref{IP}.
Here \ssf$I\!\subseteq\ssb\R$ \ssf is any time interval,
\begin{equation*}\textstyle
\sigma\ge\frac{n+1}2\ssf\bigl(\frac12\ssb-\ssb\frac1q\bigr)\,,\quad 
\tilde{\sigma}\ge\frac{n+1}2\ssf\bigl(\frac12\ssb-\ssb\frac1{\tilde{q}}\bigr)\,,
\end{equation*}
and the couples
$\bigl(\frac1p,\frac1q\bigr)$, $\bigl(\frac1{\tilde{p}},\frac1{\tilde{q}}\bigr)$
belong to the triangle
\begin{equation}\label{admissibility}\textstyle
\bigl\{\ssf\bigl(\frac1p,\frac1q\bigr)\!\in\!
\bigl(0,\frac12\bigr]\!\times\!\bigl(0,\frac12\bigr)\,\big|\;
\frac1p\!\ge\!\frac{n-1}2\ssf\bigl(\frac12\!-\!\frac1q\bigr)\ssf\bigr\}
\,\cup\,\bigl\{\bigl(0,\frac12\bigr)\bigr\}\,.
\end{equation}
These estimates are similar to those obtained in \cite{APV2}
for the shifted wave equation,
except that they involve standard Sobolev spaces and no exotic ones.
Notice that the range \eqref{admissibility} of admissible couples for \ssf$\Hn$
is substantially wider than the range \eqref{EuclideanAdmissibility} for \ssf$\Rn$,
which corresponds to the lower edge of the triangle \eqref{admissibility}.

In Section 6 we apply the results of the previous sections
to the problem of global existence with small data
for the corresponding semilinear equations.
In contrast with the Euclidean case,
where the range of admissible nonlinearities \ssf$F(u)\!\sim\!|u|^{\gamma}$
\ssf is restricted to \ssf$\gamma\!>\!\gammazero$\ssf,
we prove global well--posedness
for powers \ssf$\gamma\!>\!1$
\ssf arbitrarily close to \ssf$1$\ssf.
This result improves in particular \cite{MT},
where global well--posedness for \eqref{nonshiftedWave}
was obtained in the case \ssf$n\!=\!3$ and \ssf$\kappa\!=\!\rho$
under the assumption \ssf$\gamma\!>\!5/3$\ssf.

In conclusion, the fact that
better results hold for \ssf$\Hn$ than for \ssf$\Rn$
may be regarded as a consequence of
the stronger dispersion properties of waves in negative curvature.
Besides our results extend to the more general setting of Damek--Ricci spaces,
as done in \cite{APV1} for the Schr\"odinger equation
and in \cite{APV2} for the shifted wave equation.

\section{Spherical analysis on real hyperbolic spaces}
\label{Notation}

In this paper, we consider the simplest class of
Riemannian symmetric spaces of the noncompact type,
namely real hyperbolic spaces $\Hn$ of dimension \ssf$n\!\ge\!2$\ssf.
We refer to Helgason's books \cite{He1, He2, He3}
and to Koornwinder's survey \cite{Ko}
for their algebraic structure and geometric properties,
as well as for harmonic analysis on these spaces,
and we shall be content with the following information.
$\Hn$~can be realized as the symmetric space $G/K$,
where $G\ssb=\ssb\text{SO}(1,n)_0$ and $K\!=\ssb\text{SO}(n)$\ssf.
In geodesic polar coordinates,
the Laplace--Beltrami operator on $\Hn$ writes
\begin{equation*}\textstyle
\Delta_{\ssf\Hn}\ssb
=\partial_r^{\ssf2}
+(n\!-\!1)\,\coth r\,\partial_r
+\ssf\sinh^{-2}\ssb r\,\Delta_{\,\Sn}\,.
\end{equation*}
The spherical functions \ssf$\varphi_\lambda$ on \ssf$\Hn$ are
normalized radial eigenfunctions of \ssf$\Delta\!=\!\Delta_{\ssf\Hn}$\,:
\begin{equation*}\begin{cases}
\;\Delta\,\varphi_\lambda=-(\lambda^2\!+\!\rho^{\ssf2})\,\varphi_\lambda\,,\\
\;\varphi_\lambda(0)=1\,,
\end{cases}\end{equation*}
where \ssf$\lambda\!\in\!\mathbb C$ \ssf
and \ssf$\rho\ssb=\ssb\frac{n-1}2$\ssf.
They can be expressed in terms of special functions\,:
\begin{equation*}\textstyle
\varphi_\lambda(r)
=\phi_{\,\lambda}^{(\frac n2-1,-\frac12)}(r)
={}_2F_1\bigl(\frac\rho2\ssb+\ssb i\ssf\frac\lambda2,
\frac\rho2\ssb-\ssb i\ssf\frac\lambda2;
\frac n2;-\sinh^2\ssb r\bigr)\ssf,
\end{equation*}
where $\phi_\lambda^{(\alpha,\beta)}$ denotes the Jacobi functions
and ${}_2F_1$ the Gauss hypergeometric function.
In the sequel we shall use the Harish--Chandra formula
\begin{equation}\label{intrepr}
\varphi_\lambda(r)\,=\int_Kdk\;e^{-(\rho+i\lambda)\,\text{H}(a_{-r}k)}
\end{equation}
and the basic estimate
\begin{equation}\label{phi0}
|\ssf\varphi_\lambda(r)|
\le\varphi_0(r)
\lesssim(1\!+\ssb r)\,e^{-\rho\ssf r}
\qquad\forall\;\lambda\!\in\!\R\,,\;r\!\ge\!0\,.
\end{equation}
We shall also use the Harish--Chandra expansion
\begin{equation}\label{HCexpansion}
\varphi_\lambda(r)
=\bc(\lambda)\,\Phi_\lambda(r)+\bc(-\lambda)\,\Phi_{-\lambda}(r)
\qquad\forall\;\lambda\!\in\!\C\!\smallsetminus\!\Z\ssf,\;r\!>\!0\ssf,
\end{equation}
where the Harish--Chandra $\bc$--function is given by
\begin{equation}\label{cfunction}\textstyle
\bc(\lambda)=\ssf\frac{\Gamma(2\rho)}{\Gamma(\rho)}\ssf
\frac{\Gamma(i\lambda)}{\Gamma(i\lambda+\rho)}
\end{equation}
and
\begin{equation}\label{Philambda}\begin{aligned}
\Phi_{\lambda}(r)&\textstyle
\,=(\ssf2\sinh r)^{\ssf i\lambda-\rho}\,{}_2F_1\bigl(
\frac\rho2\!-\!i\ssf\frac\lambda2,-\frac{\rho-1}2\!-\!i\ssf\frac\lambda2\ssf;
1\!-\!i\ssf\lambda\ssf;-\sinh^{-2}\ssb r\bigr)\\
&=\,(\ssf2\sinh r)^{-\rho}\,e^{\ssf i\ssf\lambda\ssf r}\,
\sum\nolimits_{k=0}^{+\infty}\,\Gamma_k(\lambda)\,e^{-2\ssf k\ssf r}\\
&\sim\,e^{\ssf(i\lambda-\rho)\ssf r}
\qquad\text{as \,}r\!\to\!+\infty\,.
\end{aligned}\end{equation}
The coefficients \ssf$\Gamma_k(\lambda)$
\ssf in the expansion \eqref{Philambda}
are rational functions of \ssf$\lambda\!\in\!\C$\ssf,
which satisfy the recurrence formula
\begin{equation*}\begin{cases}
\,\Gamma_0(\lambda)=1\ssf,\\
\,\Gamma_k(\lambda)
=\frac{\rho\,(\rho-1)}{k\,(k-i\lambda)}\,
\sum_{j=0}^{k-1}\,(k\,-\,j)\,\Gamma_j(\lambda)\,.
\end{cases}\end{equation*}
Their classical estimates were improved as follows in \cite[Lemma 2.1]{APV1}.

\begin{lemma}
Let \,$0\!<\!\varepsilon\!<\!1$
and \,$\Omega_\varepsilon\ssb
=\{\,\lambda\!\in\!\C\mid
\Re\lambda\ssb\le\ssb\varepsilon\ssf|\lambda|\ssf,
\,\Im\lambda\ssb\le\!-1\!+\!\varepsilon\,\}$\ssf.
Then there exist \,$\nu\!\ge\!0$ and,
for every \,$\ell\!\in\!\N$\ssf,
\ssf$C_\ell\!\ge\!0$ such that
\begin{equation}\label{derGammakappa}
\bigl|\,\partial_\lambda^{\,\ell}\ssf\Gamma_k(\lambda)\,\bigr|
\le C_\ell\,k^{\ssf {\nu}}\,(\ssf1\!+\ssb|\lambda|\ssf)^{-\ell-1}
\qquad\forall\;k\!\in\!\N^*,\,
\lambda\!\in\!\C\ssb\smallsetminus\ssb\Omega_\varepsilon\,.
\end{equation}
\end{lemma}

\noindent
Under suitable assumptions,
the spherical Fourier transform
of a bi--$K$\!--invariant function $f$ on $G$
is defined by 
\begin{equation*}
\mathcal{H}f(\lambda)=\int_Gdg\,f(g)\,\varphi_{\lambda}(g)
\end{equation*}
and the following inversion formula holds\,:
\begin{equation*}
f(x)=\const\int_{\,0}^{+\infty}\hspace{-1mm}
d\lambda\,|\ssf\bc(\lambda)|^{-2}\,
\mathcal{H}f(\lambda)\,\varphi_{\lambda}(x)\,.
\end{equation*}
Here is a well--known estimate of the Plancherel density\,:
\begin{equation}\label{estimatec}
|\ssf\bc(\lambda)\ssf|^{-2}
\lesssim\,|\lambda|^2\,(1\!+\!|\lambda|)^{n-3}
\qquad\forall\;\lambda\!\in\!\R\,.
\end{equation}
Via the spherical Fourier transform,
the Laplace--Beltrami operator \ssf$\Delta$ \ssf corresponds to
\begin{equation*}%\label{FTDelta}
-\ssf\lambda^2\ssb-\rho^{\ssf2},
\end{equation*}
hence the operators \,$D\ssb
=\!\sqrt{-\Delta\!-\!\rho^{\ssf2}\hspace{-.75mm}+\!\kappa^2\ssf}$
\,and \,$\widetilde{D}\ssb
=\!\sqrt{-\Delta\!-\!\rho^{\ssf2}\hspace{-.75mm}+\!\tilde{\kappa}^2\ssf}$
\,to
\begin{equation*}%\label{FTD}
\sqrt{\ssf\lambda^2\ssb+\kappa^2\ssf}
\quad\text{and}\quad
\sqrt{\ssf\lambda^2\ssb+\tilde{\kappa}^2\ssf}\,.
\end{equation*}

\section{Kernel estimates}
\label{Kernel}

In this section we derive pointwise estimates
for the radial convolution kernel \ssf$w_{\ssf t}^{\ssf\sigma}$
of the operator \ssf$W_t^{\ssf\sigma}\!
=\ssb\widetilde{D}^{\ssf-\sigma}\,e^{\,i\,t\ssf D}$,
for suitable exponents  \ssf$\sigma\!\in\!\R$\ssf.
By the inversion formula of the spherical Fourier transform,
\begin{equation*}
w_{\ssf t}^{\ssf\sigma}(r)\ssf
=\const\int_{\,-\infty}^{+\infty}\hspace{-1mm}d\lambda\;
|\ssf\bc(\lambda)|^{-2}\,
\,(\lambda^2\!+\ssb{\tilde\kappa}^{\ssf2})^{-\frac{\sigma}2}\,
\varphi_\lambda(r)\,
e^{\,i\,t\ssf \sqrt{\lambda^2+\ssf\kappa^2\ssf}}\ssf.
\end{equation*}
Contrarily to the Euclidean case,
this kernel has different behaviors,
depending whether \ssf$t$ \ssf is small or large,
and therefore we cannot use any rescaling.
Let us split up
\begin{equation*}\begin{aligned}
w_{\ssf t}^{\ssf\sigma}(r)\ssf
&=\,w_{\,t}^{\sigma,0}(r)+\ssf w_{\,t}^{\sigma,\infty}(r)\\
&=\,\const\int_{-\infty}^{+\infty}\!d\lambda\,
\chi_0(\lambda)\,|\ssf\bc(\lambda)|^{-2}\,
(\lambda^2\!+\ssb{\tilde\kappa}^{\ssf2})^{-\frac{\sigma}2}\,
\varphi_\lambda(r)\;e^{\,i\,t\ssf\sqrt{\lambda^2+\ssf\kappa^2\ssf}}\\
&+\,\const\int_{-\infty}^{+\infty}\hspace{-1mm}d\lambda\,
\chi_\infty(\lambda)\,|\ssf\bc(\lambda)|^{-2}\,
(\lambda^2\!+\ssb{\tilde\kappa}^{\ssf2})^{-\frac{\sigma}2}\,
\varphi_\lambda(r)\;e^{\,i\,t\ssf\sqrt{\lambda^2+\ssf\kappa^2}}
\end{aligned}\end{equation*}
using smooth even cut--off functions
\ssf$\chi_0$ and \ssf$\chi_\infty$ on \ssf$\R$
\ssf such that
\begin{equation*}
\chi_0(\lambda)+\chi_\infty(\lambda)=1
\quad\text{and}\quad\begin{cases}
\,\chi_0(\lambda)\!=\!1
&\forall\;|\lambda|\!\le\!\kappa\ssf,\\
\,\chi_\infty(\lambda)\!=\!1
&\forall\;|\lambda|\!\ge\!\kappa\!+\!1\ssf.\\
\end{cases}\end{equation*}
We shall first estimate \ssf$w_{\,t}^{\sigma,0}$
and next a variant of \ssf$w_{\,t}^{\sigma,\infty}$.
The kernel \ssf$w_{\,t}^{\sigma,\infty}$ has indeed
a logarithmic singularity on the sphere \ssf$r\!=\ssb t$
\ssf when \ssf$\sigma\ssb=\ssb\frac{n+1}2$.
We bypass this problem
by considering the analytic family of operators
\begin{equation*}\textstyle
\widetilde{W}_{\,t}^{\ssf\sigma,\infty}
=\ssf\frac{\vphantom{|}e^{\ssf\sigma^2}}{\Gamma(\frac{n+1}2-\sigma)}\;
\chi_\infty(D)\,\widetilde{D}^{-\sigma}\,e^{\,i\,t\ssf D}
\end{equation*}
in the vertical strip \ssf$0\!\le\!\Re\sigma\!\le\!\frac{n+1}2$
\ssf and the corresponding kernels
\begin{equation}\label{ISFT}\textstyle
\widetilde{w}_{\,t}^{\sigma,\infty}(r)
=\ssf\const\frac{\vphantom{|}e^{\ssf\sigma^2}}{\Gamma(\frac{n+1}2-\sigma)}\,
{\displaystyle\int_{-\infty}^{+\infty}}\hspace{-1mm}
d\lambda\,\chi_\infty(\lambda)\,|\ssf\bc(\lambda)|^{-2}\,
(\lambda^2\!+\ssb\tilde{\kappa}^2)^{-\frac{\sigma}2}\,
\varphi_\lambda(r)\,e^{\,i\,t\ssf\sqrt{\lambda^2+\ssf\kappa^2}}\,.
\end{equation}
Notice that the Gamma function
(which occurs naturally in the theory of Riesz distributions)
will allow us to deal with the boundary point \ssf$\sigma\!=\!\frac{n+1}2$,
while the exponential function
yields boundedness at infinity in the vertical strip.

\subsection{Estimate of
\,$w_{\ssf t}^{\ssf0}\ssb=\ssb w_{\,t}^{\sigma,0}$\ssf.}
\label{KernelEstimatewt0}

\begin{theorem}\label{Estimatewt0}
Let \,$\sigma\!\in\!\R$\ssf.
The following pointwise estimates hold for the kernel \,$w_{\ssf t}^{\ssf0}\ssb:$
\begin{itemize}
\item[(i)]
For every \,$t\!\in\!\R$ and \,$r\!\ge\!0$\ssf, we have
\begin{equation*}
|\ssf w_{\ssf t}^{\ssf0}(r)|\ssf\lesssim\ssf\varphi_0(r)\ssf.
\end{equation*}
\item[(ii)]
Assume that \,$|t|\ssb\ge\ssb2$\ssf.
Then, for every \,$0\ssb\le\ssb r\ssb\le\ssb\frac{|t|}2$\ssf, we have
\begin{equation*}
|\ssf w_{\ssf t}^{\ssf0}(r)|\ssf
\lesssim\ssf|t|^{-\frac32}\,(1\!+\!r)\,\varphi_0(r)\ssf.
\end{equation*}
\end{itemize}\end{theorem}

\begin{proof}
Recall that
\begin{equation}\label{wt0}
w_{\ssf t}^{\ssf0}(r)=\ssf\const\int_{-\kappa-1}^{\ssf\kappa+1}\!d\lambda\,
\chi_0(\lambda)\,|\ssf\bc(\lambda)|^{-2}\,
(\lambda^2\!+\ssb\tilde{\kappa}^2)^{-\frac{\sigma}2}\,
\varphi_\lambda(r)\;e^{\,i\,t\ssf\sqrt{\lambda^2+\ssf\kappa^2}}\,.
\end{equation}
By symmetry we may assume that \ssf$t\!>\!0$\ssf.

\noindent
(i) It follows from the estimates \eqref{phi0} and \eqref{estimatec} that
\begin{equation*}
|\ssf w_{\ssf t}^{\ssf0}(r)|\,
\lesssim\int_{-\kappa-1}^{\ssf\kappa+1}\hspace{-1mm}
d\lambda\,\lambda^2\,\varphi_0(r)\,
\lesssim\,\varphi_0(r)\,.
\end{equation*}
We prove (ii) by substituting in \eqref{wt0}
the first integral representation of $\varphi_\lambda$ in \eqref{intrepr}
and by reducing in this way to Fourier analysis on \ssf$\R$\ssf.
Specifically,
\begin{equation*}
w_{\ssf t}^{\ssf0}(r)\ssf=\int_{\ssf K}dk\,e^{-\rho\,\text{H}(a_{-r}k)}
\int_{-\infty}^{+\infty}\!d\lambda\;a(\lambda)\,
e^{\,i\,t\ssf\bigl(\sqrt{\lambda^2+\ssf\kappa^2\ssf}\ssf
-\ssf\frac{\text{H}(a_{-r}k)\ssf\lambda}t\ssf\bigr)}\,,
\end{equation*}
where \ssf$a(\lambda)\ssb
=\const\chi_0(\lambda)\,|\ssf\bc(\lambda)|^{-2}\,
(\lambda^2\!+\ssb\tilde{\kappa}^2)^{-\frac{\sigma}2}$.
Since 
\begin{equation*}
\int_{\ssf K}dk\;e^{-\rho\,\text{H}(a_{-r}k)}\ssf=\,\varphi_0(r)
\end{equation*}
and \ssf$|\ssf\text{H}(a_{-r}k)|\ssb\le\ssb r$\ssf,
it remains for us to estimate the oscillatory integral
\begin{equation*}
I(t,x)\ssf=\int_{-\infty}^{+\infty}\!d\lambda\;a(\lambda)\,
e^{\,i\,t\ssf\bigl(\sqrt{\lambda^2+\ssf\kappa^2\ssf}\ssf
-\ssf\frac{x\ssf\lambda}t\ssf\bigr)}
\end{equation*}
by \ssf$|t|^{-\frac32}\ssf(1\!+\!|x|)$\ssf.
This is obtained by the method of stationary phase.
More precisely, we apply Lemma A.1 in Appendix A,
whose assumption \eqref{assumptionA1} is fulfilled,
according to \eqref{estimatec}.
\end{proof} 

\subsection{Estimate of
\,$\widetilde{w}_{\,t}^{\ssf\infty}\!
=\ssb\widetilde{w}_{\,t}^{\sigma,\infty}$.}
\label{KernelEstimatewtildetinfty}

\begin{theorem}\label{Estimatewtildetinfty}
The following pointwise estimates hold for the kernel
\,$\widetilde{w}_{\,t}^{\ssf\infty}$,
uniformy in \ssf$\sigma\!\in\!\C$
with \ssf$\Re\sigma\ssb=\ssb\frac{n+1}2:$
\begin{itemize}
\item[(i)]
Assume that \,$|t|\!\ge\!2$\ssf.
Then, for every \,$r\!\ge\!0$\ssf, we have
\begin{equation*}\label{wtildeinftytlarge}
|\ssf\widetilde{w}_{\,t}^{\ssf\infty}(r)|\ssf
\lesssim\ssf|t|^{-\infty}\,.
\end{equation*}
\item[(ii)]
Assume that \,$0\!<\!|t|\!\le\!2$\ssf.
\begin{itemize}
\item[(a)]
\,If \,$r\!\ge\!3$\ssf, then
\,$\widetilde{w}_{\,t}^{\ssf\infty}(r)\ssb
=\text{\rm O}\ssf(r^{-1}\ssf e^{-\rho\ssf r}\ssf)$\ssf.
$\vphantom{\Big|}$
\item[(b)]
If \,$0\!\le\!r\!\le\!3$\ssf, then
\,$|\ssf\widetilde{w}_{\,t}^{\ssf\infty}(r)|
\lesssim\ssf\begin{cases}
\,|t|^{-\frac{n-1}2}
&\text{if \;}n\ssb\ge\ssb3\ssf,\\
\,|t|^{-\frac12}\ssf(\ssf1\!-\ssb\log|t|\ssf)
&\text{if \;}n\ssb=\ssb2\ssf.\\
\end{cases}$
\end{itemize}
\end{itemize} 
\end{theorem}

By symmetry we may assume again \ssf$t\!>\!0$
\ssf throughout the proof of Theorem \ref{Estimatewtildetinfty}.
\medskip

\noindent
\textit{Proof of Theorem \ref{Estimatewtildetinfty}.i\/}.
By evenness we have
\begin{equation}\label{wtildetinfty}\textstyle
\widetilde{w}_{\,t}^{\ssf\infty}(r)
=2\ssf\const\frac{\vphantom{|}e^{\ssf\sigma^2}}{\Gamma(\frac{n+1}2-\sigma)}\,
{\displaystyle\int_{\,0}^{+\infty}}\!d\lambda\,
\chi_\infty(\lambda)\,|\ssf\bc(\lambda)|^{-2}\,
(\lambda^2\!+\ssb\tilde{\kappa}^2)^{-\frac{\sigma}2}\,
\varphi_\lambda(r)\,
e^{\,i\,t\ssf\sqrt{\lambda^2+\ssf\kappa^2}}\,.
\end{equation}
If \ssf$0\!\le\!r\!\le\!\frac t2$\ssf,
we resume the proof of Theorem \ref{Estimatewt0}.ii,
using Lemma A.2 instead of Lemma A.1,
and conclude that
\begin{equation}\label{estimate1wtildetinfty}
|\ssf\widetilde{w}_{\,t}^{\ssf\infty}(r)|\ssf
\lesssim\,t^{-\infty}\,\varphi_0(r)\ssf.
\end{equation}
If \ssf$r\!\ge\!\frac t2$\ssf,
we  substitute in \eqref{wtildetinfty}
the Harish--Chandra expansion (\ref{HCexpansion}) of $\varphi_\lambda(r)$
and reduce this way again to Fourier analysis on \ssf$\R$\ssf.
Specifically, our task consists in estimating the expansion
\begin{equation}\label{expansionwtildetinfty}\textstyle
\widetilde{w}_{\,t}^{\ssf\infty}(r)=
(\sinh r)^{-\rho}\;
{\displaystyle\sum\nolimits_{\ssf k=0}^{+\infty}}\,
e^{-2\ssf k\ssf r}\ssf
\bigl\{\ssf I_{\,k}^{+,\infty}(t,r)+I_{\,k}^{-,\infty}(t,r)\ssf\bigr\}
\end{equation}
involving oscillatory integrals
\begin{equation*}
I_{\,k}^{\pm,\infty}(t,r)\ssf
=\int_{\,0}^{+\infty}\hspace{-1mm}d\lambda\;
a_{\ssf k}^\pm(\lambda)\;
e^{\,i\ssf(\ssf t\ssf\sqrt{\lambda^2+\ssf\kappa^2\ssf}\ssf
\pm\,r\ssf\lambda\ssf)}
\end{equation*}
with amplitudes
\begin{equation*}\textstyle
a_{\ssf k}^\pm(\lambda)=\ssf2\ssf\const\ssf
\frac{\vphantom{|}e^{\sigma^2}}{\Gamma(\frac{n+1}2-\sigma)}\;
\chi_\infty(\lambda)\;\bc(\mp\lambda)^{-1}\,
(\lambda^2\!+\ssb\tilde{\kappa}^2)^{-\frac{\sigma}2}\,
\Gamma_k(\pm\lambda)\,.
\end{equation*}
Notice that \ssf$a_{\ssf k}^\pm(\lambda)$ \ssf is a symbol of order
\begin{equation*}
d\ssf=\begin{cases}
-1
&\text{if \,}k\!=\!0\ssf,\\
-2
&\text{if \,}k\!\in\!\N^*,
\end{cases}\end{equation*}
uniformly in \ssf$\sigma\!\in\!\C$ \ssf
with \ssf$\Re\sigma\ssb=\ssb\frac{n+1}2$\ssf.
This follows indeed
from the expression \eqref{cfunction} of the \ssf$\bc$--function
and from the estimate \eqref{derGammakappa}
of the coefficients \ssf$\Gamma_k$\ssf.
Consequently the integrals
\begin{equation}\label{estimateIkpm}
I_{\,k}^{\pm,\infty}(t,r)=\text{O}\ssf(k^{\ssf\nu})
\end{equation}
are easy to estimate when \ssf$k\!>\!0$\ssf,
\ssf while \ssf$I_{\,0}^{+,\infty}(t,r)$
and especially \ssf$I_{\,0}^{-,\infty}(t,r)$
require more work.
As far as it is concerned,
we integrate by parts
\begin{equation*}
I_{\,0}^{+,\infty}(t,r)\,
=\int_{\,0}^{+\infty}\hspace{-1mm}d\lambda\,
a_{\ssf 0}^+(\lambda)\,e^{\ssf i\ssf t\ssf\phi(\lambda)}\ssf,
\end{equation*}
using \ssf$e^{\ssf i\ssf t\ssf\phi(\lambda)}\!
=\ssb\frac1{i\ssf t\ssf\phi'\ssb(\lambda)}\ssf
\frac\partial{\partial\lambda}\ssf
e^{\ssf i\ssf t\ssf\phi(\lambda)}$
and the following properties of
\ssf$\phi(\lambda)\ssb
=\ssb\sqrt{\lambda^2\!+\ssb\kappa^2\ssf}\ssb
+\ssb\frac rt\ssf\lambda$\;:
\begin{itemize}
\item[$\bullet$]
$\,\phi'(\lambda)\ssb
=\ssb\frac{\lambda}{\sqrt{\lambda^2+\kappa^2}}\ssb+\ssb\frac rt\ssb
\ge\ssb\frac rt\ssb\ge\ssb\frac12$\,,
\item[$\bullet$]
$\,\phi''(\lambda)\ssb=\ssb\kappa^2\ssf(\lambda^2\!+\ssb\kappa^2)^{-\frac32}$
\,is a symbol of order $-\ssf3$\ssf.
\end{itemize}
We obtain this way
\begin{equation}\label{estimateI0p}
I_{\,0}^{+,\infty}(t,r)=\text{O}\ssf(\ssf r^{-1})
\end{equation}
and actually
\begin{equation*}
I_{\,0}^{+,\infty}(t,r)=\text{O}\ssf(\ssf r^{-\infty})
\end{equation*}
by repeated integrations by parts.
Let us turn to the last integral, that we rewrite as follows\,:
\begin{equation*}
I_{\,0}^{-,\infty}(t,r)\ssf
=\int_{\,0}^{+\infty}\hspace{-1mm}d\lambda\;a_{\ssf 0}^-(\lambda)\,
e^{\,i\ssf t\ssf(\sqrt{\lambda^2+\ssf\kappa^2\ssf}-\ssf\lambda)}\,
e^{\ssf i\ssf(t-r)\ssf\lambda}\,.
\end{equation*}
After performing an integration by parts based on
\ssf$e^{\ssf i\ssf(t-r)\ssf\lambda}\ssb
=\ssb\frac1{i\ssf(t-r)}\ssf\frac\partial{\partial\lambda}\ssf
e^{\ssf i\ssf(t-r)\ssf\lambda}$
\ssf and by using the fact that
\begin{equation}\label{psi}\textstyle
\psi(\lambda)=\sqrt{\lambda^2\!+\ssb\kappa^2\ssf}\ssb-\ssb\lambda
=\frac{\kappa^2}{\sqrt{\lambda^2+\ssf\kappa^2\ssf}\ssf+\,\lambda}
\end{equation}
is a symbol of order $-1$\ssf,
we obtain
\begin{equation}\label{estimate1integralI0m}\textstyle
I_{\,0}^{-,\infty}(t,r)
=\text{O}\ssf\bigl(\frac t{|\ssf r\ssf-\ssf t\ssf|}\bigr)\ssf.
\end{equation}
This estimate is enough for our purpose,
as long as \ssf$r$ stays away from $t$\ssf.
If \ssf$|\ssf r\!-\!t\ssf|\!\le\!1$\ssf,
let us split up
\begin{equation*}
e^{\ssf i\ssf t\ssf\psi(\lambda)}=1+\text{O}\ssf(\ssf t\,\psi(\lambda))
\end{equation*}
and
\begin{equation}\label{estimate2integralI0m}
I_{\,0}^{-,\infty}(t,r)\,
=\int_{\,0}^{+\infty}\hspace{-1mm}d\lambda\;
a_{\ssf 0}^-(\lambda)\;e^{\ssf i\ssf(t-r)\ssf\lambda}\,
+\;\text{O}\ssf(t)
\end{equation}
accordingly.
The remaining integral was estimated in \cite{APV1},
more precisely at the end of the proof Theorem 4.2.ii\,:
\begin{equation}\label{estimate3integralI0m}\textstyle
{\displaystyle\int_{\,0}^{+\infty}}\hspace{-1mm}d\lambda\;
a_{\ssf 0}^-(\lambda)\;e^{\ssf i\ssf(t-r)\ssf\lambda}\,
=\,\text{O}\ssf(1)\ssf.
\end{equation}
By combining the estimates
\eqref{estimateIkpm},
\eqref{estimateI0p},
\eqref{estimate1integralI0m},
\eqref{estimate2integralI0m},
\eqref{estimate3integralI0m},
we deduce from \eqref{expansionwtildetinfty} that
\begin{equation*}\textstyle
|\ssf\widetilde{w}_{\,t}^{\ssf\infty}(r)|\,
\lesssim\,e^{-\rho\ssf r}\,t\,\lesssim\,t^{-\infty}
\qquad\forall\;r\ssb\ge\ssb\frac t2\ssb\ge\ssb1\ssf,
\end{equation*}
uniformly in \ssf$\sigma\!\in\!\C$ \ssf with \ssf$\Re\sigma\!=\!\frac{n+1}2$\ssf.
This concludes the proof of Theorem \ref{Estimatewtildetinfty}.i.
\hfill$\square$
\medskip

Let us turn to the small time estimates in Theorem \ref{Estimatewtildetinfty}.
\medskip

\noindent
\textit{Proof of Theorem \ref{Estimatewtildetinfty}.ii.a\/}.
Since \ssf$0\!<\!t\!\le2$ \ssf and \ssf$r\!\ge\!3$\ssf,
we can resume the proof of Theorem \ref{Estimatewtildetinfty}.i
in the case \ssf$r\!\ge\ssb t\ssb+\!1\!\ge\!\frac t2$\ssf.
By using the expansion \eqref{expansionwtildetinfty}
and the estimates
\eqref{estimateIkpm},
\eqref{estimateI0p},
\eqref{estimate1integralI0m},
we obtain
\begin{equation*}
|\ssf\widetilde{w}_{\,t}^{\ssf\infty}(r)|\,\lesssim\,r^{-1}\,e^{-\rho\ssf r}\,,
\end{equation*}
uniformly in \ssf$\sigma\!\in\!\C$ \ssf with \ssf$\Re\sigma\!=\!\frac{n+1}2$\ssf.
This concludes the proof of Theorem \ref{Estimatewtildetinfty}.ii.a.
\hfill$\square$
\medskip

\noindent
\textit{Proof of Theorem \ref{Estimatewtildetinfty}.ii.b\/}.
Let us rewrite and expand \eqref{wtildetinfty} as follows\,:
\begin{align}
\widetilde{w}_{\,t}^{\ssf\infty}(r)
&\textstyle=\,2\ssf\const\ssf
\frac{\vphantom{|}e^{\ssf\sigma^2}}{\Gamma(\frac{n+1}2-\sigma)}\,
{\displaystyle\int_{\,0}^{+\infty}}\!d\lambda\,
\chi_\infty(\lambda)\,|\ssf\bc(\lambda)|^{-2}\,
(\lambda^2\!+\ssb\tilde{\kappa}^2)^{-\frac{\sigma}2}\,
e^{\,i\,t\,\psi(\lambda)}\,e^{\,i\,t\ssf\lambda}\,
\varphi_\lambda(r)\\
&\textstyle\label{LocalKernelEstimate}
={\displaystyle\int_{\,0}^{+\infty}}\!d\lambda\;
a(\lambda)\,e^{\,i\,t\,\lambda}\,\varphi_\lambda(r)\,
+{\displaystyle\int_{\,0}^{+\infty}}\!d\lambda\;
b(\lambda)\,e^{\,i\,t\,\lambda}\,\varphi_\lambda(r)\,,
\end{align}
where \ssf$\psi$ \ssf is given by \eqref{psi},
\begin{equation*}\textstyle
a(\lambda)=\ssf2\ssf\const\ssf
\frac{\vphantom{|}e^{\ssf\sigma^2}}{\Gamma(\frac{n+1}2-\sigma)}\;
\chi_\infty(\lambda)\,|\ssf\bc(\lambda)|^{-2}\,
(\lambda^2\!+\ssb\tilde{\kappa}^2)^{-\frac{\sigma}2}
\end{equation*}
is a symbol of order \ssf$\frac{n-3}2$\ssf,
uniformly in \ssf$\sigma\!\in\!\C$ \ssf
with \ssf$\Re\sigma\!=\ssb\frac{n+1}2$\ssf,
and
\begin{equation*}\textstyle
b(\lambda)=\ssf2\ssf\const\ssf
\frac{\vphantom{|}e^{\ssf\sigma^2}}{\Gamma(\frac{n+1}2-\sigma)}\;
\chi_\infty(\lambda)\,|\ssf\bc(\lambda)|^{-2}\,
(\lambda^2\!+\ssb\tilde{\kappa}^2)^{-\frac{\sigma}2}\,
\bigl\{\ssf e^{\,i\,t\,\psi(\lambda)}\!-\!1\ssf\bigr\}
\end{equation*}
is a symbol of \ssf$\frac{n-5}2$\ssf,
uniformly in \ssf$0\!<\!t\!\le\!2$
\ssf and \ssf$\sigma\!\in\!\C$ \ssf
with \ssf$\Re\sigma\!=\ssb\frac{n+1}2$\ssf.
The first integral in \eqref{LocalKernelEstimate}
was analyzed in \cite[Appendix C]{APV1}
and estimated there by
\begin{equation*}
C\,\begin{cases}
\;t^{-\frac{n-1}2}
&\text{if \;}n\ssb\ge\ssb3\ssf,\\
\;t^{-\frac12}\ssf(\ssf1\!-\ssb\log|t|\ssf)
&\text{if \;}n\ssb=\ssb2\ssf.\\
\end{cases}\end{equation*}
The second integral is easier to estimate,
for instance by \ssf$C\,t^{-\frac{n-2}2}$\ssf.
This concludes the proof of Theorem \ref{Estimatewtildetinfty}.ii.b.
\hfill$\square$

\begin{remark}\label{Hadamard}
As far as local estimates of wave kernels are concerned,
we might have used the \textit{Hadamard parametrix\/} \cite[\S\;17.4]{Ho3}
instead of spherical analysis.
\end{remark}

\begin{remark}\label{GeneralizedKernelEstimates}
The kernel analysis carried out in this section still holds for the operators
\,$D^{-\sigma}\widetilde{D}^{-\tilde{\sigma}}\ssf e^{\,i\,t\ssf D}$,
provided we assume
\,$\Re\sigma\!+\!\Re\widetilde{\sigma}\!=\!\frac{n+1}2$
in Theorem \ref{Estimatewtildetinfty}.
\end{remark}

\section{Dispersive estimates}
\label{Dispersive}

In this section we obtain $L^{q'}\!\to\! L^q$ estimates for the operator
\ssf$W_t^{\ssf\sigma}\!=\ssb\widetilde{D}^{\ssf-\sigma}\,e^{\,i\,t\ssf D}$,
which will be crucial role for our Strichartz estimates in next section.
Let us split up its kernel
\ssf$w_{\ssf t}^{\ssf\sigma}\!
=\ssb w_{\ssf t}^{\sigma,0}\!+\ssb w_{\ssf t}^{\sigma,\infty}$
as before.
We will handle the contribution of \ssf$w_{\ssf t}^{\sigma,0}$,
using the pointwise estimates obtained in Subsection \ref{KernelEstimatewt0}
and the following criterion (see for instance \cite[Theorem 3.4]{APV1})
based on the Kunze-Stein phenomenon.

\begin{lemma}\label{KS}
There exists a constant \,$C\!>\!0$ such that,
for every radial measurable function \,$\kappa$ on \,$\Hn$,
for every \,$2\!\le\!q\!<\!\infty$ and \ssf$f\!\in\!L^{q'}\ssb(\Hn)$,
\begin{equation*}
\|\ssf f\ssb\ast\ssb\kappa\,\|_{L^q\vphantom{L^{q'}}}
\le\,C_q\;\|f\|_{L^{q'}}\,\Bigl\{\ssf\int_{\,0}^{+\infty}\hspace{-1mm}dr\,
(\sinh r)^{n-1}\,|\kappa(r)|^{\frac q2}\,\varphi_0(r)\ssf\Bigr\}^{\frac2q}\,.
\end{equation*}
\end{lemma}

For the second part \ssf$w_{\ssf t}^{\sigma,\infty}$,
we resume the Euclidean approach,
which consists in interpolating analytically between 
\ssf$L^2\!\to\!L^2$ and \ssf$L^1\!\to\!L^\infty$ estimates
for the family of operators 
\begin{equation}\label{AnalyticFamily}\textstyle
\widetilde{W}_t^{\ssf\sigma,\infty}
=\,\frac{\vphantom{|}e^{\ssf\sigma^2}}{\Gamma(\frac{n+1}2-\sigma)}\;
\chi_\infty(D)\,\widetilde{D}^{-\sigma}\,e^{\,i\,t\ssf D}
\end{equation}
in the vertical strip \ssf$0\ssb\le\ssb\Re\sigma\ssb\le\!\frac{n+1}2$\ssf.

\subsection{Small time dispersive estimates}
 
\begin{theorem}\label{dispersive0}
Assume that
\,$0\ssb<\ssb|t|\ssb\le\ssb2$\ssf,
\ssf$2\ssb<\ssb q\ssb<\ssb\infty$\ssf
and \,$\sigma\ssb\ge\ssb(n\ssb+\!1)\ssf(\frac12\!-\!\frac1q)$\ssf.
Then, 
\begin{equation*}
\bigl\|\ssf\widetilde{D}^{\ssf -\sigma}\ssf e^{\,i\,t\ssf D}
\ssf\bigr\|_{L^{q'}\ssb\to L^q}\lesssim\,\begin{cases}
\,|t|^{-(n-1)(\frac12-\frac1q)}
&\text{if \,}n\ssb\ge\ssb3\ssf,\\
\,|t|^{-(\frac12-\frac1q)}\ssf(\ssf1\!-\ssb\log|t|\ssf)^{1-\frac2q}
&\text{if \,}n\ssb=\ssb2\ssf.\\
\end{cases}
\end{equation*}
\end{theorem}

\begin{proof}
We divide the proof into two parts,
corresponding to the kernel decomposition
\ssf$w_t\!=\ssb w_{\ssf t}^{\ssf0}\!+\ssb w_{\,t}^\infty$.
By applying Lemma \ref{KS}
and using the pointwise estimates in Theorem \ref{Estimatewt0}.i,
we obtain on one hand
\begin{equation*}
\begin{aligned}
\bigl\|\ssf f\ssb*\ssb w_{\ssf t}^{\ssf0}\ssf\bigr\|_{L^q}
&\lesssim\,\Bigl\{\ssf\int_{\,0}^{+\infty}\hspace{-1mm}dr\,
(\sinh r)^{n-1}\,\varphi_0(r)\,|\ssf w_{\ssf t}^{\ssf0}(r)|^{\frac q2}
\,\Bigr\}^{\frac2q}\;\|f\|_{L^{q'}}\\
&\lesssim\,\Big\{\ssf\int_{\,0}^{+\infty}\hspace{-1mm}dr\,
(1\!+\ssb r)^{\frac q2+1}\,e^{-\ssf(\frac q2-1)\ssf\rho\,r}
\,\Bigr\}^{\frac2q}\;\|f\|_{L^{q'}}\\
&\lesssim\;\|f\|_{L^{q'}}
\qquad\forall\;f\!\in\!L^{q'}.
\vphantom{\int_0^1}
\end{aligned}
\end{equation*}
On the other hand,
in order to estimate the \ssf$L^{q'}\hspace{-1mm}\rightarrow\!L^q$ norm
of \ssf$f\ssb\mapsto\ssb f\ssb*\ssb w_{\,t}^{\ssf\infty}$,
we proceed by interpolation for the analytic family \eqref{AnalyticFamily}.
If \ssf$\Re\sigma\ssb=\ssb0$\ssf, then
\begin{equation*}
\|\ssf f\ssb*\ssb\widetilde{w}_{\,t}^{\ssf\infty}\ssf\|_{L^2}
\lesssim\,\|f\|_{L^2}
\qquad\forall\;f\!\in\!L^2.
\end{equation*}
If \ssf$\Re\sigma\ssb=\ssb\frac{n+1}2$,
we deduce
from the pointwise estimates in Theorem \ref{Estimatewtildetinfty}.ii
that
\begin{equation*}
\|\ssf f\ssb*\ssb\widetilde{w}_{\,t}^{\ssf\infty}\ssf\|_{L^\infty}
\lesssim\,|t|^{-\frac{n-1}2}\,\|f\|_{L^1}
\qquad\forall\;f\!\in\!L^1.
\end{equation*}
By interpolation we conclude
for \ssf$\sigma\ssb=\ssb(n+1)\bigl(\frac12\!-\!\frac1q\bigr)$
\ssf that
\begin{equation*}
\bigl\|\ssf f\ssb*\ssb w_{\,t}^{\infty}\ssf\|_{L^q\vphantom{L^{q'}}}
\lesssim\,|t|^{-(n-1)(\frac12-\frac{1}{q})} \|f\|_{L^{q'}}
\qquad\forall\;f\!\in\!L^{q'}.
\end{equation*}
\end{proof}

\subsection{Large time dispersive estimate}
 
\begin{theorem}\label{dispersiveinfty}
Assume that
\,$|t|\ssb\ge\ssb2$\ssf,
\ssf$2\ssb<\ssb q\ssb<\ssb\infty$\ssf and \,$\sigma\ssb\ge\ssb(n\ssb+\!1)\ssf(\frac12\!-\!\frac1q)$\ssf.
Then
\begin{equation*}
\bigl\|\ssf\widetilde{D}^{\ssf-\sigma}\ssf e^{\,i\,t\ssf D}
\ssf\bigr\|_{L^{q'}\ssb\to L^q}\lesssim\,|t|^{\ssf-\frac{3}{2}}\,.
\end{equation*}
\end{theorem}

\begin{proof}
We divide the proof into three parts,
corresponding to the kernel decomposition
\begin{equation*}
w_t=\1_{\ssf B(0,\frac{|t|}{2})}\ssf w_{\ssf t}^{\ssf0}
+\1_{\,\Hn\smallsetminus\ssf B(0,\frac{|t|}{2})}\ssf w_{\ssf t}^{\ssf0}
+\ssf w_{\ssf t}^\infty\,.
\end{equation*}

\noindent
\textit{Estimate 1}\,:
By applying Lemma \ref{KS}
and using the pointwise estimate in Theorem \ref{Estimatewt0}.ii, 
we obtain
\begin{equation*}
\begin{aligned}
\|\ssf f*\{\1_{\ssf B(0,\frac{|t|}{2})}\ssf w_{\ssf t}^{0}\ssf\}\,\|_{L^q}
&\lesssim\,\Bigl\{\ssf\int_{\,0}^{\frac{|t|}{2}}\!dr\,
(\sinh r)^{n-1}\,\varphi_0(r)\,|\ssf w_{\ssf t}^{0}(r)|^{\frac q2}
\,\Bigr\}^{\frac2q} \;\|f\|_{L^{q'}}\\
&\lesssim\,\underbrace{\Bigl\{\ssf\int_{\,0}^{\frac{|t|}{2}}\hspace{-1mm}dr\,
(1\!+\ssb r)^{q+1}\,e^{-(\frac q2-1)\ssf\rho\ssf r}\,\Bigr\}^{\frac2q}}_{<+\infty}\,
|t|^{-\frac{3}{2}}\;\|f\|_{L^{q'}}
\qquad\forall\;f\!\in\!L^{q'}.
\end{aligned}
\end{equation*}

\noindent
\textit{Estimate 2}\,:
By applying Lemma \ref{KS}
and using the pointwise estimate in Theorem \ref{Estimatewt0}.i, 
we obtain
\begin{equation*}
\begin{aligned}
\|\ssf f*\{\1_{\ssf\Hn\smallsetminus\ssf  B(0,\frac{|t|}{2})}\ssf
w_{\ssf t}^0\ssf\}\,\|_{L^q}
&\lesssim\,\Bigl\{\ssf\int_{\frac{|t|}{2}}^{+\infty}\hspace{-1mm}dr\,
(\sinh r)^{n-1}\,\varphi_0(r)\,|\ssf w_{\ssf t}^{0}(r)|^{\frac q2}
\,\Bigr\}^{\frac2q}\;\|f\|_{L^{q'}}\\
&\lesssim\,\underbrace{
\Bigl\{\ssf\int_{\, \frac{|t|}{2}}^{+\infty}\hspace{-1mm}dr\,
r^{\frac q2+1}\,e^{-(\frac q2-1)\ssf\rho\ssf r}\,\Bigr\}^{\frac2q}
}_{\lesssim\;|t|^{-\infty}}\, 
\|f\|_{L^{q'}}
\qquad\forall\;f\!\in\!L^{q'}.
\end{aligned}
\end{equation*}

\noindent
\textit{Estimate 3}\,:
We proceed by interpolation
for the analytic family \eqref{AnalyticFamily}.
If \ssf$\Re\sigma\ssb=\ssb0\ssf$, then
\begin{equation*}
\|\ssf f*\widetilde{w}_{\,t}^{\ssf\infty}\ssf\|_{L^2}
\lesssim\,\|f\|_{L^2}
\qquad\forall\;f\!\in\!L^2.
\end{equation*}
If \ssf$\Re\sigma\ssb=\ssb\frac{n+1}2$,
we deduce from Theorem \ref{Estimatewtildetinfty}.i that
\begin{equation*}
\|\ssf f*\widetilde{w}_{\,t}^{\ssf\infty}\ssf\|_{L^\infty}
\lesssim\,|t|^{-\infty} \,\|f\|_{L^1}
\qquad\forall\;f\!\in\!L^1.
\end{equation*}
By interpolation we obtain for
\ssf$\sigma\ssb=\ssb(n\!+\!1)\bigl(\frac12\!-\!\frac1q\bigr)$
that
\begin{equation*}
\|\ssf f*w_{\,t}^\infty\ssf\|_{L^q\vphantom{L^{q'}}}
\lesssim\,|t|^{-\infty} \,\|f\|_{L^{q'}}
\qquad\forall\;f\!\in\!L^{q'}.
\end{equation*}
We conclude the proof of Theorem \ref{dispersiveinfty}
by summing up the previous estimates.
\end{proof}

\subsection{Global dispersive estimates}
${}$\medskip

As noticed in Remark \ref{GeneralizedKernelEstimates},
similar results hold for the operators
\,$D^{-\sigma}\widetilde{D}^{-\tilde{\sigma}}\ssf e^{\,i\ssf t\ssf D}$.

\begin{corollary}\label{DispersiveGlobal}
Let \,$2\!<\!q\!<\!\infty$
and \,$\sigma,\tilde{\sigma}\!\in\ssb\R$
such that \,$\sigma\!+\ssb\tilde{\sigma}\!
\ge\!(n\!+\!1)\bigl(\frac12\!-\!\frac1q\bigr)$.
Then
\begin{equation}\label{GeneralDispersiveEstimates}\textstyle
\|\,D^{-\sigma}\widetilde{D}^{-\tilde{\sigma}}\,
e^{\,i\ssf t\ssf D}\,\|_{L^{q'}\!\to L^q}
\lesssim\,\begin{cases}
\;|t|^{-(n-1)(\frac12-\frac1q)}
&\text{if \;}0\!<\!|t|\!\le\!1\ssf,\\
\;|t|^{-\frac32}
&\text{if \;}|t|\!\ge\!1\ssf.
\end{cases}\end{equation}
In particular,
if \,$2\!<\!q\!<\!\infty$
and \,$\sigma\!\ge\!(n\!+\!1)\bigl(\frac12\!-\!\frac1q\bigr)$,
then
\begin{equation}\label{ParticularDispersiveEstimates}\textstyle
\|\ssf\widetilde{D}^{-\sigma}\ssf
e^{\,i\,t\ssf D}\ssf\|_{L^{q'}\!\to L^q}
+\,\|\ssf\widetilde{D}^{\ssf1-\sigma}\,
\frac{e^{\,i\,t\ssf D}}D\ssf\|_{L^{q'}\!\to L^q}
\lesssim\,\begin{cases}
\;|t|^{-(n-1)(\frac12-\frac1q)}
&\text{if \;}0\!<\!|t|\!\le\!1\ssf,\\
\;|t|^{-\frac32}
&\text{if \;}|t|\!\ge\!1\ssf.
\end{cases}\end{equation}
These results hold in dimension \,$n\!\ge\!3$\ssf.
In dimension \,$n\!=\!2$\ssf,
there is an additional logarithmic factor in the small time bound,
which reads
\,$|t|^{-(\frac12-\frac1q)}\ssf(\ssf1\!-\ssb\log|t|\ssf)^{1-\frac2q}$.
\end{corollary}

\begin{remark}
On \ssf$L^2(\Hn)$, we know by spectral theory that
\begin{itemize}
\item
$\,e^{\,i\,t\ssf D}$ is a one--parameter group of unitary operators,
\item
$\,D^{-\sigma}\tilde{D}^{-\tilde{\sigma}}$ is a bounded operator
if \,$\sigma\ssb+\ssb\tilde{\sigma}\ssb\ge\ssb0$\ssf.
\end{itemize}
\end{remark}

\begin{remark}
Let us specialize our results
for the wave equation \eqref{nonshiftedWave}.
In this case, we have \,$D\ssb=\!\sqrt{-\Delta\ssf}$
and we may take \,$\tilde{D}\ssb=\ssb D$.
Let \,$2\!<\!q\!<\!\infty$ and
\,$\sigma\!\ge\!(n\!+\!1)\bigl(\frac12\!-\!\frac1q\bigr)$\ssf.
Then
\begin{equation}\label{DispersiveEstimateWaveND}
\bigl\|\,D^{-\sigma}\ssf e^{\,i\,t\ssf D}\,\bigr\|_{L^{q'}\!\to L^q}
\lesssim\,\begin{cases}
\;|t|^{-(n-1)(\frac12-\frac1q)}
&\text{if \;}0\!<\!|t|\!\le\!1\\
\;|t|^{-\frac32}
&\text{if \;}|t|\!\ge\!1
\end{cases}
\end{equation}
in dimension \,$n\!\ge\!3$ and
\begin{equation*}\label{DispersiveEstimateWave2D}
\bigl\|\,D^{-\sigma}\ssf e^{\,i\,t\ssf D}\,\bigr\|_{L^{q'}\!\to L^q}
\lesssim\,\begin{cases}
\;|t|^{-(\frac12-\frac1q)}\ssf(\ssf1\!-\ssb\log|t|\ssf)^{1-\frac2q}
&\text{if \;}0\!<\!|t|\!\le\!1\\
\;|t|^{-\frac32}
&\text{if \;}|t|\!\ge\!1
\end{cases}
\end{equation*}
in dimension \,$n\!=\!2$\ssf.
Let us compare \eqref{DispersiveEstimateWaveND}
with the dispersive estimates
obtained by Metcalfe and Taylor \cite[Section 3]{MT}
in dimension \,$n\!=\!3$\ssf.
Our results are the same when \,$|t|$ is small and \,$2\!<\!q\!<\!\infty$
or when \,$|t|$ is large and \,$4\!\le\!q\!<\!\infty$\ssf.
But our bound \,$|t|^{-\frac32}$
is better than their bound \,$|t|^{-\ssf6\ssf(\frac12-\frac1q)}$
when \,$|t|$ is large and \,$2\!<\!q\!<\!4$\ssf.
On the other hand,
they are able to deal with the endpoint case \,$q\!=\!\infty$\ssf,
using local Hardy and BMO spaces on \,$\Hn$.
\end{remark}

\newpage

\section{Strichartz estimates}
\label{Strichartz}

We shall assume \ssf$n\!\ge\!4$ \ssf throughout this section
and discuss the dimensions \ssf$n\!=\!3$ \ssf and \ssf$n\!=\!2$
\ssf in the final remarks.
Consider the linear equation \eqref{IP} on \ssf$\Hn$,
%\begin{equation}\label{IP}
%\begin{cases}
%\;\partial_{\,t}^{\,2\ssf}u(t,x)+D_x^{\ssf2\ssf}u(t,x)=F(t,x)\ssf,\\
%\;u(0,x)=f(x)\ssf,\;\partial_{\ssf t}|_{t=0}\,u(t,x)=g(x)\ssf,
%\end{cases}
%\end{equation}
whose solution is given by Duhamel's formula\,:
\begin{equation*}\textstyle
u(t,x)=\ssf\underbrace{\vphantom{\int_0^t}\textstyle
(\cos t\ssf D_x)\ssf f\ssf(x)
+\frac{\sin t\ssf D_x}{D_x}\ssf g\ssf(x)
}_{u_{\ssf\text{hom}}(t,x)}\ssf
+\underbrace{
\int_{\,0}^{\ssf t}\textstyle ds\,
\frac{\sin\ssf(t-s)\ssf D_x}{D_x}\ssf F(s,x)
}_{u_{\ssf\text{inhom}}(t,x)}\,.
\end{equation*}
In Appendix B, we recall the definition of Sobolev spaces on \ssf$\Hn$
and collect some of their properties.

\begin{definition}
A couple \,$(p,q)$ will be called \,{\rm admissible}
if \,$(\frac1p,\frac1q)$ belongs to the triangle
\begin{equation}\label{triangle}\textstyle
\bigl\{\ssf\bigl(\frac1p,\frac1q\bigr)\!\in\!
\bigl(0,\frac12\bigr]\!\times\!\bigl(0,\frac12\bigr)\bigm|
\frac1p\ssb\ge\ssb\frac{n-1}2\ssf\bigl(\frac12\!-\!\frac1q\bigr)\ssf\bigr\}
\ssf\cup\ssf\bigl\{\bigl(0,\frac12\bigr)\bigr\}\,.
\end{equation}
\end{definition}

\begin{figure}[ht]
\begin{center}
\psfrag{0}[c]{$0$}
\psfrag{1}[c]{$1$}
\psfrag{1/2}[c]{$\frac12$}
\psfrag{1/2-1/(n-1)}[c]{$\frac12\!-\!\frac1{n-1}$}
\psfrag{1/p}[c]{$\frac1p$}
\psfrag{1/q}[c]{$\frac1q$}
\includegraphics[width=7cm]{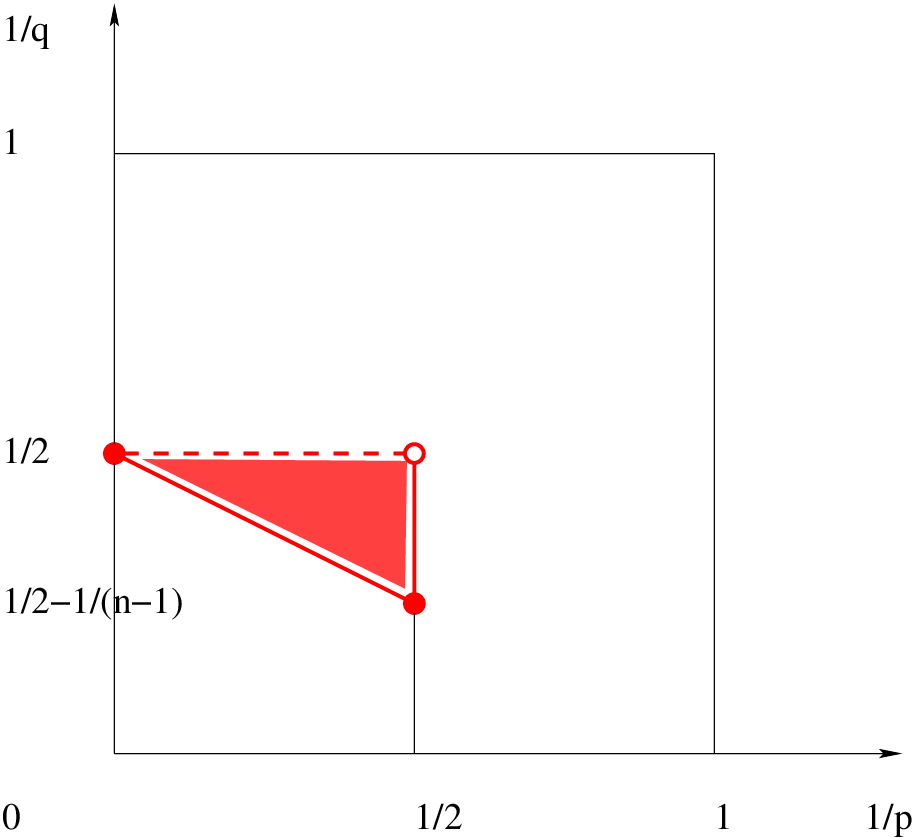}
\end{center}
\caption{Admissibility in dimension \ssf$n\!\ge\!4$}
\end{figure}

\begin{theorem}\label{StrichartzND}
Let \,$(p,q)$ and \,$(\tilde p, \tilde q)$ be two admissible couples,
and let
\begin{equation}\label{sigmas}\textstyle
\sigma\ge\frac{(n+1)}2\,\big(\frac12\ssb-\ssb\frac1q\big)
\quad\text{and}\quad
\tilde{\sigma}\ge\frac{(n+1)}2\,\big(\frac12\ssb-\ssb\frac1{\tilde{q}}\big)\,.
\end{equation}
Then the following Strichartz estimate holds
for solutions to the Cauchy problem \eqref{IP}\,{\rm :}
\begin{equation}\label{StrichartzEstimate1}
\|\ssf\nabla_{\R\times\Hn}u\ssf\|_{\vphantom{L^{p'}}
L^pH^{-\sigma,\hspace{.1mm}q}}
\lesssim\,\|\ssf f\ssf\|_{\vphantom{L^{p'}}H^1}\ssb
+\,\|\ssf g\ssf\|_{\vphantom{L^{p'}}L^2}\ssb
+\,\|\ssf F\ssf\|_{\vphantom{L^{p'}}
L^{\tilde{p}'}\!H^{\tilde{\sigma}\ssb,\tilde{q}'}}\ssf.
\end{equation}
\end{theorem}

\begin{proof}
We shall prove the following estimate,
which amounts to \eqref{StrichartzEstimate1}\,:
\begin{equation}\label{StrichartzEstimate2}\begin{aligned}
&\|\ssf\tilde{D}_{\,x}^{-\sigma+1/2}\,
u(t,x)\ssf\|_{L_t^pL_x^q\vphantom{L_x^{\tilde{q}}}}\ssb
+\,\|\ssf\tilde{D}_{\,x}^{-\sigma-1/2}\,\partial_{\ssf t}^{\vphantom{0}}
u(t,x)\ssf\|_{L_t^pL_x^q\vphantom{L_x^{\tilde{q}}}}\\
&\lesssim\ssf\|\ssf D_{\,x}^{1/2}\ssf f(x)\ssf\|_{
L_x^2\vphantom{L_x^{\tilde{q}}}}
+\ssf\|\ssf D_{\,x}^{-1/2}\ssf g(x)\ssf\|_{
L_x^2\vphantom{L_x^{\tilde{q}}}}
+\ssf\|\ssf\tilde{D}_{\,x}^{\ssf\tilde{\sigma}-1/2}\ssf
F(t,x)\,\|_{L_t^{\tilde{p}'}\!L_x^{\tilde{q}'}}\ssf.
\end{aligned}\end{equation}
Consider the operator
\begin{equation*}\textstyle
T\ssb f\ssf(t,x)=\tilde{D}_{\ssf x}^{-\sigma+1/2}\,
\frac{e^{\ssf\pm\ssf i\,t\ssf D_x}}{\sqrt{D_x}}\ssf f\ssf(x)\ssf,
\end{equation*}
initially defined from $L^2(\Hn)$ into $L^\infty(\R\ssf;\ssb L^2(\Hn)\ssb)$,
and its formal adjoint
\begin{equation*}
T^*\ssb F\ssf(x)=\ssb\int_{-\infty}^{+\infty}\hspace{-1mm}ds\;\textstyle
\tilde{D}_{\ssf x}^{-\sigma+1/2}\,\frac{e^{\ssf\mp\ssf i\,s\ssf D_x}}{\sqrt{D_x}}
\,F\ssf (s,x)\ssf,
\end{equation*}
initially defined from $L^1\ssb(\R\ssf;\ssb L^2\ssb(\Hn)\ssb)$ into $L^2(\Hn)$.
The \ssf$TT^*$ \ssb method consists in proving first the
\ssf$L^{p'}\ssb(\R\ssf;\ssb L^{q'}\ssb(\Hn)\ssb)\ssb
\to\ssb L^p(\R\ssf;\ssb L^q(\Hn)\ssb)$
\ssf boundedness of the operator
\begin{equation*}
TT^*\ssb F\ssf(t,x)
=\ssb\int_{-\infty}^{+\infty}\hspace{-1mm}ds\;\textstyle
\tilde{D}_{\ssf x}^{-2\ssf\sigma+1}\,\frac{e^{\ssf\pm\ssf i\ssf(t-s)\ssf D_x}}{D_x}
\,F\ssf(s,x)
\end{equation*}
and of its truncated version
\begin{equation*}
\mathcal{T}\ssb F\ssf(t,x)
=\ssb\int_{-\infty}^{\,t}\hspace{-1mm}ds\;\textstyle
\tilde{D}_{\ssf x}^{-2\ssf\sigma+1}\,\frac{e^{\ssf\pm\ssf i\ssf(t-s)\ssf D_x}}{D_x}
\,F\ssf(s,x)\ssf,
\end{equation*}
for every admissible couple $(p,q)$
and for every $\sigma\!\ge\!\frac{n+1}2\bigl(\frac12\!-\!\frac1q\bigr)$,
and in decoupling next the indices.

We may disregard the endpoint case \ssf$(p,q)\!=\!(\infty,2)$\ssf,
which is easily dealt with, using the boundedness on $L^2(\Hn)$ of
\,$e^{\,i\,t\ssf D}$ ($t\!\in\!\R$)
and \ssf$\tilde{D}^{-\sigma+\frac12}\ssf D^{-\frac12}$ ($\sigma\!\ge\!0$)\ssf.
Thus assume that $(p,q)$ is an admissible couple ,
which is different from the endpoints
$(\infty,2)$ and $(2,2\ssf\frac{n-1}{n-3})$.
It follows from \eqref{ParticularDispersiveEstimates} that
the norms \,$\|\ssf T\ssf T^*\ssb F(t,x)\ssf\|_{L_t^pL_x^q}$
and \,$\|\ssf\mathcal{T}\ssb F(t,x)\ssf\|_{L_t^pL_x^q}$
are bounded above by
\begin{equation}\label{HLS}
\Bigl\|\,\int_{\ssf0<|t-s|<1}\hspace{-1mm}ds\;
|\ssf t\ssb-\ssb s\ssf |^{-\alpha}\,
\|\ssf F(s,x)\ssf\|_{L_x^{q'}\vphantom{\big|}}\,\Bigr\|_{L_t^p}
+\;\Bigl\|\,\int_{\ssf|t-s|\ge1}\hspace{-1mm}ds\;
|\ssf t\ssb-\ssb s\ssf |^{-\frac32}\,
\|\ssf F(s,x)\ssf\|_{L_x^{q'}\vphantom{\big|}}\,\Bigr\|_{L_t^p}\,,
\end{equation}
where \ssf$\alpha\ssb=\ssb(n\!-\!1)(\frac12\!-\!\frac1q)\!\in\!(0,1)$\ssf.
On one hand, the convolution kernel
\,$|\ssf t\!-\!s\ssf|^{-\frac32}\ssf{\1}_{\,\{\ssf|\ssf t-s\ssf|\ge1\ssf\}}$
\ssf defines obviously
a bounded operator from \ssf $L^{p_1}\ssb(\R)$ to \ssf$L^{p_2}(\R)$,
for all \ssf$1\ssb\le\ssb p_1\!\le\ssb p_2\!\le\!\infty$\ssf,
in particular from \ssf$L^{p'}\ssb(\R)$ to \ssf$L^p(\R)$,
since \ssf$p\ssb\ge\ssb2$\ssf.
On the other hand, the convolution kernel
\,$|\ssf t\!-\!s\ssf|^{-\alpha}\,{\1}_{\,\{\ssf0<|\ssf t-s\ssf|\le1\ssf\}}$
\ssf with \ssf$0\!<\!\alpha\!<\!1$
\ssf defines a bounded operator
from \ssf$L^{p_1}\ssb(\R)$ to \ssf$L^{p_2}\ssb(\R)$,
for all \ssf$1\!<\!p_1,p_2\!<\!\infty$ \ssf such that
\ssf$0\ssb\le\!\frac1{p_1}\!-\!\frac1{p_2}\!\le\!1\!-\!\alpha\ssf$,
\ssf in particular from \ssf$L^{p'}\ssb(\R)$ to \ssf$L^p(\R)$,
since \ssf$p\ssb\ge\ssb2$ \ssf and \ssf$\frac2p\!\ge\ssb\alpha$\ssf.

At the endpoint \ssf$(p,q)\!=\!(2,2\ssf\frac{n-1}{n-3})$\ssf,
we have \ssf$\alpha\ssb=\!1$\ssf.
Thus the previous argument breaks down
and is replaced by the refined analysis carried out in \cite{KT}.
Notice that the problem lies only in the first part of \eqref{HLS}
and not in the second one,
which involves an integrable convolution kernel on \ssf$\R$\ssf.

Thus \,$TT^*$ and \,$\mathcal{T}$ are bounded
from \ssf$L^{p'}\ssb(\R\ssf;\ssb L^{q'}\ssb(\Hn)\ssb)$
to \ssf$L^p(\R\ssf;\ssb L^q(\Hn)\ssb)$\ssf,
for every admissible couple $(p,q)$\ssf.
As a consequence, \,$T^*$ is bounded
from \ssf$L^{p'}\ssb(\R\ssf;\ssb L^{q'}\ssb(\Hn)\ssb)$ \ssf to \ssf$L^2(\Hn)$
\ssf and \,$T$ \ssf is bounded
from \ssf$L^2(\Hn)$ \ssf to \ssf$L^p(\R\ssf;\ssb L^q(\Hn)\ssb)$.
We deduce in particular that
\begin{equation*}\textstyle
\bigl\|\ssf\tilde{D}_{x\vphantom{0}}^{-\sigma+1/2}\ssf(\cos t\ssf D_x)\ssf
f(x)\ssf\bigr\|_{L_t^pL_{x\vphantom{0}}^q\vphantom{\big|}}
\lesssim\,\bigl\|\ssf\tilde{D}_{x\vphantom{0}}^{-\sigma+1/2}\ssf
e^{\ssf\pm\ssf i\,t\ssf D_x\vphantom{\big|}}f\ssf(x)\ssf
\bigr\|_{L_t^pL_{x\vphantom{0}}^q\vphantom{\big|}}
\lesssim\,\bigl\|\ssf D_{x\vphantom{0}}^{1/2}f(x)\ssf
\bigr\|_{L_{x\vphantom{0}}^2\vphantom{\big|}}
\end{equation*}
and
\begin{equation*}\textstyle
\bigl\|\ssf\tilde{D}_{x\vphantom{0}}^{-\sigma+1/2}\ssf\frac{\sin t\ssf D_x}{D_x}\,
g\ssf(x)\ssf\bigr\|_{L_t^pL_{x\vphantom{0}}^q\vphantom{\big|}}
\lesssim\,\bigl\|\ssf\tilde{D}_{x\vphantom{0}}^{-\sigma+1/2}
D_{\ssf x\vphantom{0}}^{-1}\ssf
e^{\ssf\pm\ssf i\,t\ssf D_x\vphantom{\big|}}\ssf
g\ssf(x)\ssf\bigr\|_{L_t^pL_{x\vphantom{0}}^q\vphantom{\big|}}
\lesssim\,\bigl\|\ssf D_{\ssf x\vphantom{0}}^{-1/2}\ssf g\ssf(x)
\bigr\|_{L_{x\vphantom{0}}^2\vphantom{\big|}}\,.
\end{equation*}
%and that
%\begin{align*}\textstyle
%\bigl\|\ssf\tilde{D}_{x\vphantom{0}}^{-\sigma-1/2}D_x\ssf(\sin t\ssf D_x)\ssf
%f(x)\ssf\bigr\|_{L_t^pL_{x\vphantom{0}}^q\vphantom{\big|}}
%&\textstyle
%\lesssim\,\bigl\|\ssf\tilde{D}_{x\vphantom{0}}^{-\sigma-1/2}D_x\,
%e^{\ssf\pm\ssf i\,t\ssf D_x\vphantom{\big|}}f\ssf(x)\ssf
%\bigr\|_{L_t^pL_{x\vphantom{0}}^q\vphantom{\big|}}\\
%&\textstyleD
%\lesssim\,\bigl\|\ssf\tilde{D}_{x\vphantom{0}}^{-1}D_{x\vphantom{0}}^{3/2}f(x)\ssf
%\bigr\|_{L_{x\vphantom{0}}^2\vphantom{\big|}}
%\lesssim\,\bigl\|\ssf D_{x\vphantom{0}}^{1/2}f(x)\ssf
%\bigr\|_{L_{x\vphantom{0}}^2\vphantom{\big|}}\,,\\
%\textstyle
%\bigl\|\ssf\tilde{D}_{x\vphantom{0}}^{-\sigma-1/2}\ssf(\cos t\ssf D_x)\,
%g\ssf(x)\ssf\bigr\|_{L_t^pL_{x\vphantom{0}}^q\vphantom{\big|}}
%&\lesssim\,\bigl\|\ssf\tilde{D}_{x\vphantom{0}}^{-\sigma-1/2}\ssf
%e^{\ssf\pm\ssf i\,t\ssf D_x\vphantom{\big|}}\ssf
%g\ssf(x)\ssf\bigr\|_{L_t^pL_{x\vphantom{0}}^q\vphantom{\big|}}\\
%&\textstyle
%\lesssim\,\bigl\|\ssf\tilde{D}_{x\vphantom{0}}^{-1}
%D_{\ssf x\vphantom{0}}^{1/2}\ssf g\ssf(x)
%\bigr\|_{L_{x\vphantom{0}}^2\vphantom{\big|}}
%\lesssim\,\bigl\|\ssf D_{\ssf x\vphantom{0}}^{-1/2}\ssf g\ssf(x)
%\bigr\|_{L_{x\vphantom{0}}^2\vphantom{\big|}}\,,
%\end{align*}
%by using in addition the boundedness of the operator $\tilde{D}^{-1}D$ on $L^2(\Hn)$.
In summary, 
\begin{equation}\label{HomogeneousStrichartzEstimate}\textstyle
\|\ssf\tilde{D}_{x\vphantom{0}}^{-\sigma+1/2}\ssf
u_{\ssf\text{hom}}(t,x)\ssf\|_{L_t^pL_{x\vphantom{0}}^q\vphantom{\big|}}
\lesssim\,\|\ssf D_{x\vphantom{0}}^{1/2}
f(x)\|_{L_{x\vphantom{0}}^2\vphantom{\big|}}
+\,\|\ssf D_{x\vphantom{0}}^{-1/2}
g\ssf(x)\ssf\|_{L_{x\vphantom{0}}^2\vphantom{\big|}}\,.
\end{equation}

We next decouple the indices
in the \ssf$L^{p'}\!L^{q'}\hspace{-1mm}\to\!L^qL^q$ estimate
of \ssf$TT^*$ and \ssf$\mathcal{T}$.
Let $(p,q)\ssb\ne\ssb(\tilde{p},\tilde{q})$ be two admissible couples
and let $\sigma\!\ge\!\frac{n+1}2\bigl(\frac12\!-\!\frac1q\bigr)$,
$\tilde{\sigma}\!\ge\!\frac{n+1}2\bigl(\frac12\!-\!\frac1{\tilde{q}}\bigr)$.
Since \,$T$ and \,$T^*$ are separately continuous,
the operator
\begin{equation*}
TT^*\ssb F(t,x)\ssf=\int_{-\infty}^{+\infty}\hspace{-1mm}ds\;\textstyle
\tilde{D}_{x\vphantom{0}}^{-\sigma-\tilde{\sigma}+1}\,
\frac{e^{\ssf\pm\ssf i\ssf(t-s)\ssf D_x}}{D_x}\,F(s,x)
\end{equation*}
is bounded from $L^{\tilde{p}'}\ssb(\R\ssf;\ssb L^{\tilde{q}'}\ssb(\Hn)\ssb)$
to $L^p(\R\ssf;\ssb L^q(\Hn)\ssb)$\ssf.
According to \cite{CK},
this result remains true for the truncated operator
\begin{equation*}
\mathcal{T}\ssb F(t,x)\,=\int_{-\infty}^{\,t}\hspace{-1mm}ds\;\textstyle
\tilde{D}_{x\vphantom{0}}^{-\sigma-\tilde{\sigma}+1}\,
\frac{e^{\ssf\pm\ssf i\ssf(t-s)\ssf D_x}}{D_x}\,F(s,x)
\end{equation*}
and hence for
\begin{equation*}
\widetilde{\mathcal{T}}\ssb F(t,x)\,=\int_{\,0}^{\ssf t}ds\;\textstyle
\tilde{D}_{x\vphantom{0}}^{-\sigma-\tilde{\sigma}+1}\,
\frac{\sin\ssf(t-s)D_x\ssf}{D_x}\,F(s,x)
\end{equation*}
as long as \ssf$p$ \ssf and \ssf$\tilde{p}$
\ssf are not both equal to \ssf$2$\ssf.
For the remaining case,
where \ssf$p\ssb=\ssb\tilde{p}\ssb=\ssb2$ \ssf and
\ssf$2\ssb<\ssb q\ssb\ne\ssb\tilde{q}\ssb\le\ssb2\ssf\frac{n-1}{n-3}$\ssf,
we argue as in the proof of \cite[Theorem 6.3]{APV1}
by resuming part of the bilinear approach in \cite{KT}.
Hence
\begin{equation}\label{InhomogeneousStrichartzEstimate}
\|\ssf\tilde{D}_{x\vphantom{0}}^{-\sigma+1/2}\ssf
u_{\ssf\text{inhom}}(t,x)\ssf\|_{L_t^pL_{x\vphantom{0}}^q\vphantom{\big|}}
\lesssim\,\|\ssf\tilde{D}_{x\vphantom{0}}^{\ssf\tilde{\sigma}-1/2}\ssf
F(t,x)\ssf\|_{L_t^{\tilde{p}'}\!L_{x\phantom{0}}^{\tilde{q}'}\vphantom{\big|}}
\end{equation}
for all admissible couples $(p,q)$ and $(\tilde{p},\tilde{q})$.

The Strichartz estimate
\begin{equation*}
\|\ssf\tilde{D}_{x\vphantom{0}}^{-\sigma+1/2}\ssf
u(t,x)\ssf\|_{L_t^pL_{x\vphantom{0}}^q\vphantom{\big|}}
\lesssim\,\|\ssf D_{x\vphantom{0}}^{1/2}f(x)\ssf
\|_{L_{x\vphantom{0}}^2\vphantom{\big|}}\ssb
+\,\|\ssf D_{x\vphantom{0}}^{-1/2}g\ssf(x)\ssf
\|_{L_{x\vphantom{0}}^2\vphantom{\big|}}\ssb
+\,\|\ssf\tilde{D}_{x\vphantom{0}}^{\ssf\tilde{\sigma}-1/2}\ssf
F(t,x)\|_{L_t^{\tilde{p}'}\!L_{x\vphantom{0}}^{\tilde{q}'}\vphantom{\big|}}
\end{equation*}
is obtained by summing up
the homogeneous estimate \eqref{HomogeneousStrichartzEstimate}
and the inhomogeneous estimate \eqref{InhomogeneousStrichartzEstimate}.
As far as it is concerned, the Strichartz estimate of
\begin{equation*}\textstyle
\partial_{\ssf t}u(t,x)
=-\,(\ssf\sin t\ssf D_x)\ssf D_x\ssf f\ssf(x)
+(\ssf\cos t\ssf D_x)\ssf g\ssf(x)\ssb
+{\displaystyle\int_{\,0}^{\ssf t}}ds\,
[\ssf\cos\ssf(t-s)\ssf D_x\ssf]\ssf F(s,x)
\end{equation*}
is obtained in the same way and is actually easier.
More precisely, we consider this time the operator
\begin{equation*}
\widetilde{T}\ssb f\ssf(t,x)=\tilde{D}_{\ssf x}^{-\sigma}\,
e^{\ssf\pm\ssf i\,t\ssf D_x}\ssf f\ssf(x)\ssf,
\end{equation*}
and its adjoint
\begin{equation*}
\widetilde{T}^*\ssb F\ssf(x)\ssf=\ssb\int_{-\infty}^{+\infty}\hspace{-1mm}ds\;
\tilde{D}_{\ssf x}^{-\sigma}\,e^{\ssf\mp\ssf i\,s\ssf D_x}\,F\ssf (s,x)\ssf.
\end{equation*}
\end{proof}

By using the Sobolev embedding theorem,
Theorem \ref{StrichartzND} can be extended to all couples
$\bigl(\frac1p,\frac1q\bigr)$ and $\bigl(\frac1{\tilde{p}},\frac1{\tilde{q}}\bigr)$
in the square
\begin{equation}\label{square}\textstyle
\bigl[\ssf0\ssf,\frac12\ssf\bigr]\!\times\!\bigl(\ssf0\ssf,\frac12\ssf\bigr)\,
\cup\,\bigl\{\bigr(0,\frac12\bigr)\bigr\}\,.
\end{equation}

\begin{corollary}\label{GeneralizedStrichartzND}
Let \,$(p,q)$, $(\tilde{p},\tilde{q})$ be two couples
corresponding to the square \eqref{square}
and let $\sigma,\tilde{\sigma}\!\in\!\R$\ssf.
Assume that \,$\sigma\!\ge\!\sigma(p,q)$\ssf,
where
\begin{equation*}\textstyle
\sigma(p,q)=\frac{n+1}2\ssf(\frac12\ssb-\ssb\frac1q)
+\ssf\max\,\{\ssf0\ssf,
\frac{n-1}2\ssf(\frac12\ssb-\ssb\frac1q)\ssb-\ssb\frac1p\ssf\}
=\begin{cases}
\ssf\frac{n+1}2\ssf(\frac12\ssb-\ssb\frac1q)
&\text{if \;}\frac1p\ssb\ge\ssb\frac{n-1}2\ssf(\frac12\ssb-\ssb\frac1q)\ssf,\\
\,n\,(\frac12\ssb-\ssb\frac1q)\ssb-\ssb\frac1p
&\text{if \;}\frac1p\ssb\le\ssb\frac{n-1}2\ssf(\frac12\ssb-\ssb\frac1q)\ssf,\\
\end{cases}\end{equation*}
and similarly \,$\tilde{\sigma}\ssb\ge\ssb\sigma(\tilde{p},\tilde{q})$\ssf.
Then the conclusion of Theorem \ref{StrichartzND}
holds for solutions to the Cauchy problem \eqref{IP}.
More precisely we have again the Strichartz estimate
\vspace{2mm}

\centerline{\eqref{StrichartzEstimate1}\hfill$
\|\ssf\nabla_{\R\times\Hn}u\ssf\|_{\vphantom{L^{p'}}
L^pH^{-\sigma,\hspace{.1mm}q}}
\lesssim\,\|\ssf f\ssf\|_{\vphantom{L^{p'}}H^1}\ssb
+\,\|\ssf g\ssf\|_{\vphantom{L^{p'}}L^2}\ssb
+\,\|\ssf F\ssf\|_{\vphantom{L^{p'}}
L^{\tilde{p}'}\!H^{\tilde{\sigma}\ssb,\tilde{q}'}}\ssf,
$\hfill}\vspace{2mm}

\noindent
which amounts to
\vspace{2mm}

\centerline{\eqref{StrichartzEstimate2}\hspace{22.5mm}$
\|\ssf\tilde{D}_{\,x}^{-\sigma+1/2}\,
u(t,x)\ssf\|_{L_t^pL_x^q\vphantom{L_x^{\tilde{q}}}}\ssb
+\,\|\ssf\tilde{D}_{\,x}^{-\sigma-1/2}\,\partial_{\ssf t}^{\vphantom{0}}
u(t,x)\ssf\|_{L_t^pL_x^q\vphantom{L_x^{\tilde{q}}}}\\
$\hfill}\vspace{1mm}

\centerline{\hspace{10mm}$
\lesssim\ssf\|\ssf D_{\,x}^{1/2}\ssf f(x)\ssf\|_{
L_x^2\vphantom{L_x^{\tilde{q}}}}
+\ssf\|\ssf D_{\,x}^{-1/2}\ssf g(x)\ssf\|_{
L_x^2\vphantom{L_x^{\tilde{q}}}}
+\ssf\|\ssf\tilde{D}_{\,x}^{\ssf\tilde{\sigma}-1/2}\ssf
F(t,x)\,\|_{L_t^{\tilde{p}'}\!L_x^{\tilde{q}'}}\ssf.
$}
\end{corollary}

\begin{figure}[ht]
\begin{center}
\psfrag{0}[c]{$0$}
\psfrag{1}[c]{$1$}
\psfrag{1/2}[c]{$\frac12$}
\psfrag{1/2-1/(n-1)}[c]{$\frac12\!-\!\frac1{n-1}$}
\psfrag{1/p}[c]{$\frac1p$}
\psfrag{1/q}[c]{$\frac1q$}
\includegraphics[width=7cm]{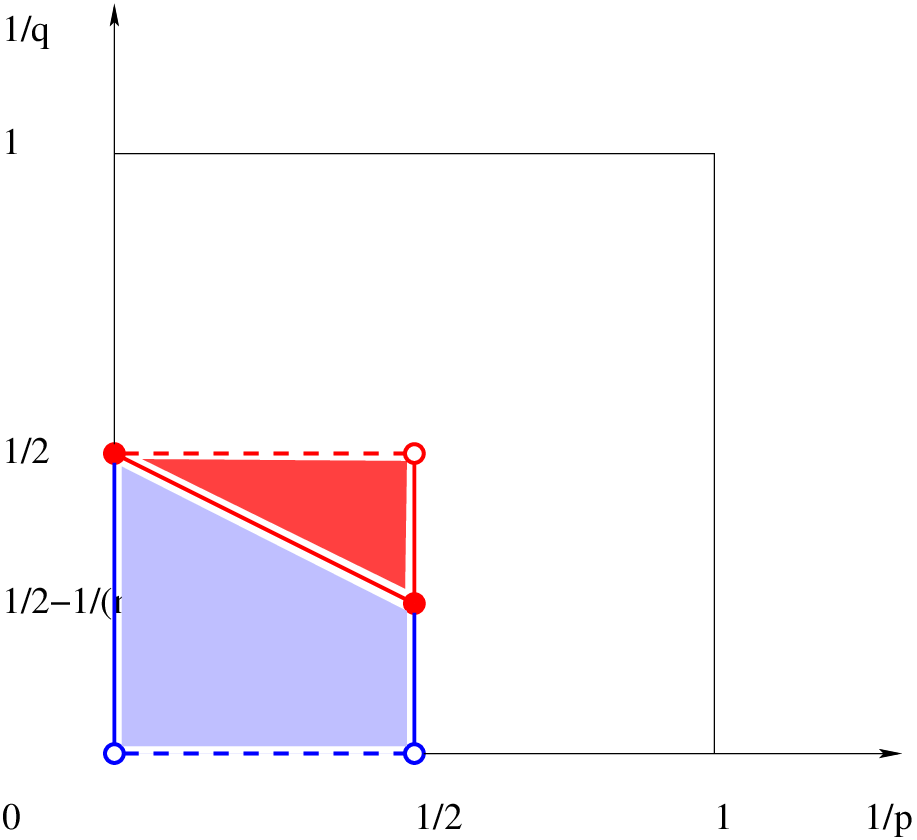}
\end{center}
\caption{Case \ssf$n\!\ge\!4$}
\end{figure}

\begin{proof}
We may retrict to the limit cases \ssf$\sigma\ssb=\ssb\sigma(p,q)$
\ssf and \ssf$\tilde{\sigma}\ssb=\ssb\sigma(\tilde{p},\tilde{q})$\ssf.
Define \ssf$Q$ \ssf by
\begin{equation*}\textstyle
\frac1Q=\ssf\begin{cases}
\ssf\frac1q
&\text{if \;}\frac1p\ssb
\ge\ssb\frac{n-1}2\ssf(\frac12\ssb-\ssb\frac1q)\ssf,\\
\ssf\frac12\ssb-\ssb\frac2{n-1}\ssf\frac1p
&\text{if \;}\frac1p\ssb
\le\ssb\frac{n-1}2\ssf(\frac12\ssb-\ssb\frac1q)\ssf.\\
\end{cases}\end{equation*}
and similarly \ssf$\tilde{Q}$\ssf.
Since $(p,Q)$ and $(\tilde{p},\tilde{Q})$ are admissible couples,
it follows from Theorem \ref{StrichartzND}
and more precisely from \eqref{StrichartzEstimate2} that
\begin{equation}\label{StrichartzEstimate3}\begin{aligned}
&\bigl\|\ssf\tilde{D}_{x\vphantom{|}}^{-\Sigma+1/2}\ssf u(t,x)\ssf\bigr\|_{
L_{t\vphantom{|}}^{p\vphantom{|}}L_{x\vphantom{|}}^Q\vphantom{L_x^{\tilde{Q}}}}
+\,\bigl\|\ssf\tilde{D}_{x\vphantom{|}}^{-\tilde{\Sigma}-1/2}\ssf
\partial_{\ssf t\ssf}u(t,x)\ssf\bigr\|_{
L_{t\vphantom{|}}^{p\vphantom{|}}L_{x\vphantom{|}}^Q\vphantom{L_x^{\tilde{Q}}}}\\
&\lesssim\,\bigl\|\ssf D_{x\vphantom{|}}^{1/2}f(x)\ssf
\bigr\|_{L_{x\vphantom{|}}^{2\vphantom{|}}\vphantom{L_x^{\tilde{Q}}}}
+\,\bigl\|\ssf D_{x\vphantom{|}}^{-1/2}\ssf g\ssf(x)\ssf
\bigr\|_{L_{x\vphantom{|}}^{2\vphantom{|}}\vphantom{L_x^{\tilde{Q}}}}
+\,\bigl\|\ssf\tilde{D}_{x\vphantom{|}}^{\ssf\tilde{\Sigma}-1/2}\ssf F(t,x)\ssf
\bigr\|_{L_{t\vphantom{|}}^{\tilde{p}'\vphantom{|}}\!L_{x\vphantom{|}}^{\tilde{Q}'}}
\end{aligned}\end{equation}
where \ssf$\Sigma\ssb=\ssb\frac{n+1}2\ssf(\frac12\!-\!\frac1Q)$ \ssf and
\ssf$\tilde{\Sigma}\ssb=\ssb\frac{n+1}2\ssf(\frac12\!-\!\frac1{\tilde{Q}})$\ssf.
Since \ssf$\sigma\!-\!\Sigma\ssb=\ssb n\,(\frac1Q\!-\!\frac1q)$\ssf,
we have
\begin{equation}\label{SobolevLeft}
\bigl\|\ssf\tilde{D}_{x\vphantom{|}}^{-\sigma+1/2}\ssf u(t,x)\ssf
\bigr\|_{L_{t\vphantom{|}}^{p\vphantom{|}}L_{x\vphantom{|}}^{q\vphantom{|}}}
\lesssim\,\bigl\|\ssf\tilde{D}_{x\vphantom{|}}^{-\Sigma+1/2}\ssf u(t,x)\ssf
\bigr\|_{L_{t\vphantom{|}}^{p\vphantom{|}}L_{x\vphantom{|}}^{Q\vphantom{|}}\vphantom{L_x^Q}}
\end{equation}
according to the Sobolev embedding theorem (Proposition B.1).
Similarly,
\begin{equation}\label{SobolevRight}
\bigl\|\ssf\tilde{D}_{x\vphantom{|}}^{\ssf\tilde{\Sigma}-1/2}\ssf F(t,x)\ssf
\bigr\|_{L_{t\vphantom{|}}^{\tilde{p}'\vphantom{|}}\!L_{x\vphantom{|}}^{\tilde{Q}'}}
\lesssim\,\bigl\|\ssf\tilde{D}_{x\vphantom{|}}^{\ssf\tilde{\sigma}-1/2}\ssf F(t,x)\ssf
\bigr\|_{L_{t\vphantom{|}}^{\tilde{p}'\vphantom{|}}\!L_{x\vphantom{|}}^{\tilde{q}}}\,.
\end{equation}
We conclude by combining
\eqref{StrichartzEstimate3}, \eqref{SobolevLeft}, \eqref{SobolevRight}, 
\end{proof}

\begin{remark}\label{Strichartz3D}
Theorem \ref{StrichartzND} and Corollary \ref{GeneralizedStrichartzND}
hold true in dimension \,$n\!=\!3$ with the same proofs.
Notice that the endpoint \,$(p,q)\!=\!(2,\infty)$ is excluded.
These results hold in particular for the 3D wave equation \eqref{nonshiftedWave}
and include the Strichartz estimates obtained
by Metcalfe and Taylor \cite[Section 4]{MT}
in the smaller region
\begin{equation*}\textstyle
\bigl\{\ssf\bigl(\frac1p,\frac1q\bigr)\!\in\!
\bigl[\ssf0,\frac12\ssf\bigr]\!\times\!\bigl(\ssf0,\frac12\ssf\bigr]\bigm|
\frac1p\ssb\le\ssb3\ssf\bigl(\frac12\!-\!\frac1q\bigr)\ssf\bigr\}
\smallsetminus\bigl\{\bigl(\frac12,\frac13\bigr)\bigr\}\,.
\end{equation*}
\end{remark}

\begin{figure}[ht]
\begin{center}
\psfrag{0}[c]{$0$}
\psfrag{1}[c]{$1$}
\psfrag{1/2}[c]{$\frac12$}
\psfrag{1/p}[c]{$\frac1p$}
\psfrag{1/q}[c]{$\frac1q$}
\includegraphics[width=7cm]{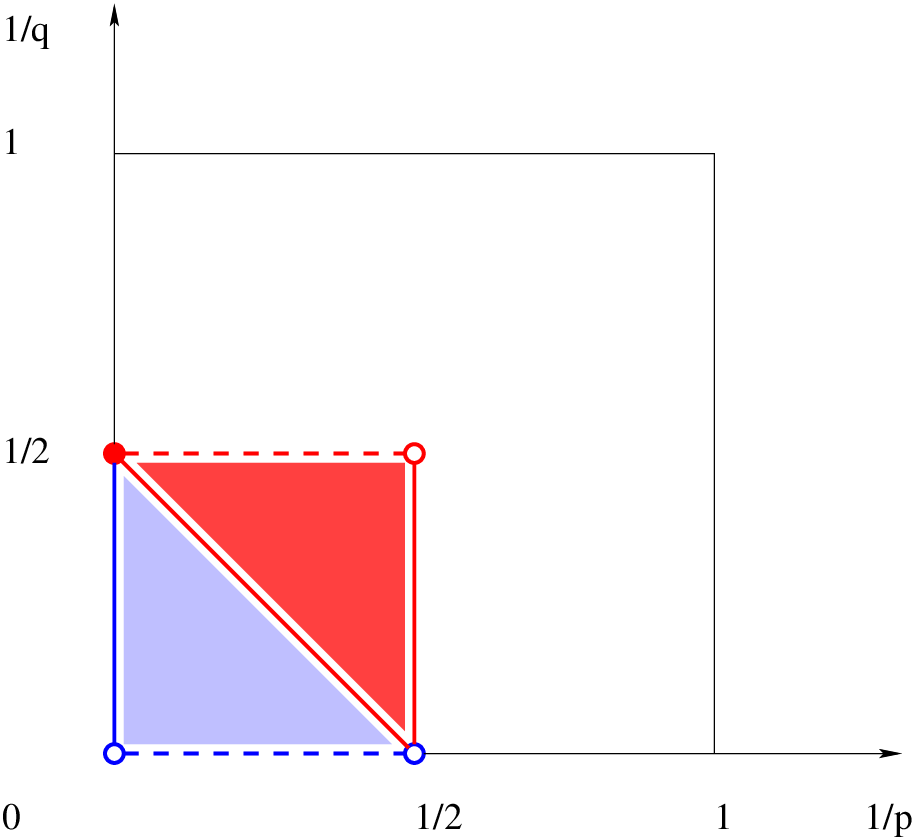}
\end{center}
\caption{Case \ssf$n\!=\!3$}
\end{figure}

\begin{remark}\label{dim2end}
The analysis carried out in this section still holds in dimension \,$n\!=\!2$\ssf,
except for the first convolution kernel in \eqref{HLS},
which becomes
\begin{equation*}
|\ssf t\ssb-\ssb s\ssf|^{-\alpha}\,
(\ssf1\!-\ssb\log|\ssf t\ssb-\ssb s\ssf|\ssf)^{\ssf\beta}\,
{\1}_{\,\{\ssf0<|\ssf t-s\ssf|<1\ssf\}}\,,
\end{equation*}
with \,$\alpha\ssb=\ssb\frac12\ssb-\ssb\frac1q$
and \,$\beta\ssb=\ssb2\,(\frac12\ssb-\ssb\frac1q)$\ssf.
Consequently,
the admissibility region in Theorem \ref{StrichartzND} becomes
\begin{equation*}\textstyle
\bigl\{\ssf\bigl(\frac1p,\frac1q\bigr)\!\in\!
\bigl(\ssf0,\frac12\ssf\bigr]\!\times\!\bigl(0,\frac12\ssf\bigr)
\bigm|\frac1p\!>\!\frac12\bigl(\frac12\!-\!\frac1q\bigr)\ssf\bigr\}
\cup\bigl\{\ssb\bigl(0,\frac12\bigr)\ssb\bigr\}
\end{equation*}
and the inequality \,$\sigma\!\ge\!\sigma(p,q)$\ssf,
resp. $\tilde{\sigma}\!\ge\!\sigma(\tilde{p},\tilde{q})$
in Corollary \ref{GeneralizedStrichartzND}
becomes strict in the triangle
\begin{equation*}\textstyle
\bigl\{\ssf\bigl(\frac1p,\frac1q\bigr)\!\in\!
\bigl(\ssf0,\frac14\ssf\bigr)\!\times\!\bigl(\ssf0,\frac12\ssf\bigr)
\bigm|\frac1p\!\le\!\frac12\bigl(\frac12\!-\!\frac1q\bigr)\ssf\bigr\}\,.
\end{equation*}
\end{remark}

\begin{figure}[ht]
\begin{center}
\psfrag{0}[c]{$0$}
\psfrag{1}[c]{$1$}
\psfrag{1/2}[c]{$\frac12$}
\psfrag{1/4}[c]{$\frac14$}
\psfrag{1/p}[c]{$\frac1p$}
\psfrag{1/q}[c]{$\frac1q$}
\includegraphics[width=7cm]{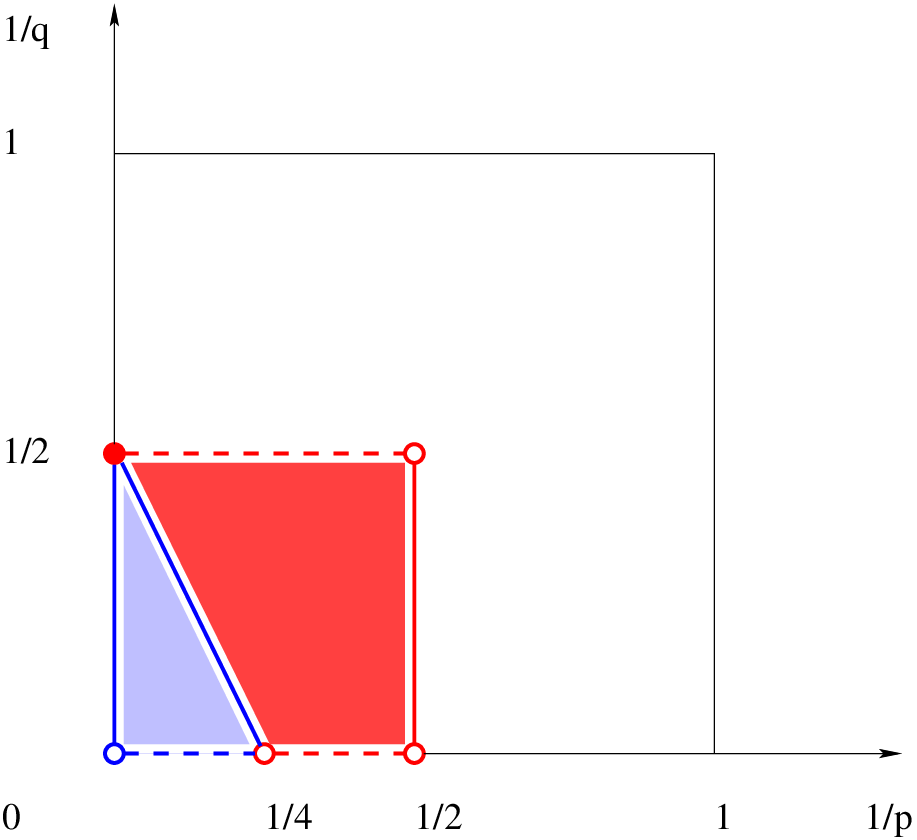}
\end{center}
\caption{Case \ssf$n\!=\!2$}
\end{figure}

\section{Global well--posedness in $L^p(\R\ssf,L^q(\Hn))$}
\label{GWP}

In this section, following the classical fixed point scheme,
we use the Strichartz estimates obtained in Section \ref{Strichartz}
to prove global well--posedness for the semilinear equation
\begin{equation}\label{NLKG}
\begin{cases}
\;\partial_{\,t}^{\,2\ssf}u(t,x)+D_x^{\ssf2\ssf}u(t,x)=F(u(t,x))\\
\;u(0,x)=f(x)\ssf,\;\partial_{\ssf t}|_{t=0}\,u(t,x)=g(x)
\end{cases}
\end{equation}
on \ssf$\Hn$ with power--like nonlinearities
\begin{equation*}
F(u)\sim|u|^{\gamma}
\qquad(\gamma\!>\!1)
\end{equation*}
and small initial data \ssf$f$ and \ssf$g$\ssf.
We assume \ssf$n\!\ge\!3$ \ssf throughout the section
and discuss the 2--dimensional case in the final remark.
The statement and proof of our result involve the following powers
\begin{equation}\label{powers}\begin{gathered}\textstyle
\gammaone\ssb=\ssb1\ssb+\ssb\frac3n\,,\quad
\gammatwo\ssb=\ssb1\ssb
+\ssb\frac{2\vphantom{\frac12}}{\frac{n-1}2+\frac2{n-1}}\,,\quad
\gammaconf\ssb=\ssb1\ssb+\ssb\frac4{n-1}\,,\\
\gammathree=\begin{cases}
\;\frac{\frac{n+6}2+\frac2{n-1}
+\sqrt{\ssf4\ssf n\ssf+\ssf(\frac{6-n}2+\frac2{n-1})^2\ssf}}
{n\vphantom{\frac12}}
&\text{if \,}n\ssb\le\ssb5\ssf,\\
\;1\ssb+\ssb\frac{2\vphantom{\frac12}}{\frac{n-1}2-\frac1{n-1}}
&\text{if \,}n\ssb\ge\ssb6\ssf,\\
\end{cases}\\
\gammafour=\begin{cases}
\;1\ssb+\ssb\frac4{n-2}
&\text{if \,}n\ssb\le\ssb5\ssf,\\
\,\frac{n-1}2\ssb+\ssb\frac3{n+1}\ssb
-\ssb\sqrt{\bigl(\frac{n-3}2\ssb+\ssb\frac3{n+1}\bigr)^2\!
-4\,\frac{n-1}{n+1}\ssf}
&\text{if \,}n\ssb\ge\ssb6\ssf,\\
\end{cases}
\end{gathered}\end{equation}
and the following curves
\begin{equation}\label{curves}\begin{gathered}\textstyle
\sigma_1(\gamma)\ssb=\ssb\frac{n+1}4\ssb
-\ssb\frac{(n+1)\ssf(n+5)}{8\ssf n}\ssf
\frac1{\gamma\ssf-\ssf\frac{n+1}{2\ssf n}}\,,\quad
\sigma_2(\gamma)\ssb=\ssb\frac{n+1}4\ssb-\ssb\frac1{\gamma-1}\,,\quad
\sigma_3(\gamma)\ssb=\ssb\frac n2\ssb-\ssb\frac2{\gamma-1}\,.
\end{gathered}\end{equation}
\smallskip

\begin{center}
\begin{tabular}{|c|c|c|c|c|c|}
\hline
$\hphantom{........}n\hphantom{.........}\vphantom{\Big|}$
&$\hphantom{........}\gammaone\hphantom{........}$
&$\hphantom{........}\gammatwo\hphantom{........}$
&$\hphantom{......}\ssf\gammaconf\ssf\hphantom{......}$
&$\hphantom{........}\gammathree\hphantom{........}$
&$\hphantom{........}\gammafour\hphantom{........}$
\\\hline
$3\vphantom{\Big|}$
&$2$
&$2$
&$3$
&$\frac{11+\sqrt{73}}6\ssb\simeq\ssb3,\hspace{-.5mm}26$
&$5$
\\\hline
$4\vphantom{\Big|}$
&$\frac74\ssb=\ssb1,\hspace{-.6mm}75$
&$\frac{25}{13}\ssb\simeq\ssb1,\hspace{-.5mm}92$
&$\frac73\ssb\simeq\ssb2,\hspace{-.5mm}33$
&$\frac52\ssb\simeq\ssb2,\hspace{-.5mm}5$
&$3$
\\\hline
$5\vphantom{\Big|}$
&$\frac85\ssb\simeq\ssb1,\hspace{-.5mm}6$
&$\frac95\ssb\simeq\ssb1,\hspace{-.5mm}8$
&$2$
&$\frac{6\ssf+\sqrt{21}}5\ssb\simeq\ssb2,\hspace{-.6mm}12$
&$\frac73\ssb\simeq\ssb2,\hspace{-.5mm}33$
\\\hline
$6\vphantom{\Big|}$
&$\frac32\ssb=\ssb1,\hspace{-.5mm}5$
&$\frac{49}{29}\ssb\simeq\ssb1,\hspace{-.5mm}69$
&$\frac95\ssb=\ssb1,\hspace{-.5mm}8$
&$\frac{43}{23}\ssb\simeq\ssb1,\hspace{-.5mm}87$
&$2$
\\\hline
$\ge\!7\vphantom{\Big|}$
&$<\ssb\gammatwo$
&$<\ssb\gammaconf$
&$<\ssb\gammathree$
&$<\ssb\gammafour$
&$<\ssb2$
\\\hline
\end{tabular}
\end{center}
\medskip

The powers $\gammaone$, $\gammatwo$\ssf, $\gammaconf$
and the curves \ssf$C_1$, $C_2$\ssf, $C_3$
parametrized by $\sigma_1$, $\sigma_2$\ssf, $\sigma_3$
occur already in the Euclidean setting.
More precisely, they are involved in the conditions,
illustrated in Figure \ref{LocalRegularityEuclidean},
of minimal regularity $\sigma$ on the initial data $f$\ssb, $g$
which are needed in order to ensure local well--posedness of \eqref{NLKG}.
We refer again to \cite{Kap, LS, KT} for more details.
Notice that, in dimension $n\!=\!3$\ssf,
\ssf$\gammaone$ coincides with \ssf$\gammatwo$
and there is no curve \ssf$C_1$.

\begin{figure}[ht]
\begin{center}
\psfrag{0}[c]{$0$}
\psfrag{1}[c]{$1$}
\psfrag{1/2}[c]{$\frac12$}
\psfrag{n/2}[c]{$\frac n2$}
\psfrag{C1}[c]{\color{red}$C_1$}
\psfrag{C2}[c]{\color{red}$C_2$}
\psfrag{C3}[c]{\color{red}$C_3$}
\psfrag{gamma}[c]{$\gamma$}
\psfrag{gamma1}[c]{$\gammaone$}
\psfrag{gamma2}[c]{$\gammatwo$}
\psfrag{gammaconf}[c]{$\gammaconf$}
\psfrag{sigma}[c]{$\sigma$}
\includegraphics[width=120mm]{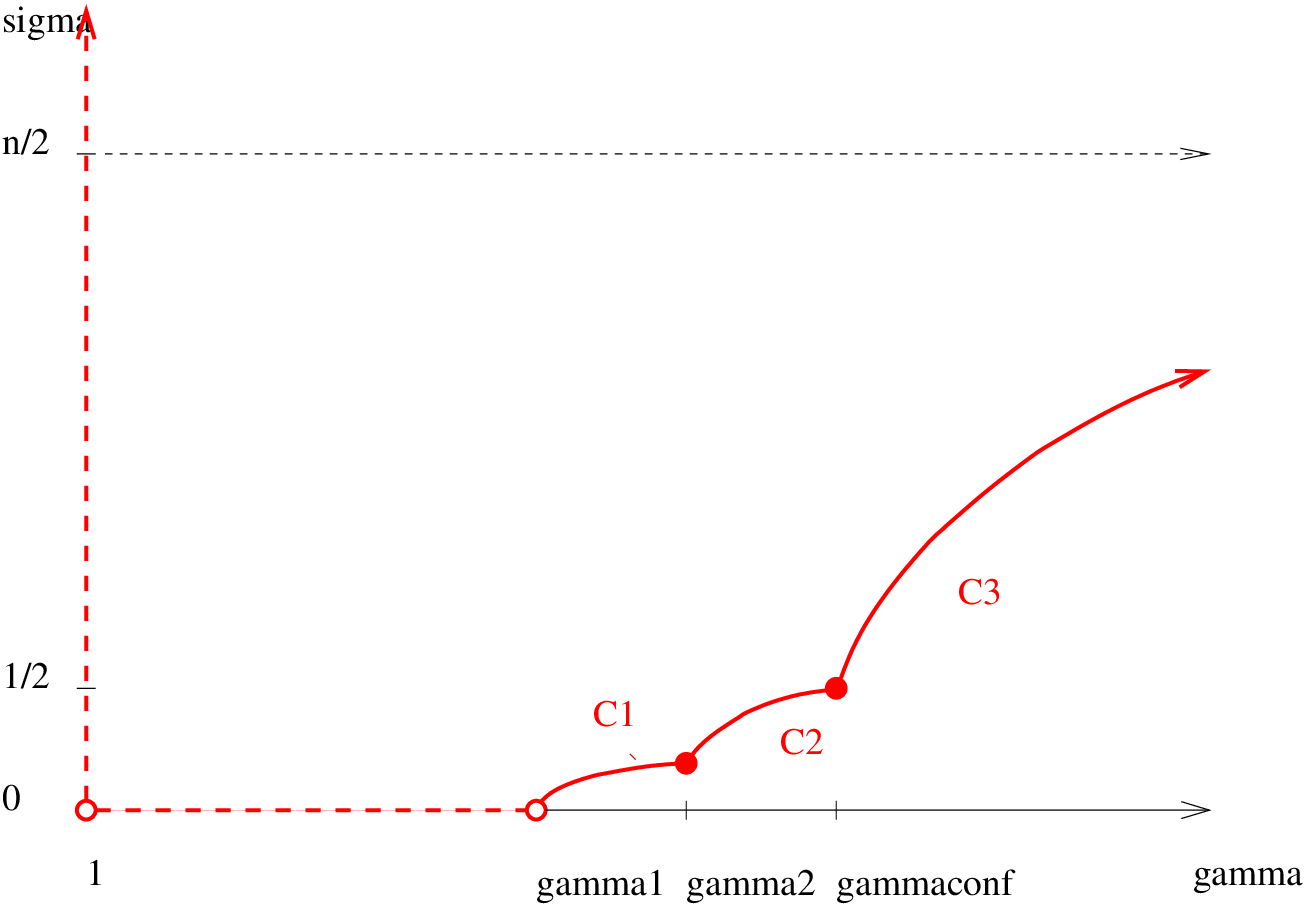}
\caption{Regularity for LWP on $\Rn$ in dimension $n\!\ge\!3$}
\label{LocalRegularityEuclidean}
\end{center}
\end{figure}

As mentionned in the introduction,
global well--posedness of \eqref{NLKG} on \ssf$\Rn$ requires additional conditions.
Recall that smooth solutions with small amplitude blow up or not
depending whether $\gamma$ is smaller or larger than
the critical power $\gammazero$ defined in \eqref{strauss}.

In Section \ref{Strichartz}
we have obtained Strichartz estimates on \ssf$\Hn$
for a range of admissible couples which is wider than on \ssf$\Rn$.
As a consequence, we deduce in this section
stronger well--posedness results for \eqref{NLKG}.
In particular, we prove global well--posedness
for small initial data in \ssf$H^\sigma(\Hn)\!\times\!H^{\sigma-1}(\Hn)$\ssf,
if \ssf$1\!<\!\gamma\!<\!\gammaone$ \ssf and \ssf$\sigma\!>\!0$ \ssf is small.
Thus there is no blow--up for small powers \ssf$\gamma\!>\!1$ \ssf on \ssf$\Hn$,
in sharp contrast with \ssf$\Rn$. 

\begin{theorem}\label{GWPND}
Assume that the nonlinearity \,$F$ satisfies
\begin{equation}\label{nonlinearity}
|F(u)|\le C\,|u|^\gamma,\quad 
|\ssf F(u)\!-\!F(v)\ssf|\le
C\,(\ssf|u|^{\gamma-1}\!+\ssb|v|^{\gamma-1}\ssf)\,|\ssf u\!-\!v\ssf|\,.
\end{equation}
Then, in dimension \,$n\!\ge\!3$\ssf,
the equation \eqref{NLKG} is globally well--posed
for small initial data in \,$H^\sigma(\Hn)\ssb\times\ssb H^{\sigma-1}(\Hn)$
provided
\begin{equation}\label{RegularityND}\begin{cases}
\;\sigma\ssb=\ssb0^+
&\text{if \;}1\ssb<\ssb\gamma\ssb\le\ssb\gammaone\ssf,\\
\;\sigma\ssb=\ssb\sigma_1(\gamma)
&\text{if \;}\gammaone\ssb<\ssb\gamma\ssb\le\gammatwo\ssf,\\
\;\sigma\ssb=\ssb\sigma_2(\gamma)
&\text{if \;}\gammatwo\ssb\le\ssb\gamma\ssb\le\gammaconf\,,\\
\;\sigma\ssb=\ssb\sigma_3(\gamma)
&\text{if \;}\gammaconf\ssb\le\ssb\gamma\ssb\le\gammafour\ssf,\\
\end{cases}\end{equation}
where \,$\sigma\ssb=\ssb0^+$ stands for
any \,$\sigma\!>\!0$ sufficiently close to \ssf$0$\ssf.
More precisely, in each case, there exist
\,$2\!\le\!p,q\!<\!\infty$ and \,$\delta,\varepsilon\!>\!0$
such that, for any initial data
\ssf$(f,g)\ssb\in\ssb H^\sigma(\Hn)\ssb\times\ssb H^{\sigma-1}(\Hn)$
with norm $\le\ssb\delta$,
the Cauchy problem \eqref{NLKG}
has a unique solution \,$u$ with norm $\le\ssb\varepsilon$
in the Banach space
\begin{equation*}
X=C\ssf(\R\ssf;\ssb H^{\sigma}({\mathbb{H}^n}))\ssf
\cap C^1(\R\ssf;\ssb H^{\sigma-1}({\mathbb{H}^n}))\ssf
\cap L^p(\R\ssf;\ssb L^q({\mathbb{H}^n}))\ssf.
\end{equation*}
\end{theorem}

\begin{figure}[ht]
\begin{center}
\psfrag{0}[c]{$0$}
\psfrag{1}[c]{$1$}
\psfrag{1/2}[c]{$\frac12$}
\psfrag{n/2}[c]{$\frac n2$}
\psfrag{C1}[c]{\color{red}$C_1$}
\psfrag{C2}[c]{\color{red}$C_2$}
\psfrag{C3}[c]{\color{red}$C_3$}
\psfrag{gamma}[c]{$\gamma$}
\psfrag{gamma1}[c]{$\gammaone$}
\psfrag{gamma2}[c]{$\gammatwo$}
\psfrag{gamma3}[c]{$\gammafour$}
\psfrag{gammaconf}[c]{$\gammaconf$}
\psfrag{sigma}[c]{$\sigma$}
\includegraphics[width=120mm]{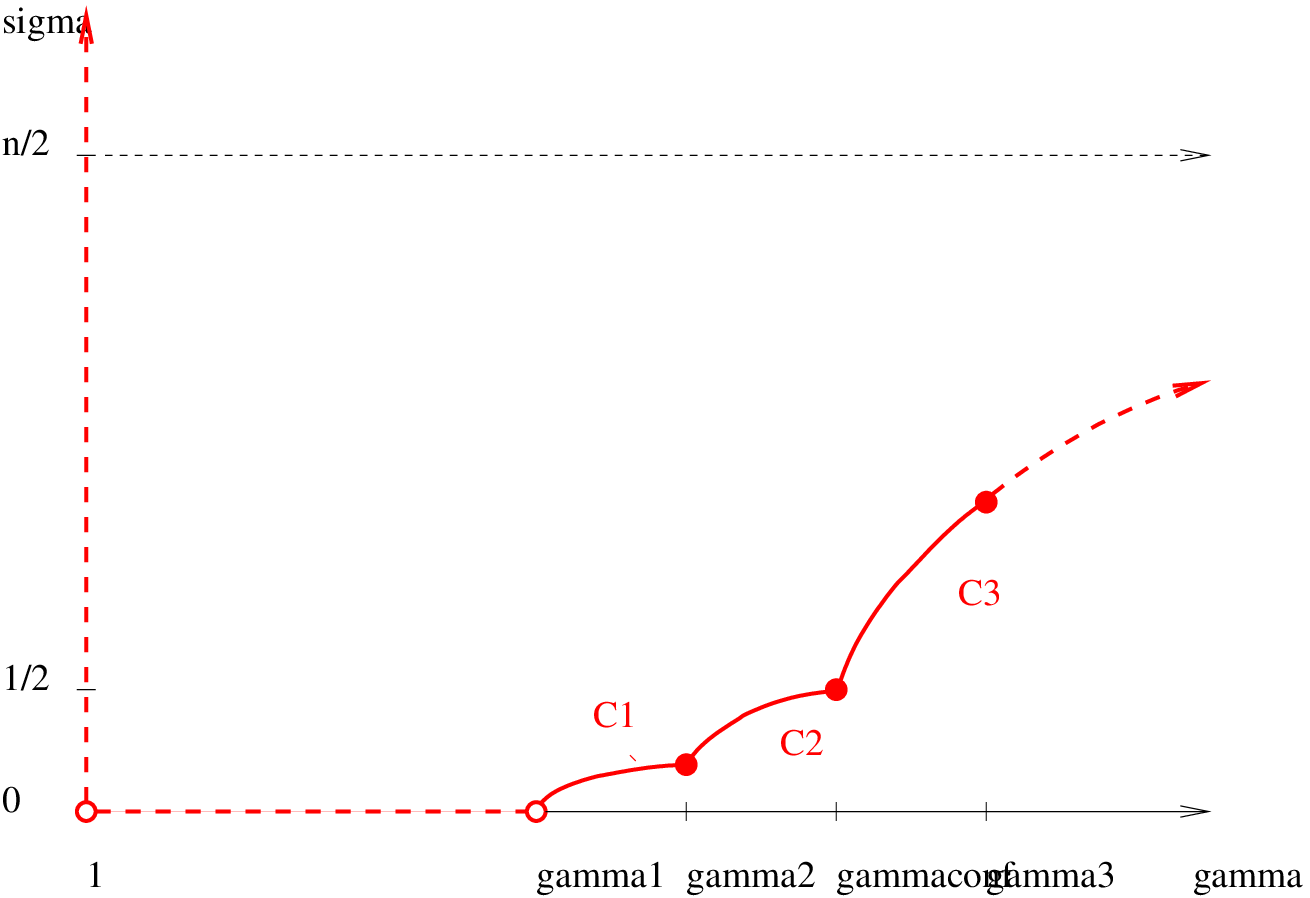}
\caption{Regularity for GWP on $\Hn$ in dimension $n\!\ge\!3$}
\label{GlobalRegularityn}
\end{center}
\end{figure}

\begin{remark}
In dimension \,$n\!=\!3$\ssf,
$\gammaone$ coincides with \ssf$\gammatwo$,
the second and third conditions in \eqref{RegularityND} boil down to
\vskip-.5mm

\centerline{
$\sigma\ssb\ge\ssb\sigma_2(\gamma)$\quad if
\;$\gammaone\!=\ssb\gammatwo\!<\ssb\gamma\ssb\le\ssb\gammaconf$
}\vskip.5mm

\noindent
and there is no curve \,$C_1$ in Figure \ref{GlobalRegularityn}.
\end{remark}

\begin{proof}[Proof of Theorem \ref{GWPND}
for \,$1\!<\!\gamma\ssb\le\!\gammaconf$\ssf]
We resume the fixed point method based on Strichartz estimates.
Define \ssf$u\ssb=\ssb\Phi(v)$
\ssf as the solution to the Cauchy problem
\begin{equation}
\begin{cases}
\;\partial_{\ssf t}^{\,2}u(t,x)-D_x^{\ssf2}\ssf u(t,x)=F(v(t,x))\ssf,\\
\;u(0,x)=f(x)\ssf,
\;\partial_t|_{t=0}\,u(t,x)=g(x)\ssf,
\end{cases}
\end{equation}
which is given by Duhamel's formula\,:
\begin{equation*}\textstyle
u(t,x)=(\cos t\ssf D_x)\ssf f(x)
+\frac{\sin t\ssf D_x}{D_x}\ssf g\ssf(x)
+\!{\displaystyle\int_{\,0}^{\,t}}\ssb ds\;
\frac{\sin\ssf(t-s)\ssf D_x}{D_x}\,F(s,x)\,.
\end{equation*}
On one hand,  according to Theorem \ref{StrichartzND},
the Strichartz estimate
\begin{align*}
&\|\ssf u(t,x)\ssf\|_{L_t^{\infty\vphantom{\tilde{p}'}}\ssb
H_{x\vphantom{t}}^{\sigma\vphantom{\tilde{p}'}}}
+\,\|\ssf\partial_{\ssf t}u(t,x)\ssf\|_{L_t^{\infty\vphantom{\tilde{p}'}}\ssb
H_{x\vphantom{t}}^{\sigma-1\vphantom{\tilde{p}'}}}
+\,\|\ssf u(t,x)\ssf\|_{L_t^{p\vphantom{\tilde{p}'}}
L_{x\vphantom{t}}^{q\vphantom{\tilde{p}'}}}\\
&\lesssim\,\|\ssf f(x)\ssf\|_{H_{x\vphantom{t}}^{\sigma\vphantom{\tilde{p}'}}}
+\,\|\,g(x)\ssf\|_{H_{x\vphantom{t}}^{\sigma-1\vphantom{\tilde{p}'}}}
+\,\|\ssf F(v(t,x))\ssf\|_{L_t^{\tilde{p}'}\hspace{-.75mm}
H_{x\vphantom{t}}^{\sigma+\tilde{\sigma}-1,\ssf\tilde{q}'}}
\end{align*}
holds whenever
\begin{equation*}
\begin{cases}
\,\text{$(p,q)$ \ssf and \ssf$(\tilde{p},\tilde{q})$
\ssf are admissible couples\,;}\\
\;\sigma\!\ge\!\frac{n+1}2\ssf\big(\frac12\!-\!\frac1{\tilde{q}}\bigr)
\text{ \ssf and \ssf}
\tilde{\sigma}\!\ge\!\frac{n+1}2\ssf\big(\frac12\!-\!\frac1{\tilde{q}}\bigr).
\end{cases}
\end{equation*}
On the other hand,
by our nonlinear assumption \eqref{nonlinearity}
and by the Sobolev embedding theorem (Theorem B.1),
we have
\begin{equation*}
\|\ssf F(v(t,x))\ssf\|_{L_t^{\tilde{p}'}\!
H_{x\vphantom{t}}^{\sigma+\tilde{\sigma}-1,\ssf\tilde{q}'}}
\!\lesssim\|\ssf|v(t,x)|^\gamma\|_{L_t^{\tilde{p}'}\!
H_{x\vphantom{t}}^{\sigma+\tilde{\sigma}-1,\tilde{q}'}}
\!\lesssim\|\ssf|v(t,x)|^\gamma
\|_{L_t^{\tilde{p}'}\!L_{x\vphantom{t}}^{\tilde{Q}'}}
\!\lesssim\|\ssf v(t,x)\ssf
\|_{L_t^{\gamma\tilde{p}'}\!L_{x\vphantom{t}}^{\gamma\tilde{Q}'}}^{\,\gamma}\ssf,
\end{equation*}
provided
\begin{equation}\label{condition2}\textstyle
\sigma\ssb+\ssb\tilde{\sigma}\ssb\le\ssb1\ssf,
\;1\ssb<\ssb\tilde{Q}'\!\le\ssb\tilde{q}'\!<\ssb\infty
\quad\text{and}\quad
\frac n{\tilde{Q}'}\ssb-\ssb\frac n{\tilde{q}'}\ssb
\le\ssb1\!-\ssb\sigma\ssb-\ssb\tilde{\sigma}\,.
\end{equation}
In order to remain within the same function space,
we require in addition that
\begin{equation*}
\gamma\,\tilde{p}^{\,\prime}\!=p
\quad\text{and}\quad
\gamma\,\tilde{Q}^{\,\prime}\!=q\ssf.
\end{equation*}
In summary,
\begin{equation}\begin{aligned}\label{FixedPoint1}
&\|\ssf u(t,x)\ssf\|_{L_{\ssf t}^{\infty\vphantom{\tilde{p}'}}\ssb
H_{x\vphantom{t}}^{\sigma\vphantom{\tilde{p}'}}}
+\,\|\ssf\partial_{\ssf t}u(t,x)\ssf
\|_{L_{\ssf t}^{\infty\vphantom{\tilde{p}'}}\ssb
H_{\ssf x\vphantom{t}}^{\sigma-1\vphantom{\tilde{p}'}}}
+\,\|\ssf u(t,x)\ssf\|_{L_t^{p\vphantom{\tilde{p}'}}
L_{x\vphantom{t}}^{q\vphantom{\tilde{p}'}}}\\
&\le\,C\;\Bigl\{\,
\|\ssf f(x)\ssf\|_{H_{x\vphantom{t}}^{\sigma\vphantom{\tilde{p}'}}}
+\,\|\,g(x)\ssf\|_{H_{\ssf x\vphantom{t}}^{\sigma-1\vphantom{\tilde{p}'}}}
+\,\|v\|_{L_{\ssf t}^{p\vphantom{0}}
L_{\ssf x\vphantom{t}}^{q\vphantom{0}}}^{\,\gamma}\ssf\Bigr\}
\end{aligned}\end{equation}
if the following set of conditions is satisfied\,:
\begin{equation}\label{conditions1}\begin{cases}
\,\text{(a)}&\hspace{-2mm}
\text{$(p,q)$ \ssf and \ssf$(\tilde{p},\tilde{q})$
\ssf are admissible couples\,;}\\
\,\text{(b)}&\hspace{-2mm}
\sigma\!\ge\!\frac{n+1}2\ssf\big(\frac12\!-\!\frac1q\bigr)\ssf,
\;\tilde{\sigma}\!\ge\!\frac{n+1}2\ssf\big(\frac12\!-\!\frac1{\tilde{q}}\bigr)\ssf,
\;\sigma\ssb+\ssb\tilde{\sigma}\ssb\le\ssb1\,;\\
\,\text{(c)}&\hspace{-2mm}
\frac\gamma p\ssb+\ssb\frac1{\tilde{p}}\ssb=\ssb1\,;\\
\,\text{(d)}&\hspace{-2mm}
\text{$1\ssb\le\ssb\frac\gamma q\ssb+\ssb\frac1{\tilde{q}}\ssb
\le\ssb1\ssb+\ssb\frac{1-\sigma-\tilde{\sigma}}n$\,;}\\
\,\text{(e)}&\hspace{-2mm}
\text{$q\ssb>\ssb\gamma$\ssf.}\\
\end{cases}\end{equation}
For such a choice,
\ssf$\Phi$ \ssf maps the Banach space
\begin{equation*}
X=\ssf C\ssf(\R\ssf;\ssb H^\sigma(\Hn))\ssf\cap\ssf
C^1(\R\ssf;\ssb H^{\sigma-1}(\Hn))\ssf\cap\ssf
L^p(\R\ssf;\ssb L^q(\Hn))\,,
\end{equation*}
equipped with the norm
\begin{equation*}
\|u\|_X=\ssf\|\ssf u(t,x)\ssf
\|_{L_{\ssf t}^{\infty\vphantom{p}}H_{x\vphantom{t}}^{\sigma\vphantom{p}}}
+\,\|\ssf\partial_{\ssf t}u(t,x)\ssf
\|_{L_{\ssf t}^{\infty\vphantom{p}}H_{\ssf x\vphantom{t}}^{\sigma-1\vphantom{p}}}
+\,\|u\|_{L_t^pL_{x\vphantom{t}}^q}\,,
\end{equation*}
into itself.
Let us show that \ssf$\Phi$ \ssf is a contraction on the ball
\begin{equation*}
X_\varepsilon=\{\,u\!\in\!X\mid\|u\|_X\!\le\!\varepsilon\,\}\,,
\end{equation*}
provided \ssf$\varepsilon\!>\!0$
\ssf and \ssf$\|f\|_{H^\sigma}\!+\ssb\|g\|_{H^{\sigma-1}}$
\ssf are sufficiently small.
Let \ssf$v,\tilde{v}\!\in\!X$ \ssf and
\ssf$u\!=\!\Phi(v)\ssf,$ \ssf$\tilde{u}\!=\!\Phi(\tilde{v})$\ssf.
By resuming the arguments leading to \eqref{FixedPoint1}
and by using in addition H\"older's inequality,
we obtain the estimate
\begin{equation}\begin{aligned}\label{FixedPoint2}
\|\ssf u\ssb-\ssb\tilde{u}\ssf\|_{X\vphantom{L_t^{\tilde{Q}'}}}\!
&\le C\,\|\ssf F(v)\ssb-\ssb F(\tilde v)\ssf
\|_{L_{\ssf t}^{\tilde{p}'}\!L_{x\vphantom{t}}^{\tilde{Q}'}}\\
&\le C\,\|\ssf\{\ssf|v|^{\gamma-1}\!+\ssb|\tilde v|^{\gamma-1}\}\,
|\ssf v\ssb-\ssb\tilde v\ssf|\ssf
\|_{L_{\ssf t}^{\tilde{p}'}\!L_{x\vphantom{t}}^{\tilde{Q}'}}\\
&\le C\,\bigl\{\ssf\|v\|_{L_t^pL_x^q}^{\,\gamma-1}\ssb
+\|\tilde v\|_{L_t^pL_x^q}^{\,\gamma-1}\bigr\}\,
\|\ssf v\ssb-\ssb\tilde{v}\ssf\|_{L_t^pL_x^q}\\
&\le C\,\bigl\{\ssf\|v\|_{X\vphantom{L_t^Q}}^{\gamma-1}\!
+\ssb\|\tilde v\|_{X\vphantom{L_t^Q}}^{\gamma-1}\bigr\}\,
\|\ssf v\ssb-\ssb\tilde{v}\ssf\|_{X\vphantom{L_t^Q}}\,.
\end{aligned}\end{equation}
Thus, if we assume \ssf$\|v\|_X\!\le\ssb\varepsilon$\ssf,
\ssf$\|\tilde v\|_X\!\le\ssb\varepsilon$ \ssf and
$\|f\|_{H^\sigma}\!+\ssb\|g\|_{H^{\sigma-1}}\!\le\ssb\delta$\ssf,
then \eqref{FixedPoint1} and \eqref{FixedPoint2} yield
\begin{equation*}
\|u\|_X\!\le\ssb C\ssf\delta\ssb+\ssb C\ssf\varepsilon^\gamma\ssf,\quad
\|\tilde{u}\|_X\!\le\ssb C\ssf\delta\ssb+\ssb C\ssf\varepsilon^\gamma
\quad\text{and}\quad
\|\ssf u\ssb-\ssb\tilde{u}\ssf\|_X\!
\le\ssb2\,C\ssf\varepsilon^{\gamma-1}\,\|\ssf v\ssb-\ssb\tilde{v}\ssf\|_X\,.
\end{equation*}
Hence
\begin{equation*}\textstyle
\|u\|_X\!\le\ssb\varepsilon\ssf,\quad
\|\tilde{u}\|_X\le\varepsilon
\quad\text{and}\quad
\|\,u-\tilde u\,\|_X\le\frac12\,\|\,v-\tilde v\,\|_X
\end{equation*}
if \,$C\,\varepsilon^{\ssf\gamma-1}\!\le\ssb\frac14$
\,and \,$C\,\delta\ssb\le\ssb\frac34\,\varepsilon$\ssf.
One concludes
by applying the fixed point theorem
in the complete metric space $X_\varepsilon$\ssf.

It remains for us to check that
the set of conditions \eqref{conditions1}
can be fulfilled in the various cases \eqref{RegularityND}.
Notice that we may assume the following equalities in (\ref{conditions1}.b)\,:
\begin{equation*}\textstyle
\sigma=\frac{n+1}2\ssf\bigl(\frac12\ssb-\ssb\frac1q\bigr)
\quad\text{and}\quad
\tilde{\sigma}=\frac{n+1}2\ssf\bigl(\frac12\ssb-\ssb\frac1{\tilde{q}}\bigr)\ssf.
\end{equation*}
Thus \eqref{conditions1} reduces to the set of conditions\,:
\begin{equation}\label{conditions2}\begin{cases}
\,\text{(a)}&\hspace{-2mm}
\text{$(p,q)$ \ssf and \ssf$(\tilde{p},\tilde{q})$
\ssf are admissible couples\,;}\\
\,\text{(b)}&\hspace{-2mm}
\frac1q\ssb+\ssb\frac1{\tilde{q}}\ssb\ge\ssb\frac{n-1}{n+1}\,;\\
\,\text{(c)}&\hspace{-2mm}
\frac\gamma p\ssb+\ssb\frac1{\tilde{p}}\ssb=\ssb1\,;\\
\,\text{(d.i)}&\hspace{-2mm}
\frac\gamma q\ssb+\ssb\frac1{\tilde{q}}\ssb\ge\ssb1\,;\\
\,\text{(d.ii)}&\hspace{-2mm}
\bigl(\frac{2\ssf n}{n-1}\,\gamma\ssb-\ssb\frac{n+1}{n-1}\bigr)\ssf
\frac1q\ssb+\frac1{\tilde{q}}\ssb\le\ssb\frac{n+1}{n-1}\,;\\
\,\text{(e)}&\hspace{-2mm}
q\ssb>\ssb\gamma\ssf.\\
\end{cases}\end{equation}
We shall discuss these conditions first in high dimensions
\ssf and next in low dimensions.
\smallskip

\noindent$\blacktriangleright$
\,Assume that \ssf$n\ssb\ge\ssb6$\ssf.
\smallskip

Firstly notice that \ssf$\gammaconf\ssb<\ssb2$\ssf.
As \ssf$\gamma\!\le\!\gammaconf$ \ssf and \ssf$q\!>\!2$\ssf,
(\ref{conditions2}.e) is trivially satisfied.

Secondly we claim that
(\ref{conditions2}.a) and (\ref{conditions2}.c)
reduce to the single condition
\begin{equation}\label{ac}\textstyle
\frac\gamma q\ssb+\ssb\frac1{\tilde{q}}\ssb\ge\ssb\frac{\gamma+1}2\ssb-\ssb\frac2{n-1}
\end{equation}
in the square
\begin{equation}\label{SQUARE1}\textstyle
R\ssf=\bigl[\ssf\frac12\!-\!\frac1{n-1},\frac12\ssf\bigr)\ssb
\times\ssb\bigl[\ssf\frac12\!-\!\frac1{n-1},\frac12\ssf\bigr)\ssf.
\end{equation}
More precisely, if \ssf$(p,q)$ \ssf and \ssf$(\tilde{p},\tilde{q})$
\ssf are admissible couples satisfying (\ref{conditions2}.c),
then \ssf$\bigl(\frac1q,\frac1{\tilde{q}}\bigr)$ \ssf is a
\linebreak
\vspace{-4.8mm}

\noindent
point in the square $R$ \ssf satisfying \eqref{ac}.
Conversely,
if \ssf$\bigl(\frac1q,\frac1{\tilde{q}}\bigr)\!\in\!R$ \ssf satisfies \eqref{ac},
then there
\linebreak
\vspace{-4.75mm}

\noindent
exists a one parameter family
of admissible couples \ssf$(p,q)$ \ssf and \ssf$(\tilde{p},\tilde{q})$
\ssf satisfying (\ref{conditions2}.c).
All these claims can be deduced from the following four quadrant figure.

\begin{figure}[ht]
\begin{center}
\psfrag{1/p}[c]{$\frac1p$}
\psfrag{1/ptilde}[c]{$\frac1{\tilde{p}}$}
\psfrag{1/q}[c]{$\frac1q$}
\psfrag{1/qtilde}[c]{$\frac1{\tilde{q}}$}
\psfrag{1/2}[c]{$\frac12$}
\psfrag{1/2-1/(n-1)center}[c]{$\frac12\!-\!\frac1{n-1}$}
\psfrag{1/2-1/(n-1)left}[l]{$\frac12\!-\!\frac1{n-1}$}
\psfrag{1/2-1/(n-1)right}[r]{$\frac12\!-\!\frac1{n-1}$}
\psfrag{1/(2gamma)}[c]{$\frac1{2\ssf\gamma}$}
\psfrag{1-gamma/2left}[l]{$1\!-\!\frac\gamma2$}
\psfrag{1-gamma/2right}[r]{$1\!-\!\frac\gamma2$}
\psfrag{1/2-1/gamma1/(n-1)}[c]{$\frac12\!-\!\frac1\gamma\ssf\frac1{n-1}$}
\psfrag{1/2-(2-gamma)/(n-1)}[c]{$\frac12\!-\!\frac{2-\gamma}{n-1}$}
\psfrag{1-gamma/2}[c]{$1\!-\!\frac\gamma2$}
\psfrag{gamma/p+1/ptilde=1}[c]{$\frac\gamma p\ssb+\ssb\frac1{\tilde{p}}=1$}
\psfrag{gamma/q+1/qtilde=(gamma+1)/2-2/(n-1)}[c]
{$\frac\gamma q\ssb+\ssb\frac1{\tilde{q}}
=\frac{\gamma+1}2\ssb-\ssb\frac2{n-1}$}
\includegraphics[height=110mm]{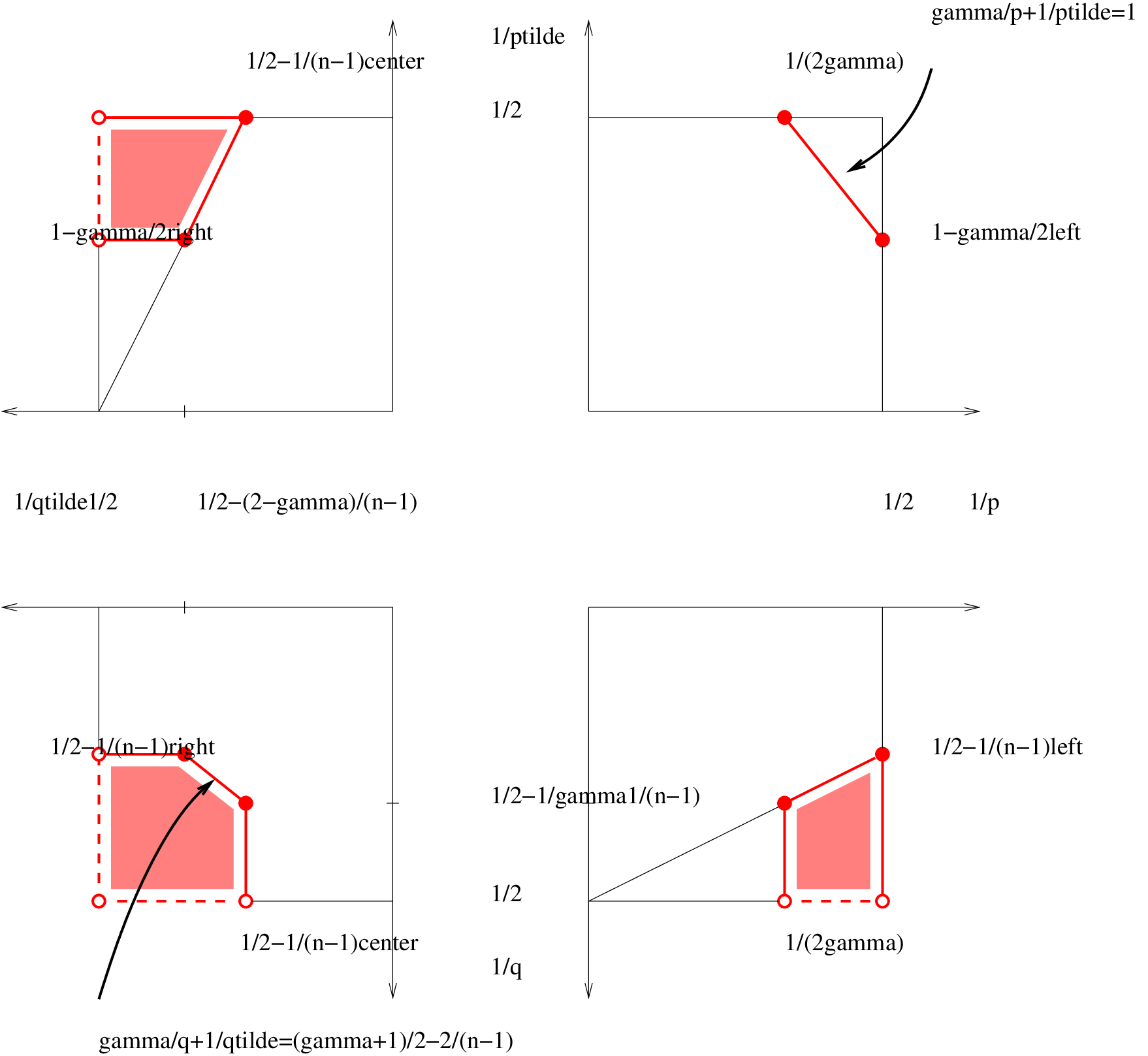}
\caption{Case \,$\gamma\ssb<\ssb2$}
\label{FIGURE1}
\end{center}
\end{figure}

Thirdly, as \ssf$\gamma\!\le\!\gammaconf$\ssf,
\eqref{ac} follows actually from (\ref{conditions2}.d.i).

Fourthly we claim that (\ref{conditions2}.b) follows
from (\ref{conditions2}.d.i) and (\ref{conditions2}.d.ii).
Consider indeed the three lines
\begin{equation}\label{3lines}\begin{cases}
\,\text{(b)}&\hspace{-2mm}
\frac1q\ssb+\ssb\frac1{\tilde{q}}\ssb=\ssb\frac{n-1}{n+1}\\
\,\text{(d.i)}&\hspace{-2mm}
\frac\gamma q\ssb+\ssb\frac1{\tilde{q}}\ssb=\ssb1\\
\,\text{(d.ii)}&\hspace{-2mm}
\bigl(\frac{2\ssf n}{n-1}\,\gamma\ssb-\ssb\frac{n+1}{n-1}\bigr)\ssf
\frac1q\ssb+\frac1{\tilde{q}}\ssb=\ssb\frac{n+1}{n-1}\\
\end{cases}\end{equation}
in the plane with coordinates $\bigl(\frac1q,\frac1{\tilde q}\bigr)$.
On one hand, they meet at the same point,
whose coordinates are
\begin{equation}\label{qone}\begin{cases}
\,\frac1{q_1}=\frac2{n+1}\ssf\frac1{\gamma-1}\,,\\
\,\frac1{\tilde{q}_1}=\frac{n-1}{n+1}-\frac2{n+1}\ssf\frac1{\gamma-1}\,.\\
\end{cases}\end{equation}
On the other hand, the coefficients of \,$\frac1q$
\ssf occur in increasing order in \eqref{3lines}\,:
\begin{equation*}\textstyle
1<\gamma<\frac{2\ssf n}{n-1}\,\gamma\ssb-\ssb\frac{n+1}{n-1}\,.
\end{equation*}
Hence (\ref{conditions2}.b) follows
from (\ref{conditions2}.d.i) and (\ref{conditions2}.d.ii),
which define the sector \ssf$S$
with vertex $\bigl(\frac1{q_1},\frac1{\tilde{q}_1}\bigr)$
and edges (\ref{3lines}.d.i), (\ref{3lines}.d.ii)
depicted in Figure \ref{Sector}.

\begin{figure}[ht]
\begin{center}
\psfrag{b}[r]{(\ref{3lines}.b)}
\psfrag{b}[r]{(\ref{3lines}.b)}
\psfrag{d.i}[r]{(\ref{3lines}.d.i)}
\psfrag{d.ii}[l]{(\ref{3lines}.d.ii)}
\psfrag{(1/q1,1/qtilde1)}[l]{$\bigl(\frac1{q_1},\frac1{\tilde{q}_1}\bigr)$}
\includegraphics[height=60mm]{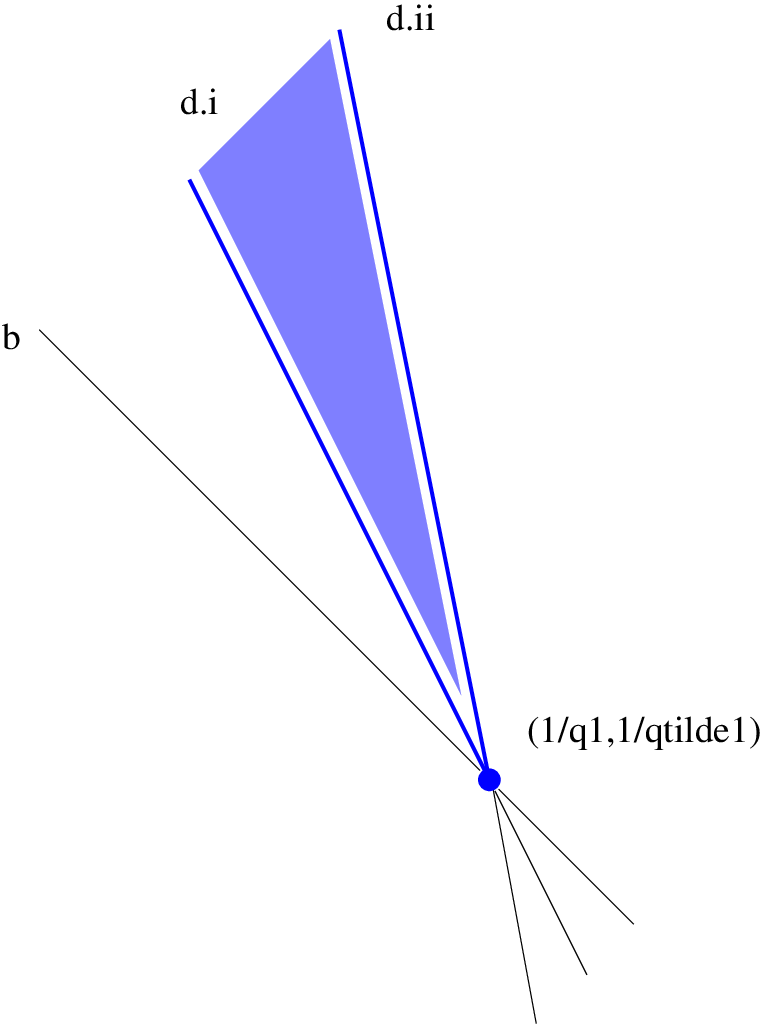}
\caption{Sector $S$}
\label{Sector}
\end{center}
\end{figure}

In summary, the set of conditions \eqref{conditions2}
reduce to the three conditions
(\ref{conditions2}.d.i), (\ref{conditions2}.d.ii), \eqref{SQUARE1}
in the plane with coordinates $\bigl(\frac1q,\frac1{\tilde{q}}\bigr)$.
In order to conclude,
we examine the possible
\linebreak
\vspace{-4.75mm}

\noindent
intersections of the sector \ssf$S$
defined by (\ref{conditions2}.d.i) and (\ref{conditions2}.d.ii)
with the square $R$ \ssf
defined by \eqref{SQUARE1},
and we determine in each case the minimal regularity
\ssf$\sigma\ssb=\ssb\frac{n+1}2\bigl(\frac12\ssb-\ssb\frac1q\bigr)$.
\medskip

\newpage

\noindent$\bullet$
\textit{\,Case 1}\,:
\,$1\!<\!\gamma\ssb\le\!\gammaone$
\medskip

\noindent
In the following three subcases,
the minimal regulatity condition is \ssf$\sigma\!>\!0$\ssf,
as \ssf$\frac1q\!>\!\frac12$ \ssf can
\linebreak
\vspace{-5mm}

\noindent
be chosen arbitrarily close to \ssf$\frac12$\ssf.
\medskip

\noindent$\circ$
\textit{\,Subcase 1.1}\,:
\,$1\!<\!\gamma\ssb\le\!1\!+\ssb\frac2n$
\vspace{-2.5mm}

\begin{figure}[ht]
\begin{center}
\psfrag{(1/2,1/2)}[l]{\color{red}$(\frac12,\frac12)$}
\psfrag{(1/2-1/(n-1),1/2-1/(n-1))}[r]
{\color{red}$(\frac12\!-\!\frac1{n-1},\frac12\!-\!\frac1{n-1})$}
\psfrag{(1/q1,1/qtilde1)left}[l]{\color{blue}$(\frac1{q_1},\frac1{\tilde{q}_1})$}
\includegraphics[height=60mm]{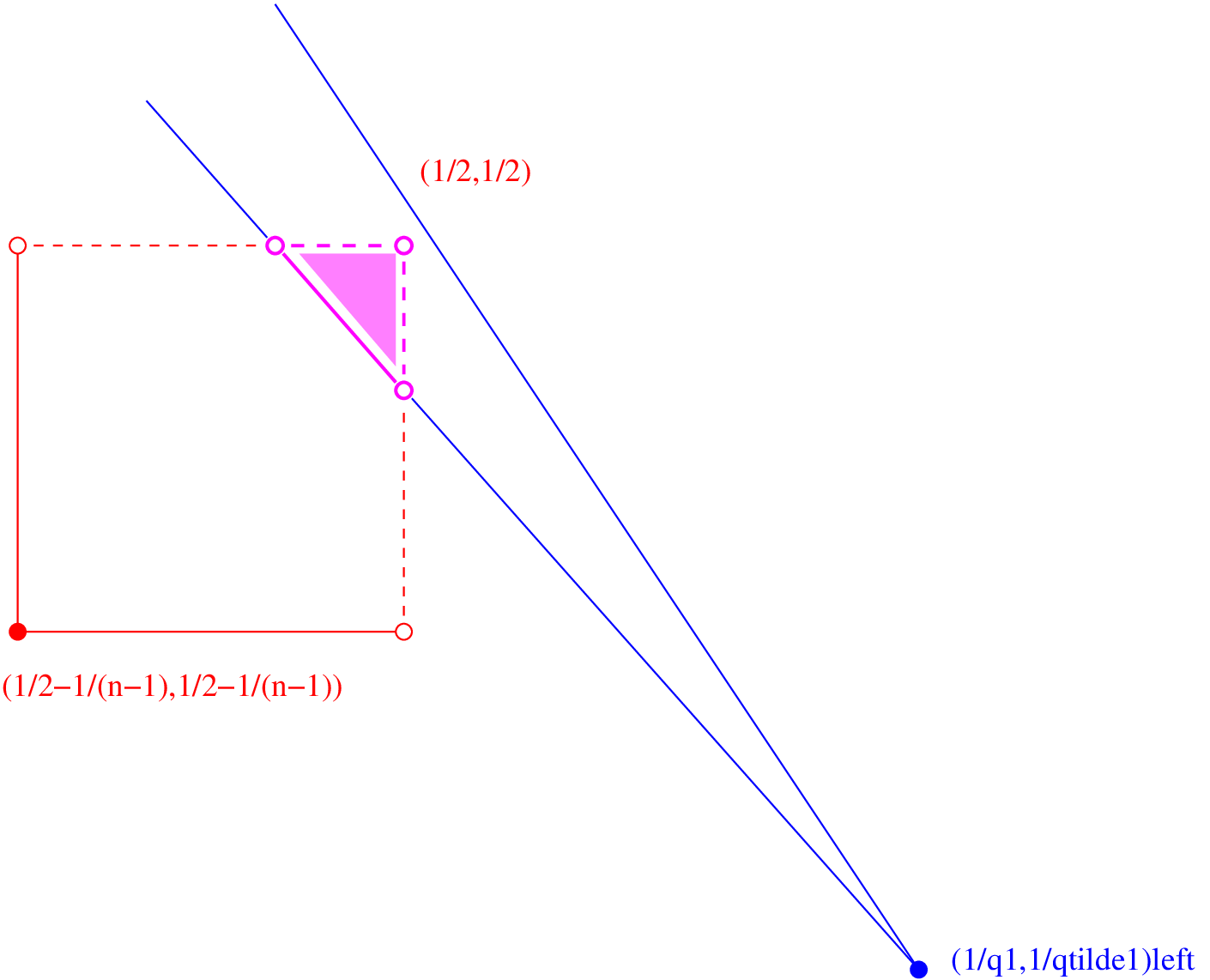}
\caption{Case \,$1\!<\!\gamma\ssb\le\!1\!+\ssb\frac2n$}
\label{Case1.1}
\end{center}
\end{figure}

\noindent$\circ$
\textit{\,Subcase 1.2}\,:
\,$1\!+\ssb\frac2n\!\le\!\gamma\ssb\le\!1\!+\ssb\frac2{n-1}$
\vspace{-2.5mm}

\begin{figure}[ht]
\begin{center}
\psfrag{(1/2,1/2)}[l]{\color{red}$(\frac12,\frac12)$}
\psfrag{(1/2-1/(n-1),1/2-1/(n-1))}[r]
{\color{red}$(\frac12\!-\!\frac1{n-1},\frac12\!-\!\frac1{n-1})$}
\psfrag{(1/q1,1/qtilde1)left}[l]{\color{blue}$(\frac1{q_1},\frac1{\tilde{q}_1})$}
\includegraphics[height=60mm]{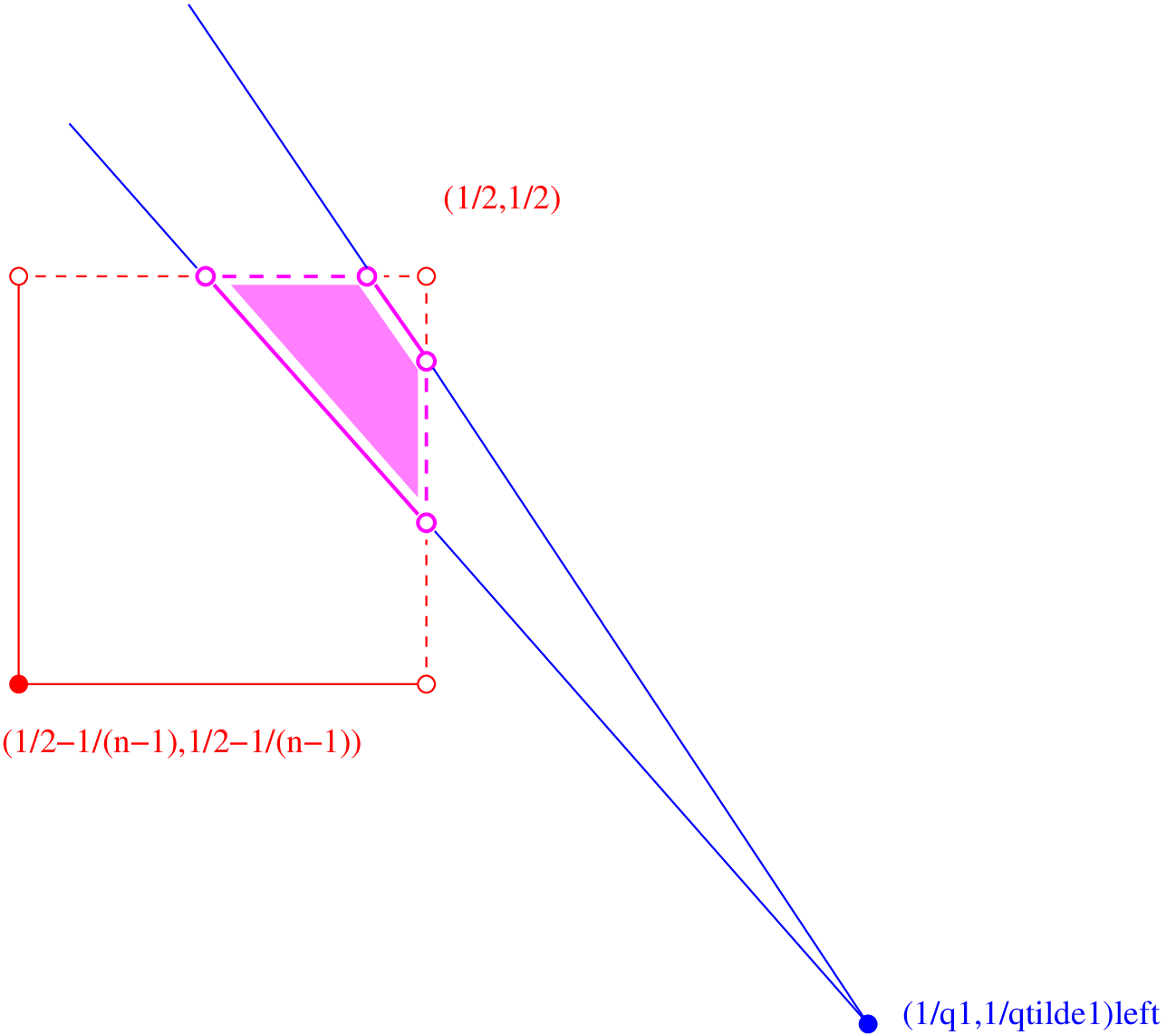}
\caption{Case \,$1\!+\ssb\frac2n\!\le\!\gamma\ssb\le\!1\!+\ssb\frac2{n-1}$}
\label{Case1.2}
\end{center}
\end{figure}

\newpage

\noindent$\circ$
\textit{\,Subcase 1.3}\,:
\,$1\!+\ssb\frac2{n-1}\!\le\!\gamma\ssb\le\!\gammaone$
\vspace{-2.5mm}

\begin{figure}[ht]
\begin{center}
\psfrag{(1/2,1/2)}[l]{\color{red}$(\frac12,\frac12)$}
\psfrag{(1/2-1/(n-1),1/2-1/(n-1))}[r]
{\color{red}$(\frac12\!-\!\frac1{n-1},\frac12\!-\!\frac1{n-1})$}
\psfrag{(1/q1,1/qtilde1)left}[l]{\color{blue}$(\frac1{q_1},\frac1{\tilde{q}_1})$}
\includegraphics[height=47mm]{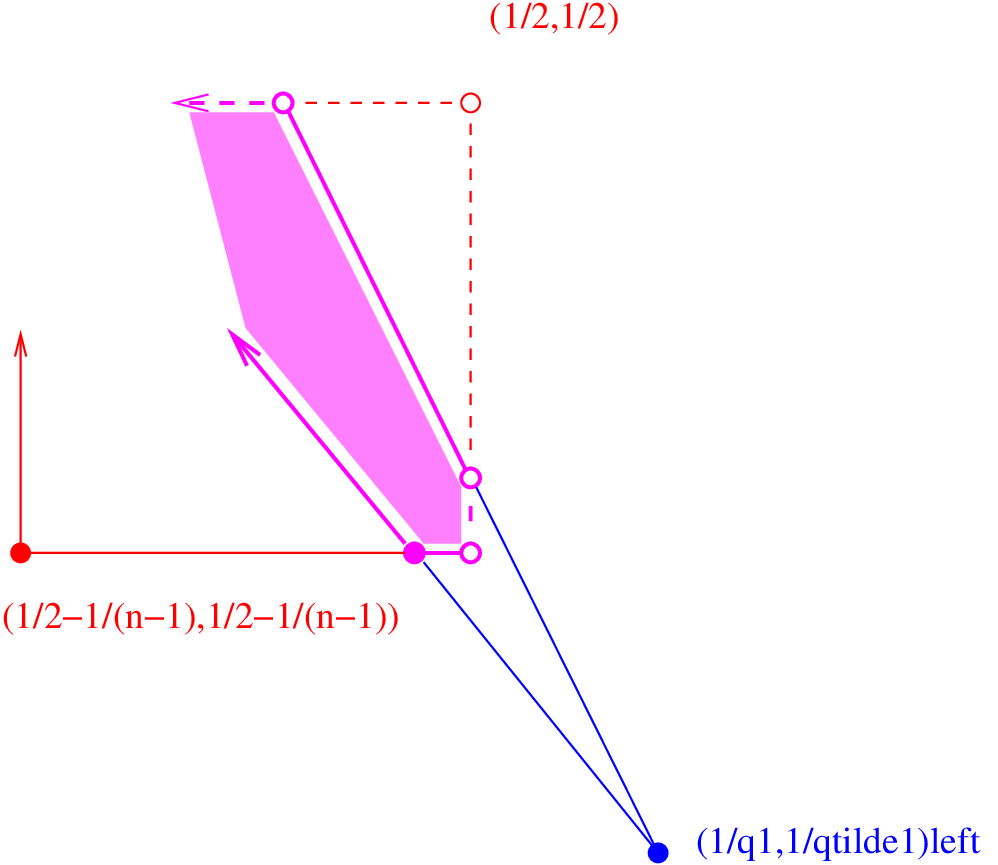}
\caption{Case \,$1\!+\ssb\frac2{n-1}\!\le\!\gamma\ssb\le\!\gammaone$}
\label{Case1.3}
\end{center}
\end{figure}

\noindent$\bullet$
\textit{\,Case 2}\,:
\,$\gammaone\!<\!\gamma\ssb\le\!\gammatwo$

\noindent
The minimal regularity
\ssf$\sigma\ssb=\ssb\sigma_1(\gamma)$
\ssf is reached at the boundary point
\ssf$\bigl(\frac1q,\frac1{\tilde{q}}\bigr)\!
=\!\bigl(\frac{n+5}{4\ssf n}\ssf\frac1{\gamma\ssf-\ssf\frac{n+1}{2\ssf n}}\ssf,$
\linebreak
\vskip-5mm

\noindent
$\frac12\ssb-\ssb\frac1{n-1}\bigr)$.
\vspace{-5mm}

\begin{figure}[ht]
\begin{center}
\psfrag{(1/2,1/2)}[l]{\color{red}$(\frac12,\frac12)$}
\psfrag{(1/2-1/(n-1),1/2-1/(n-1))}[r]
{\color{red}$(\frac12\!-\!\frac1{n-1},\frac12\!-\!\frac1{n-1})$}
\psfrag{(1/q1,1/qtilde1)}[l]{\color{blue}$(\frac1{q_1},\frac1{\tilde{q}_1})$}
\includegraphics[height=45mm]{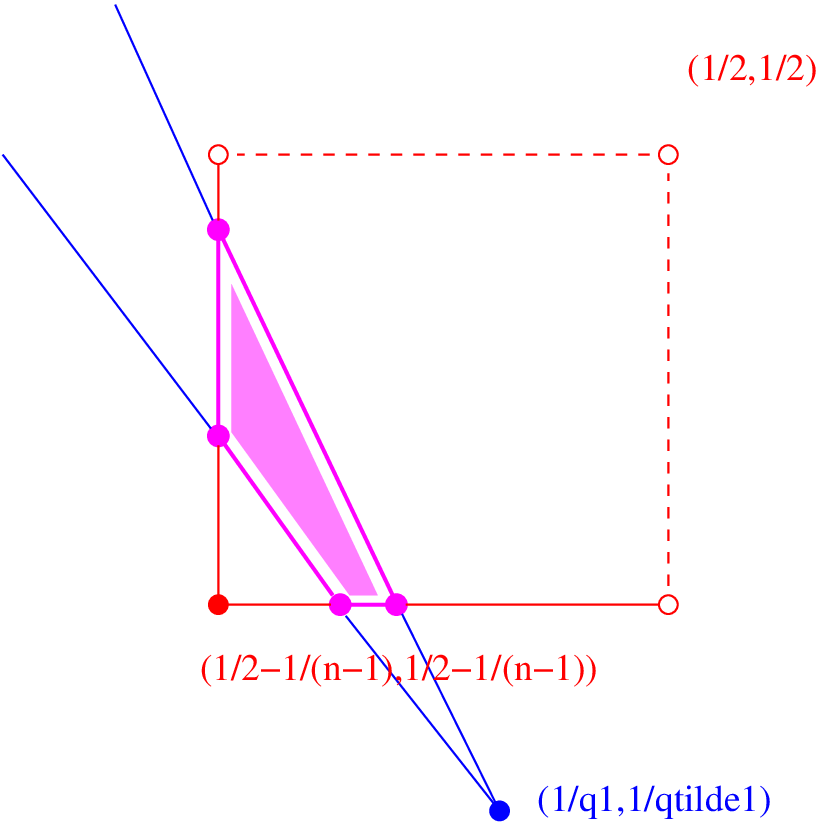}
\caption{Case \,$\gammaone\!<\!\gamma\ssb\le\!\gammatwo$}
\label{Case2}
\end{center}
\end{figure}

\noindent$\bullet$
\textit{\,Case 3}\,:
\,$\gammatwo\!\le\!\gamma\ssb\le\ssb\gammaconf$

\noindent
The minimal regularity
\ssf$\sigma\ssb=\ssb\sigma_2(\gamma)$
\ssf is reached at the vertex
$\bigl(\frac1{q_1},\frac1{\tilde{q}_1})$.
\vspace{-2.5mm}

\begin{figure}[ht]
\begin{center}
\psfrag{(1/2,1/2)}[l]{\color{red}$(\frac12,\frac12)$}
\psfrag{(1/2-1/(n-1),1/2-1/(n-1))}[r]
{\color{red}$(\frac12\!-\!\frac1{n-1},\frac12\!-\!\frac1{n-1})$}
\psfrag{(1/q1,1/qtilde1)}[l]{\color{violet}$(\frac1{q_1},\frac1{\tilde q_1})$}
\includegraphics[height=37mm]{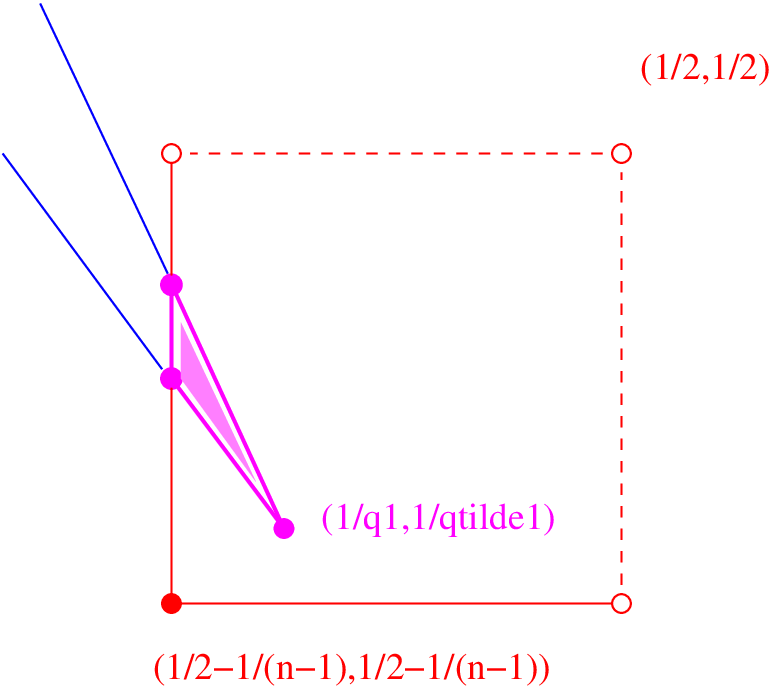}
\caption{Case \,$\gammatwo\!\le\!\gamma\ssb\le\ssb\gammaconf$}
\label{Case3}
\end{center}
\end{figure}

\noindent
In the limit case \ssf$\gamma\ssb=\ssb\gammaconf$\ssf,
notice that all indices \ssf$\frac1{q_1}$, \ssf$\frac1{\tilde{q}_1}$,
\ssf$\frac1{p_1}\ssb=\ssb\frac{n-1}2\ssf\bigl(\frac12\ssb-\ssb\frac1{q_1}\bigr)$,
\ssf$\frac1{\tilde{p}_1}\ssb
=\ssb\frac{n-1}2\ssf\bigl(\frac12\ssb-\ssb\frac1{\tilde{q}_1}\bigr)$
\linebreak
\vskip-4.25mm

\noindent
become equal to the Strichartz index
\ssf$\frac12\ssf\frac{n-1}{n+1}\ssb=\ssf\frac12\ssb-\ssb\frac1{n+1}\ssf$.
\smallskip

This concludes the proof of Theorem \ref{GWPND}
for $1\!<\!\gamma\ssb\le\!\gammaconf$
\ssf and \ssf$n\ssb\ge\ssb6$\ssf.
\smallskip

\noindent$\blacktriangleright$
\,Assume that \ssf$n\!=\ssb4$ \ssf or \ssf$5$\ssf.
\smallskip

Let us adapt the proof above.
If \ssf$\gamma\ssb\ge\ssb2$\ssf,
(\ref{conditions2}.e) must be checked
and (\ref{conditions2}.a), (\ref{conditions2}.c)
reduce again to \eqref{ac}, but this time in the slightly larger square
\begin{equation}\label{SQUARE2}\textstyle
R\ssf=\bigl[\ssf\frac12\!-\!\frac1{n-1},\frac12\ssf\bigr)\ssb
\times\ssb\bigl[\ssf\frac12\!-\!\frac1{n-1},\frac12\ssf\bigr]\ssf.
\end{equation}
See Figure \ref{FIGURE2}.

\begin{figure}[ht]
\begin{center}
\psfrag{1/p}[c]{$\frac1p$}
\psfrag{1/ptilde}[c]{$\frac1{\tilde{p}}$}
\psfrag{1/q}[c]{$\frac1q$}
\psfrag{1/qtilde}[c]{$\frac1{\tilde{q}}$}
\psfrag{1/2}[c]{$\frac12$}
\psfrag{1/2-1/(n-1)center}[c]{$\frac12\!-\!\frac1{n-1}$}
\psfrag{1/2-1/(n-1)left}[l]{$\frac12\!-\!\frac1{n-1}$}
\psfrag{1/2-1/(n-1)right}[r]{$\frac12\!-\!\frac1{n-1}$}
\psfrag{1/gamma}[c]{$\frac1\gamma$}
\psfrag{1/(2gamma)}[c]{$\frac1{2\ssf\gamma}$}
%\psfrag{1-gamma/2left}[l]{$1\!-\!\frac\gamma2$}
%\psfrag{1-gamma/2right}[r]{$1\!-\!\frac\gamma2$}
\psfrag{1/2-1/gamma1/(n-1)}[c]{$\frac12\!-\!\frac1\gamma\ssf\frac1{n-1}$}
\psfrag{1/2-2/gamma1/(n-1)}[c]{$\frac12\!-\!\frac2\gamma\ssf\frac1{n-1}$}
%\psfrag{1/2-(2-gamma)/(n-1)}[c]{$\frac12\!-\!\frac{2-\gamma}{n-1}$}
%\psfrag{1-gamma/2}[c]{$1\!-\!\frac\gamma2$}
\psfrag{gamma/p+1/ptilde=1}[c]{$\frac\gamma p\ssb+\ssb\frac1{\tilde{p}}=1$}
\psfrag{gamma/q+1/qtilde=(gamma+1)/2-2/(n-1)}[c]
{$\frac\gamma q\ssb+\ssb\frac1{\tilde{q}}
=\frac{\gamma+1}2\ssb-\ssb\frac2{n-1}$}
\includegraphics[height=110mm]{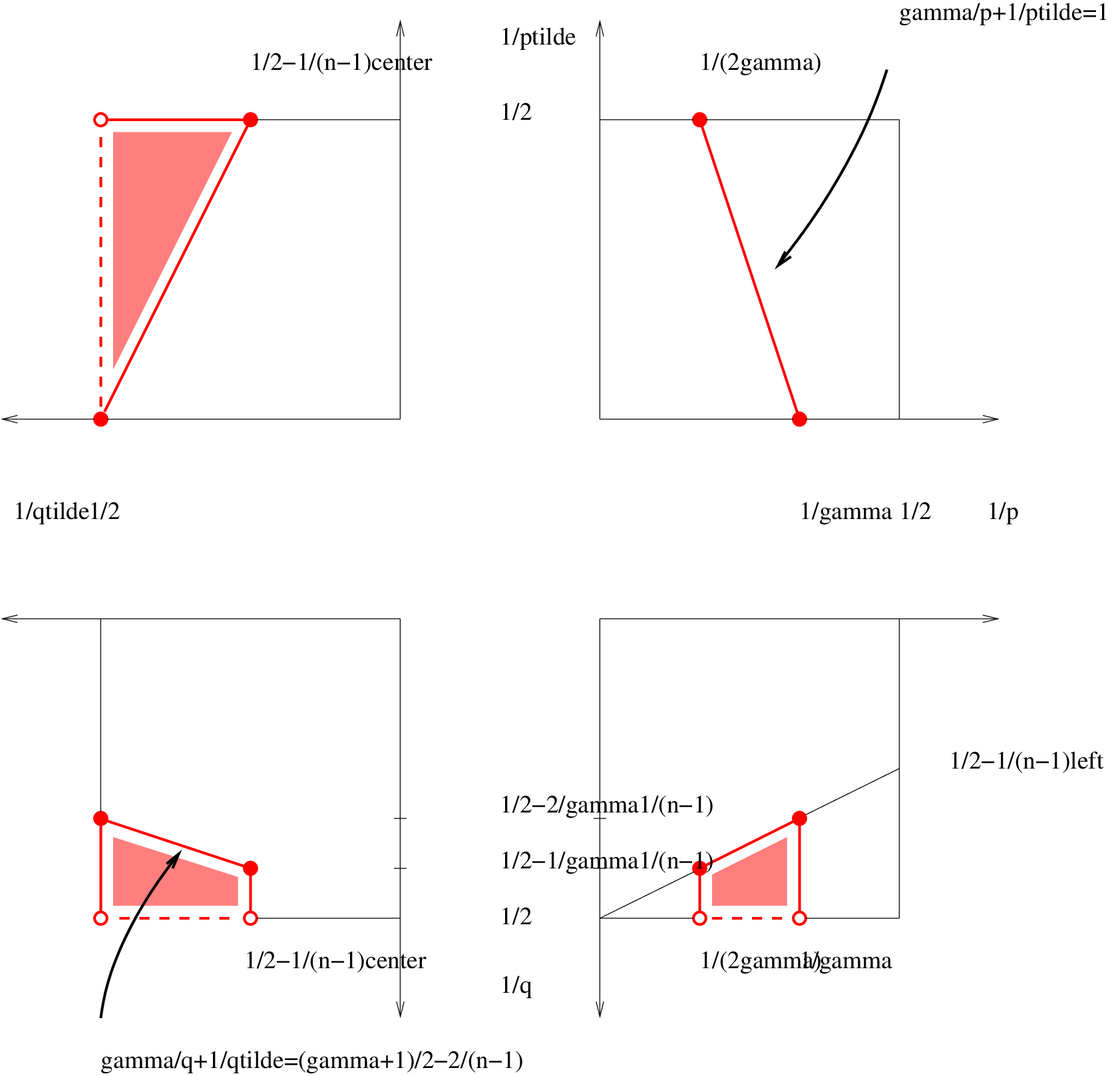}
\caption{Case \,$\gamma\ssb\ge\ssb2$}
\label{FIGURE2}
\end{center}
\end{figure}

\noindent
Thus \eqref{conditions2} reduce to
\begin{equation*}\begin{cases}
\,\text{(\ref{conditions2}.d.i), (\ref{conditions2}.d.ii), \eqref{SQUARE2}}
&\text{if \,$1\!<\!\gamma\ssb<\!2$\ssf,}\\
\,\text{(\ref{conditions2}.d.i), (\ref{conditions2}.d.ii),
(\ref{conditions2}.e), \eqref{SQUARE2}}
&\text{if \,$2\ssb\le\!\gamma\ssb\le\!\gammaconf$\,.}\\
\end{cases}\end{equation*}
The case--by--case study of the intersection \ssf$S\ssb\cap\ssb R$
\ssf is carried out as above and yield the same results.
The only difference lies in the fact that
the sector \ssf$S$ exists the square $R$ \ssf
through the top edge instead of the left edge
(see Figures \ref{Case2bis}, \ref{Case3bis}, \ref{Case3ter} below).
Notice that (\ref{conditions2}.e) is satisfied,
as $q_1\!>\ssb\gamma$
when $2\ssb\le\!\gamma\ssb\le\!\gammaconf$\ssf.
\medskip

\noindent$\bullet$
\textit{\,Case 2}\,:
\,$\gammaone\!<\!\gamma\ssb\le\!\gammatwo$
\vspace{-2.5mm}

\begin{figure}[ht]
\begin{center}
\psfrag{(1/2,1/2)}[l]{\color{red}$(\frac12,\frac12)$}
\psfrag{(1/2-1/(n-1),1/2-1/(n-1))}[r]
{\color{red}$(\frac12\!-\!\frac1{n-1},\frac12\!-\!\frac1{n-1})$}
\psfrag{(1/q1,1/qtilde1)}[l]{\color{blue}$(\frac1{q_1},\frac1{\tilde{q}_1})$}
\includegraphics[height=43mm]{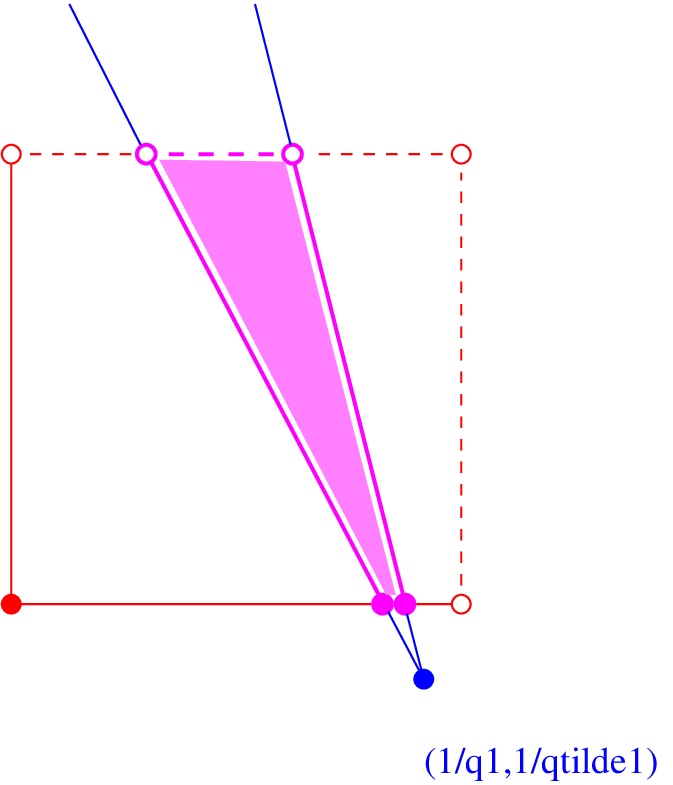}
\caption{Case \,$\gammaone\!<\!\gamma\ssb\le\!\gammatwo$}
\label{Case2bis}
\end{center}
\end{figure}

\noindent$\bullet$
\textit{\,Case 3}\,:
\,$\gammatwo\!\le\!\gamma\ssb<\!\gammaconf$
\vspace{-2.5mm}

\begin{figure}[ht]
\begin{center}
\psfrag{(1/2,1/2)}[l]{\color{red}$(\frac12,\frac12)$}
\psfrag{(1/2-1/(n-1),1/2-1/(n-1))}[r]
{\color{red}$(\frac12\!-\!\frac1{n-1},\frac12\!-\!\frac1{n-1})$}
\psfrag{(1/q1,1/qtilde1)}[l]{\color{violet}$(\frac1{q_1},\frac1{\tilde q_1})$}
\includegraphics[height=34mm]{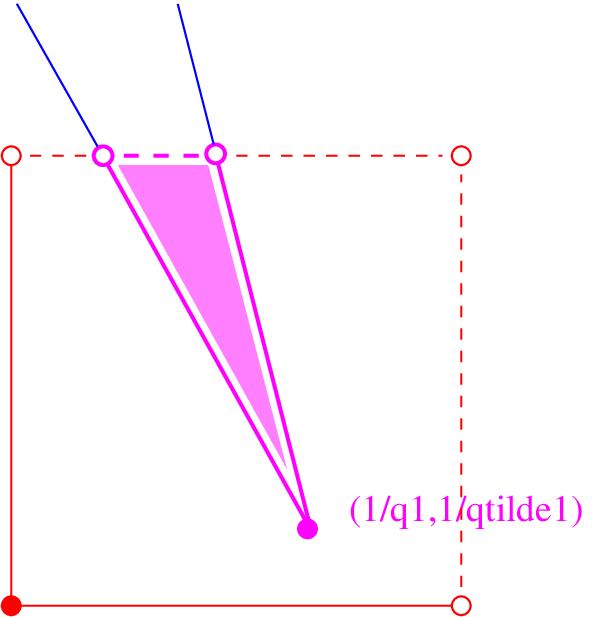}
\caption{Subcase \,$\gammatwo\!\le\!\gamma\ssb<\ssb2$}
\label{Case3bis}
\end{center}
\end{figure}
\vskip-5mm

\begin{figure}[ht]
\begin{center}
\psfrag{(1/2,1/2)}[l]{\color{red}$(\frac12,\frac12)$}
\psfrag{(1/2-1/(n-1),1/2-1/(n-1))}[r]
{\color{red}$(\frac12\!-\!\frac1{n-1},\frac12\!-\!\frac1{n-1})$}
\psfrag{(1/q1,1/qtilde1)}[l]{\color{violet}$(\frac1{q_1},\frac1{\tilde q_1})$}
\includegraphics[height=34mm]{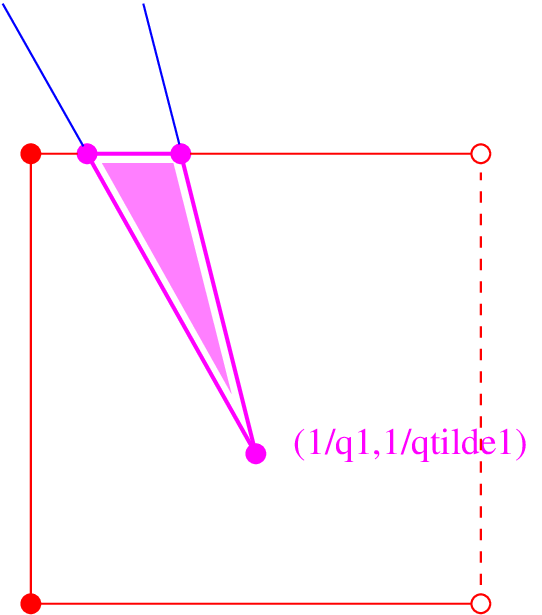}
\caption{Subcase \,$2\ssb\le\!\gamma\ssb\le\ssb\gammaconf$}
\label{Case3ter}
\end{center}
\end{figure}

\noindent
This concludes the proof of Theorem \ref{GWPND}
for $1\!<\!\gamma\ssb\le\!\gammaconf$
\ssf and \ssf$n\ssb=\ssb4\ssf,5$\ssf.
\smallskip

\noindent$\blacktriangleright$
\,Assume that \ssf$n\ssb=\ssb3$\ssf.
\smallskip

The proof works the same, except that the square becomes
\begin{equation}
R\,=\ssf\begin{cases}
\,\bigl(\ssf0,\frac12\ssf\bigr)\ssb\times\ssb\bigl(\ssf0,\frac12\ssf\bigr)
&\text{if \,}1\!<\!\gamma\ssb<\ssb2\ssf,\\
\,\bigl(\ssf0,\frac12\ssf\bigr)\ssb\times\ssb\bigl(\ssf0,\frac12\ssf\bigr]
&\text{if \,}2\ssb\le\!\gamma\ssb\le\!\gammaconf\ssf.\\
\end{cases}\end{equation}
and that $\bigl(\frac1{q_1},\frac1{\tilde{q}_1}\bigr)$ enters the square $R$
\ssf through the vertex $\bigl(\frac12,0\bigr)$ instead of the bottom edge.
This happens when $\gamma\ssb=\ssb2$
\ssf and in this case (\ref{conditions2}.e) is satisfied.
\vspace{-2.5mm}

\begin{figure}[ht]
\begin{center}
\psfrag{(1/2,1/2)}[l]{\color{red}$(\frac12,\frac12)$}
\psfrag{(1/2-1/(n-1),1/2-1/(n-1))}[r]
{\color{red}$(\frac12\!-\!\frac1{n-1},\frac12\!-\!\frac1{n-1})$}
\psfrag{(1/q1,1/qtilde1)}[l]{\color{violet}$(\frac1{q_1},\frac1{\tilde q_1})$}
\includegraphics[height=33mm]{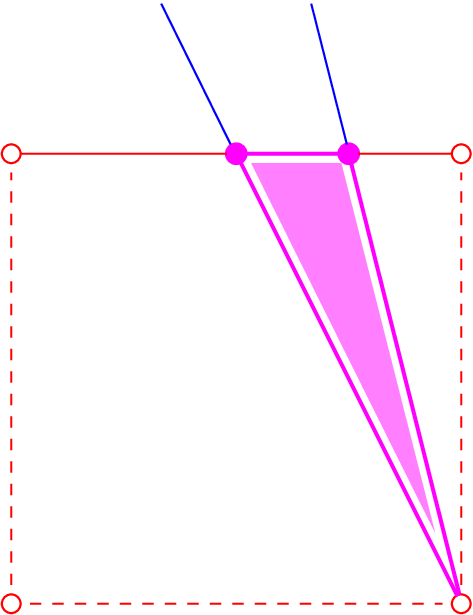}
\caption{Case \,$\gamma\ssb=\ssb2$}
\label{Case2-3}
\end{center}
\end{figure}

\noindent
It is further satisfied
when $2\ssb<\!\gamma\ssb\le\!\gammaconf$\ssf,
as $q_1\!>\ssb\gamma$\ssf.
\smallskip

This concludes the proof of Theorem \ref{GWPND}
for $1\!<\!\gamma\ssb\le\!\gammaconf$\ssf.
\end{proof}

\begin{proof}[Proof of Theorem \ref{GWPND}
for \,$\gammaconf\ssb\le\!\gamma\ssb\le\!\gammafour$\ssf]
We resume the fixed point method above,
using Corollary \ref{GeneralizedStrichartzND}
instead of Theorem \ref{StrichartzND},
and obtain in this way the set of conditions\,:
\begin{equation}\label{conditions3}\begin{cases}
\,\text{(a)}&\hspace{-2mm}
\text{$2\ssb\le\ssb p\ssb\le\!\infty$
\,and \,$2\ssb\le\ssb q\ssb<\!\infty$
\,satisfy \ssf$\frac1p\!
\le\!\frac{n-1}2\ssf\bigl(\frac12\!-\!\frac1q\bigr)$\,;}\\
\,\text{($\tilde{\text{a}}$)}&\hspace{-2mm}
\text{$2\ssb\le\ssb\tilde{p}\ssb\le\!\infty$
\,and \,$2\ssb\le\ssb\tilde{q}<\!\infty$
\,satisfy \ssf$\frac1{\tilde{p}}\!
\le\!\frac{n-1}2\ssf\bigl(\frac12\!-\!\frac1{\tilde{q}}\bigr)$\,;}\\
\,\text{(b)}&\hspace{-2mm}
\sigma\!\ge\ssb n\ssf\big(\frac12\!-\!\frac1q\bigr)\hspace{-1mm}-\!\frac1p\ssf,
\;\tilde{\sigma}\!\ge\ssb
n\ssf\big(\frac12\!-\!\frac1{\tilde{q}}\bigr)
\hspace{-1mm}-\!\frac1{\tilde{p}}\ssf,
\;\sigma\ssb+\ssb\tilde{\sigma}\ssb\le\ssb1\,;\\
\,\text{(c)}&\hspace{-2mm}
\frac\gamma p\ssb+\ssb\frac1{\tilde{p}}\ssb=\ssb1\,;\\
\,\text{(d)}&\hspace{-2mm}
\text{$1\ssb\le\ssb\frac\gamma q\ssb+\ssb\frac1{\tilde{q}}\ssb
\le\ssb1\ssb+\ssb\frac{1-\sigma-\tilde{\sigma}}n$\,;}\\
\,\text{(e)}&\hspace{-2mm}
\text{$q\ssb>\ssb\gamma$\ssf.}\\
\end{cases}\end{equation}
We may assume that
\begin{equation*}\textstyle
\sigma=n\,\bigl(\frac12\ssb-\ssb\frac1q\bigr)\ssb-\frac1p
\quad\text{and}\quad\tilde{\sigma}=
n\,\bigl(\frac12\ssb-\ssb\frac1{\tilde{q}}\bigr)\ssb-\frac1{\tilde{p}}\,.
\end{equation*}
With this choice,
the conditions
\begin{equation*}\textstyle
\sigma\ssb+\ssb\tilde{\sigma}\ssb\le\ssb1
\quad\text{and}\quad
\frac\gamma q\ssb+\ssb\frac1{\tilde{q}}\ssb
\le\ssb1\ssb+\ssb\frac{1-\sigma-\tilde{\sigma}}n
\end{equation*}
become
\begin{equation}\label{firstcondition}\textstyle
\frac1p\ssb+\ssb\frac1{\tilde{p}}\ssb+\ssb1\ssb
\ge\ssb n\ssf\bigl(1\!-\ssb\frac1q\ssb-\ssb\frac1{\tilde{q}}\ssf\bigr)
\end{equation}
and
\begin{equation}\label{secondcondition}\textstyle
\frac1p\ssb+\ssb\frac1{\tilde{p}}\ssb+\ssb1\ssb
\ge\ssb(\gamma\!-\!1)\ssf\frac nq\,.
\end{equation}
Notice moreover that \eqref{firstcondition} follows from \eqref{secondcondition},
combined with \ssf$\frac\gamma q\ssb+\ssb\frac1{\tilde{q}}\ssb\ge\ssb1$\ssf,
and that \eqref{secondcondition} can be rewritten as follows,
using (\ref{conditions3}.c)\,:
\begin{equation*}\textstyle
\frac1p\ssb+\ssb\frac nq\ssb\le\ssb\frac2{\gamma-1}\,.
\end{equation*}
Thus \eqref{conditions3} reduces to the set of conditions
\begin{equation}\label{conditions4}\begin{cases}
\,\text{(a)}&\hspace{-2mm}
\text{$2\ssb\le\ssb p\ssb\le\!\infty$
\,and \,$2\ssb\le\ssb q\ssb<\!\infty$
\,satisfy \ssf$\frac1p\ssb
\le\ssb\frac{n-1}2\ssf\bigl(\frac12\!-\!\frac1q\bigr)$\,;}\\
\,\text{($\tilde{\text{a}}$)}&\hspace{-2mm}
\text{$2\ssb\le\ssb\tilde{p}\ssb\le\!\infty$
\,and \,$2\ssb\le\ssb\tilde{q}<\!\infty$
\,satisfy \ssf$\frac1{\tilde{p}}\ssb
\le\ssb\frac{n-1}2\ssf\bigl(\frac12\!-\!\frac1{\tilde{q}}\bigr)$\,;}\\
\,\text{(c)}&\hspace{-2mm}
\frac\gamma p\ssb+\ssb\frac1{\tilde{p}}\ssb=\ssb1\,;\\
\,\text{(d.i)}&\hspace{-2mm}
\frac\gamma q\ssb+\ssb\frac1{\tilde{q}}\ssb\ge\ssb1\,;\\
\,\text{(d.ii)}&\hspace{-2mm}
\frac1p\ssb+\ssb\frac nq\ssb\le\ssb\frac2{\gamma-1}\,;\\
\,\text{(e)}&\hspace{-2mm}
q\ssb>\ssb\gamma\ssf.\\
\end{cases}\end{equation}
Among these conditions, consider first (\ref{conditions4}.a) and (\ref{conditions4}.d.ii).
In the plane with coordinates $\bigl(\frac1p,\frac1q\bigr)$, the two lines
\begin{equation}\label{2lines}\begin{cases}
\,\text{(a)}&\hspace{-2mm}
\frac1p\ssb+\ssb\frac{n-1}2\ssf\frac1q\ssb=\ssb\frac{n-1}4\\
\,\text{(d.ii)}&\hspace{-2mm}
\frac1p\ssb+\ssb\frac nq\ssb=\ssb\frac2{\gamma-1}
\end{cases}\end{equation}
meet at the point $\bigl(\frac1{\ptwo},\frac1{\qtwo}\bigr)$ given by
\begin{equation}\label{p2q2}\begin{cases}
\,\frac1{\ptwo}\ssb=
\frac{n-1}{n+1}\ssf\bigl(\frac n2\ssb-\ssb\frac2{\gamma-1}\bigr)\ssf,\\
\,\frac1{\qtwo}\ssb=
\frac1{n+1}\ssf\bigl(\frac4{\gamma-1}\ssb-\ssb\frac{n-1}2\bigr)\ssf.
\end{cases}\end{equation}
As \ssf$\gamma$ \ssf varies
between \ssf$\gammaconf$ \ssf and \ssf$\gammathree$\ssf,
this point moves on the line (\ref{2lines}.a)
between the Strichartz point
$\bigl(\frac12\ssb-\ssb\frac1{n+1},\frac12\ssb-\ssb\frac1{n+1}\bigr)$
and the Keel--Tao endpoint
$\bigl(\frac12,\frac12\ssb-\ssb\frac1{n-1}\bigr)$,
where it exists the square
\ssf$[\ssf0,\frac12\ssf]\times\ssb(0,\frac12\ssf]$\ssf.
Thus (\ref{conditions4}.a) and (\ref{conditions4}.d.ii)
determine the following regions\,:

\begin{figure}[ht]
\begin{center}
\psfrag{0}[c]{\color{red}$0$}
\psfrag{1/2}[c]{$\frac12$}
\psfrag{1/2red}[c]{\color{red}$\frac12$}
\psfrag{1/p}[c]{$\frac1p$}
\psfrag{1/q}[c]{$\frac1q$}
\psfrag{(1/p2,1/q2)}[c]{\color{red}$\bigl(\frac1{\ptwo},\frac1{\qtwo}\bigr)$}
\psfrag{1/2-1/(n-1)}[l]{\color{red}$\frac12\!-\!\frac1{n-1}$}
\psfrag{linea}[l]{(\ref{2lines}.a)}
\psfrag{linedii}[l]{(\ref{2lines}.d.ii)}
\psfrag{2/n(gamma-1)}[c]{\color{red}$\frac2{n\ssf(\gamma-1)}$}
\includegraphics[height=50mm]{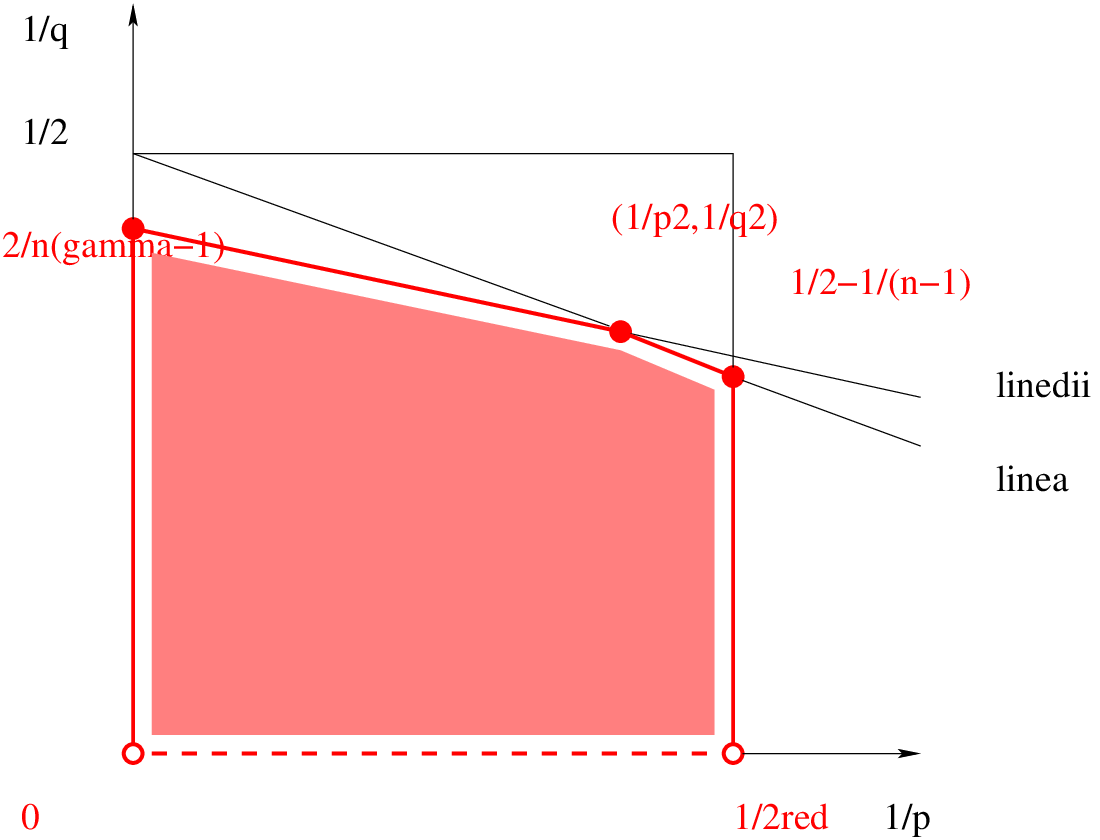}
\caption{\textit{Case 4}\,: \,$\gammaconf\ssb\le\gamma\ssb\le\ssb\gammathree$}
\label{polygon1}
\end{center}
\end{figure}

\newpage

\begin{figure}[ht]
\begin{center}
\psfrag{0}[c]{\color{red}$0$}
\psfrag{1/2}[c]{$\frac12$}
\psfrag{1/2red}[c]{\color{red}$\frac12$}
\psfrag{1/p}[c]{$\frac1p$}
\psfrag{1/q}[c]{$\frac1q$}
\psfrag{linea}[l]{(\ref{2lines}.a)}
\psfrag{linedii}[l]{(\ref{2lines}.d.ii)}
\psfrag{2/n(gamma-1)}[c]{\color{red}$\frac2{n\ssf(\gamma-1)}$}
\psfrag{1/n(2/(gamma-1)-1/2)}[l]{\color{red}$\frac1n\bigl(\frac2{\gamma-1}\!-\!\frac12\bigr)$}
\includegraphics[height=50mm]{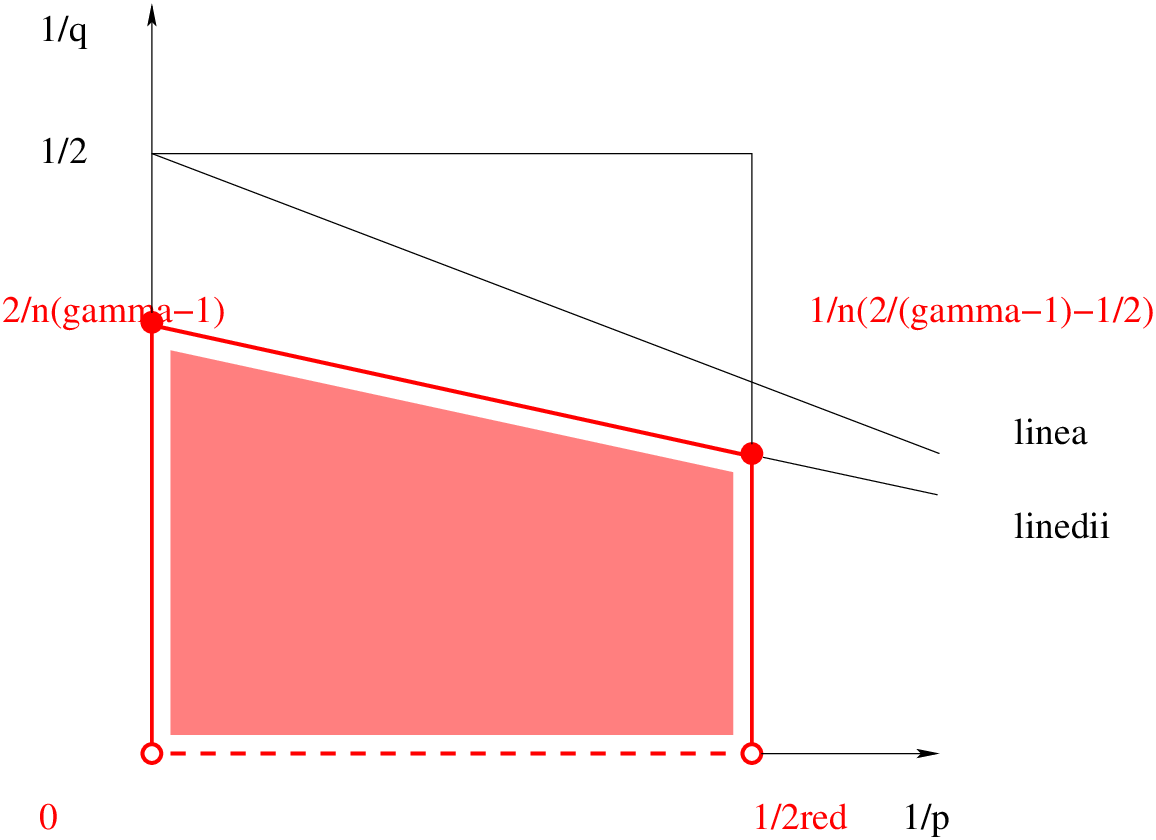}
\caption{\textit{Case 5}\,: \,$\gammathree\ssb\le\gamma\ssb\le\ssb\gammafour$}
\label{polygon2}
\end{center}
\end{figure}

\noindent
For later use, notice that the minimal regularity
\begin{equation}\label{sigma}\textstyle
\sigma\ssb
=\ssb\frac n2\bigl(\frac12\!-\!\frac1q\bigr)\!-\ssb\frac1p\ssb
\ge\ssb\sigma_3(\gamma)
\end{equation}
is reached on the boundary line (\ref{2lines}.d.ii)
and that
\begin{equation}\label{2gamma}
\ptwo\ssb<2\ssf\gamma\,.
\end{equation}
This inequality holds indeed when \ssf$\gamma\ssb=\ssb\gammaconf$
\ssf and it remains true as \ssf$\gamma$ \ssf increases while \ssf$\ptwo$ decreases.
\smallskip

Let us next discuss all conditions \eqref{conditions4},
first in high dimensions and next in low dimensions.
\smallskip

\noindent$\blacktriangleright$
\,Assume that \ssf$n\ssb\ge\ssb6$\ssf.
\smallskip

Firstly notice that (\ref{conditions4}.e) is trivially satisfied in this case.
On one hand, we have indeed \ssf$\gamma\ssb\le\ssb\gammafour\ssb\le2$\ssf.
On the other hand, it follows from (\ref{conditions4}.d.ii) that
\vspace{.5mm}

\centerline{$
\frac1q\ssb\le\ssb\frac2{n\,(\ssf\gamma\ssf-\ssf1\ssf)}
\le\frac2{n\,(\ssf\gammaconf\,-\ssf1\ssf)}
=\frac12\ssf\bigl(1\ssb-\ssb\frac1n\bigr)<\frac12\,.
$}

\noindent
Hence \ssf$\gamma\ssb\le\ssb2\ssb<\ssb q$\ssf.
\smallskip

Secondly we claim that
(\ref{conditions4}.a), (\ref{conditions4}.$\tilde{\text{a}}$),
(\ref{conditions4}.c), (\ref{conditions4}.d.ii)
reduce to the conditions
\begin{equation}\label{aacd}\begin{cases}
\,\text{(a)}&\hspace{-2mm}
\frac\gamma q\ssb+\ssb\frac1{\tilde{q}}\ssb
\le\ssb\frac{\gamma+1}2\ssb-\ssb\frac2{n-1}\\
\,\text{(d.ii)}&\hspace{-2mm}
\frac\gamma q\ssb+\ssb\frac{n-1}{2\,n}\ssf\frac1q\ssb
\le\ssb\frac{n+3}{4\,n}\ssb+\ssb\frac2n\ssf\frac1{\gamma-1}
\end{cases}\end{equation}
in the rectangle
\begin{equation}\label{SQUARE3}\textstyle
R=\bigl(\ssf0\ssf,
\frac1n\ssf\bigl(\frac2{\gamma-1}\!-\!\frac1{2\ssf\gamma}\bigr)\ssf\bigr]
\ssb\times\ssb\bigl(\ssf0\ssf,\frac12\!-\!\frac{2-\gamma}{n-1}\ssf\bigr]\,.
\end{equation}
Actually they even reduce to the single condition (\ref{aacd}.d.ii)
if \ssf$\gamma\ssb\ge\ssb\gammathree$\ssf.
All these claims are obtained again
by examining the following four quadrant figures.

\newpage

\begin{figure}[ht]
\begin{center}
\psfrag{1/p}[c]{$\frac1p$}
\psfrag{1/ptilde}[c]{$\frac1{\tilde{p}}$}
\psfrag{1/q}[c]{$\frac1q$}
\psfrag{1/qtilde}[c]{$\frac1{\tilde{q}}$}
\psfrag{(1/p2,1/q2)}[c]{\color{red}$\bigl(\frac1{\ptwo},\frac1{\qtwo}\bigr)$}
\psfrag{(1/tildep2,1/tildeq2)}[c]{\color{red}$\bigl(\frac1{\tilde{p}_2},\frac1{\tilde{q}_2}\bigr)$}
\psfrag{(1/p2,1/tildep2)}[c]{\color{red}$\bigl(\frac1{\ptwo},\frac1{\tilde{p}_2}\bigr)$}
\psfrag{(1/q2,1/tildeq2)}[c]{\color{red}$\bigl(\frac1{\qtwo},\frac1{\tilde{q}_2}\bigr)$}
\psfrag{1/2}[c]{$\frac12$}
\psfrag{1/2right}[r]{$\frac12$}
\psfrag{1/2-1/(n-1)}[c]{$\frac12\!-\!\frac1{n-1}$}
\psfrag{1/2gamma}[c]{$\frac1{2\ssf\gamma}$}
\psfrag{1-gamma/2}[c]{$1\!-\!\frac\gamma2$}
\psfrag{1/n(2/(gamma-1)-1/2gamma)}[c]{\color{red}$\frac1n\ssf\bigl(\frac2{\gamma-1}\!-\!\frac1{2\ssf\gamma}\bigr)$}
\psfrag{1/2-(2-gamma)/(n-1)}[l]{\color{red}$\frac12\ssb-\ssb\frac{2-\gamma}{n-1}$}
\psfrag{conditions4.c}[c]{(\ref{conditions4}.c)}
\psfrag{2lines.a}[c]{(\ref{2lines}.a)}
\psfrag{2lines.d.ii}[c]{(\ref{2lines}.d.ii)}
\psfrag{edges1.a}[c]{(\ref{edges1}.a)}
\psfrag{edges1.d.ii}[c]{(\ref{edges1}.d.ii)}
\includegraphics[height=110mm]{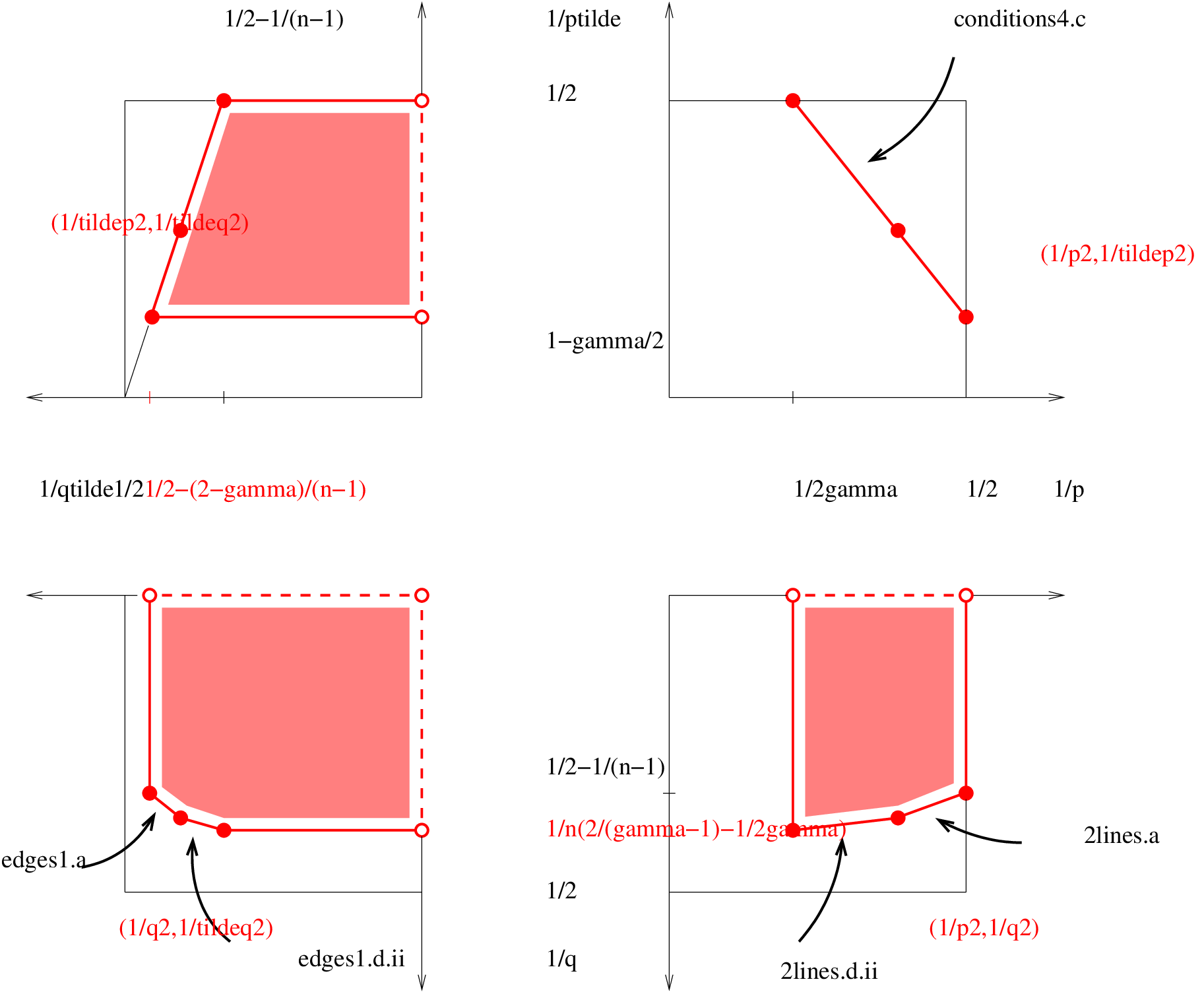}
\caption{\textit{Case 4}\,:
\,$\gammaconf\ssb\le\ssb\gamma\ssb\le\ssb\gammathree$}
\label{FIGURE3}
\end{center}
\end{figure}

\newpage

\begin{figure}[ht]
\begin{center}
\psfrag{1/p}[c]{$\frac1p$}
\psfrag{1/ptilde}[c]{$\frac1{\tilde{p}}$}
\psfrag{1/q}[c]{$\frac1q$}
\psfrag{1/qtilde}[c]{$\frac1{\tilde{q}}$}
\psfrag{1/2}[c]{$\frac{1\vphantom{\gamma}}2$}
\psfrag{1/2-1/(n-1)}[c]{$\frac12\!-\!\frac1{n-1}$}
\psfrag{1/2gamma}[c]{$\frac1{2\ssf\gamma}$}
\psfrag{1-gamma/2}[c]{$1\!-\!\frac\gamma2$}
\psfrag{1/n(2/(gamma-1)-1/2)}[c]
{$\frac1n\ssf\bigl(\frac2{\gamma-1}\!-\!\frac12\bigr)$}
\psfrag{1/n(2/(gamma-1)-1/2gamma)}[c]
{\color{red}$\frac1n\ssf\bigl(\frac2{\gamma-1}\!-\!\frac1{2\ssf\gamma}\bigr)$}
\psfrag{1/2-(2-gamma)/(n-1)}[l]{\color{red}$\frac12\ssb-\ssb\frac{2-\gamma}{n-1}$}
\psfrag{conditions4.c}[c]{(\ref{conditions4}.c)}
\psfrag{2lines.d.ii}[c]{(\ref{2lines}.d.ii)}
\psfrag{edges1.d.ii}[c]{(\ref{edges1}.d.ii)}
\includegraphics[height=110mm]{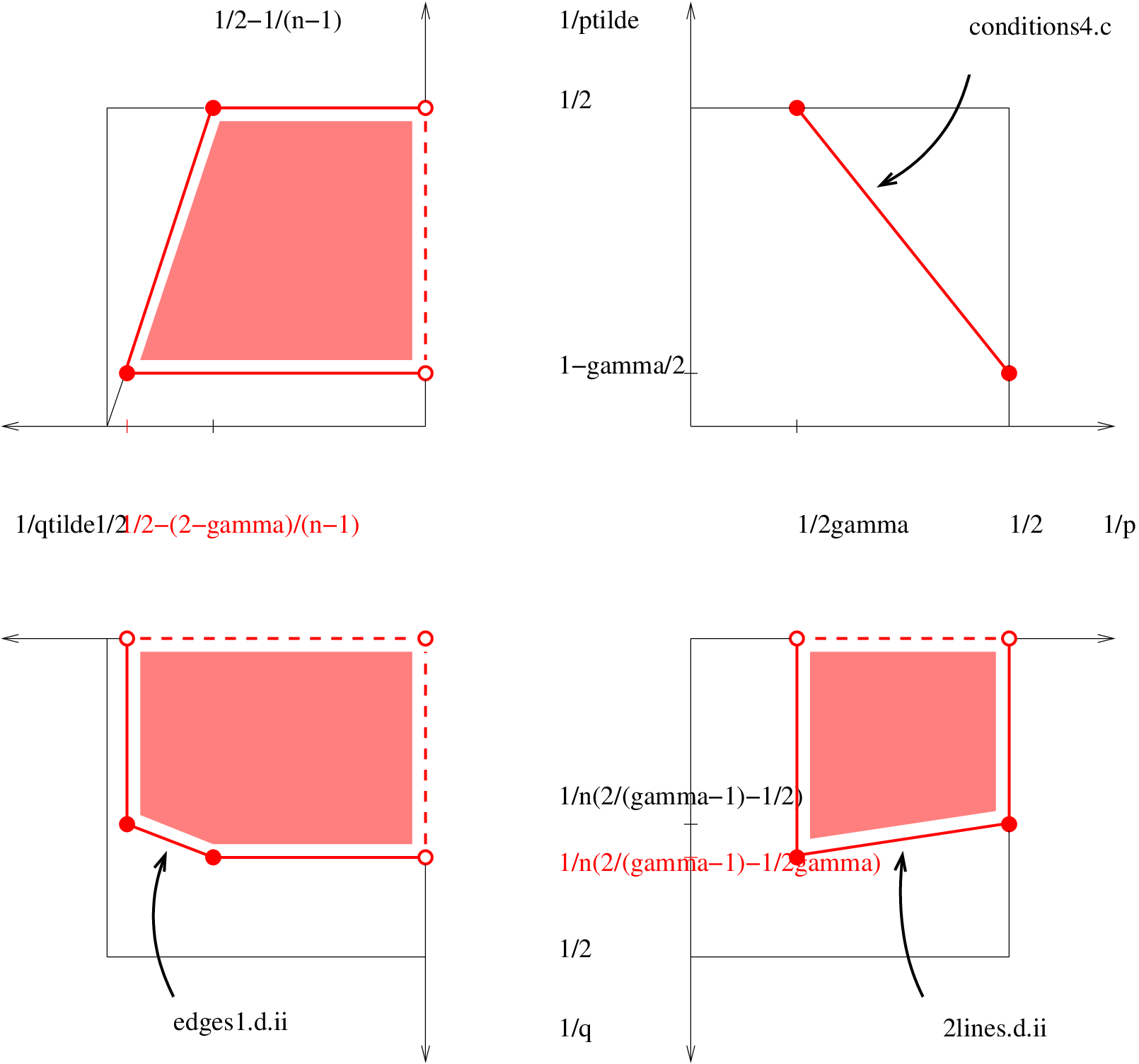}
\caption{\textit{Case 5}\,:
\,$\gammathree\ssb\le\ssb\gamma\ssb\le\ssb\gammafour$}
\label{FIGURE4}
\end{center}
\end{figure}

\newpage

\noindent
Thirdly, in the plane with coordinates $\bigl(\frac1q,\frac1{\tilde{q}}\bigr)$,
the conditions (\ref{conditions4}.d.i), (\ref{aacd}.a), (\ref{aacd}.d.ii)
define the convex region \ssf$C$ in Figure \ref{Convex} with edges
\begin{equation}\label{edges1}\begin{cases}
\,\text{(a)}&\hspace{-2mm}
\frac\gamma q\ssb+\ssb\frac1{\tilde{q}}\ssb
=\ssb\frac{\gamma+1}2\ssb-\ssb\frac2{n-1}\,,\\
\,\text{(d.i)}&\hspace{-2mm}
\frac\gamma q\ssb+\ssb\frac1{\tilde{q}}\ssb=\ssb1\,,\\
\,\text{(d.ii)}&\hspace{-2mm}
\frac\gamma q\ssb+\ssb\frac{n-1}{2\,n}\ssf\frac1{\tilde{q}}\ssb
=\ssb\frac{n+3}{4\,n}\ssb+\ssb\frac2n\ssf\frac1{\gamma-1}\,,
\end{cases}\end{equation}
and with vertices given by
\begin{equation}\label{vertices}\begin{cases}
\;\frac1{\qtwo}\ssb=
\frac4{n+1}\ssf\frac1{\gamma-1}\ssb
-\ssb\frac12\ssf\frac{n-1}{n+1}\,,
&\frac1{\tildeqtwo}\ssb
=\ssb\frac n{n+1}\ssf\gamma\ssb-\ssb\frac4{n+1}\ssf\frac1{\gamma-1}\ssb
+\ssb\frac12\ssb-\ssb\frac2{n-1}\ssb-\ssb\frac4{n+1}\,,\\
\;\frac1{\qthree}\ssb=
\frac4{n+1}\ssf\frac1{\gamma-1}\ssb
-\ssb\frac12\ssf\frac{n+3}{n+1}\ssf\frac1\gamma\,,
&\frac1{\tildeqthree}\ssb
=\ssb\frac32\ssf\frac{n-1}{n+1}\ssb
-\ssb\frac4{n+1}\ssf\frac1{\gamma-1}\ssf.
\end{cases}\end{equation}

\begin{figure}[ht]
\begin{center}
\psfrag{a}[c]{(\ref{edges1}.a)}
\psfrag{d.i}[c]{(\ref{edges1}.d.i)}
\psfrag{d.ii}[c]{(\ref{edges1}.d.ii)}
\psfrag{(1/q2,1/tildeq2)}[l]{\color{blue}$\bigl(\frac1{\qtwo},\frac1{\tildeqtwo}\bigr)$}
\psfrag{(1/q3,1/tildeq3)}[r]{\color{blue}$\bigl(\frac1{\qthree},\frac1{\tildeqthree}\bigr)$}
\includegraphics[height=75mm]{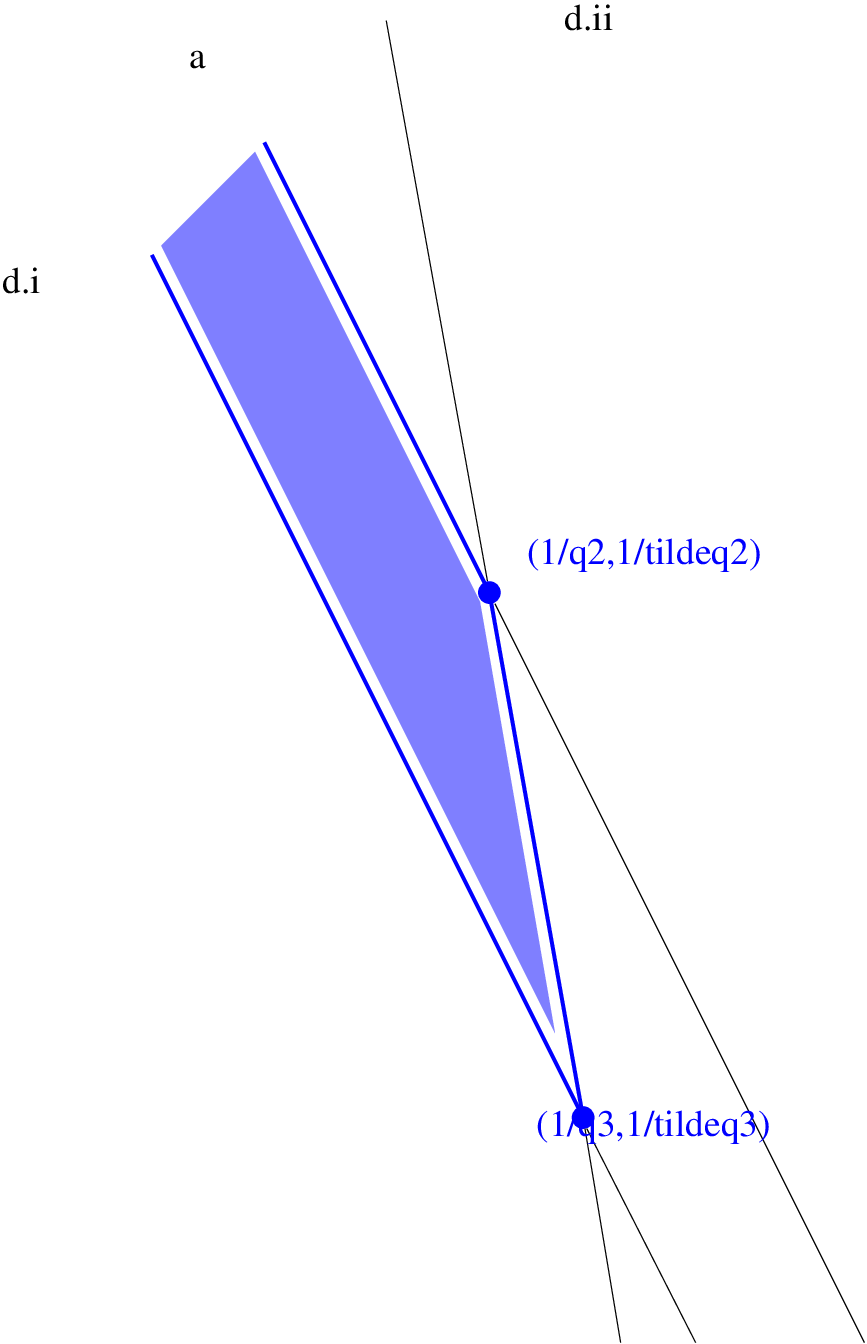}
\caption{Convex region \,$C$}
\label{Convex}
\end{center}
\end{figure}

\noindent
In order to conclude, it remains for us
to determine the possible intersections of the convex region \ssf$C$ \ssf above
with the rectangle \ssf$R$ \ssf defined by \eqref{SQUARE3}
and in each case the minimal regularity
$\sigma\hspace{-.5mm}=\ssb n\ssf\bigl(\frac12\!-\!\frac1q)\!-\!\frac1p$\ssf.

\newpage

\noindent$\bullet$
\textit{\,Case 4}\,:
\,$\gammaconf\ssb\le\ssb\gamma\ssb\le\ssb\gammathree$
\vspace{-2mm}

\begin{figure}[ht]
\begin{center}
\psfrag{0}[r]{$0$}
\psfrag{1/q}[c]{$\frac1q$}
\psfrag{1/qtilde}[r]{$\frac1{\tilde{q}}$}
\psfrag{(1/q2,1/tildeq2)}[c]{\color{violet}$(\frac1{q_2},\frac1{\tilde q_2})$}
\psfrag{(1/q3,1/tildeq3)}[c]{\color{violet}$(\frac1{q_3},\frac1{\tilde q_3})$}
\psfrag{1/n(2/(gamma-1)-1/2gamma)}[c]
{\color{red}$\frac1n\ssf(\frac2{\gamma\ssf-1}\!-\!\frac1{2\ssf\gamma})$}
\psfrag{1/2-(2-gamma)/(n-1)}[r]
{\color{red}$\frac12\!-\!\frac{2\ssf-\ssf\gamma}{n\ssf-1}$}
\psfrag{a}[l]{(\ref{edges1}.a)}
\psfrag{d.i}[l]{(\ref{edges1}.d.i)}
\psfrag{d.ii}[l]{(\ref{edges1}.d.ii)}
\includegraphics[height=60mm]{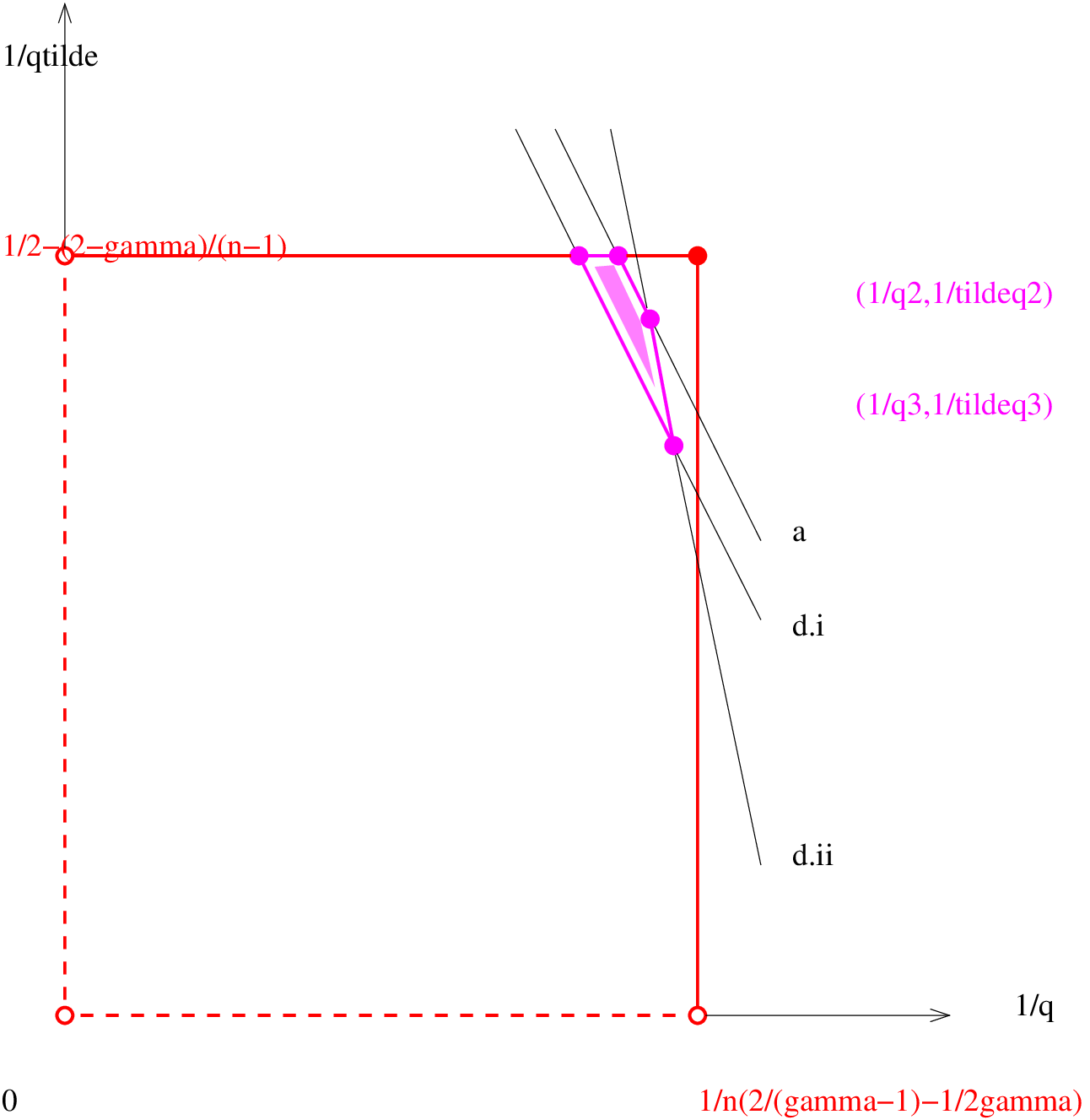}
\caption{\textit{Case 4}\,:
\,$\gammaconf\ssb\le\!\gamma\ssb\le\ssb\gammathree$}
\label{Case4}
\end{center}
\end{figure}

\noindent$\bullet$
\textit{\,Case 5}\,:
\,$\gammathree\ssb\le\ssb\gamma\ssb\le\ssb\gammafour$
\vspace{-2mm}

\begin{figure}[ht]
\begin{center}
\psfrag{0}[r]{$0$}
\psfrag{1/q}[c]{$\frac1q$}
\psfrag{1/qtilde}[r]{$\frac1{\tilde{q}}$}
\psfrag{(1/q2,1/tildeq2)}[c]{\color{violet}$(\frac1{q_2},\frac1{\tilde q_2})$}
\psfrag{(1/q3,1/tildeq3)}[c]{\color{violet}$(\frac1{q_3},\frac1{\tilde q_3})$}
\psfrag{1/n(2/(gamma-1)-1/2gamma)}[c]
{\color{red}$\frac1n\ssf(\frac2{\gamma\ssf-1}\!-\!\frac1{2\ssf\gamma})$}
\psfrag{1/2-(2-gamma)/(n-1)}[r]
{\color{red}$\frac12\!-\!\frac{2\ssf-\ssf\gamma}{n\ssf-1}$}
\psfrag{a}[l]{(\ref{edges1}.a)}
\psfrag{d.i}[l]{(\ref{edges1}.d.i)}
\psfrag{d.ii}[l]{(\ref{edges1}.d.ii)}
\includegraphics[height=60mm]{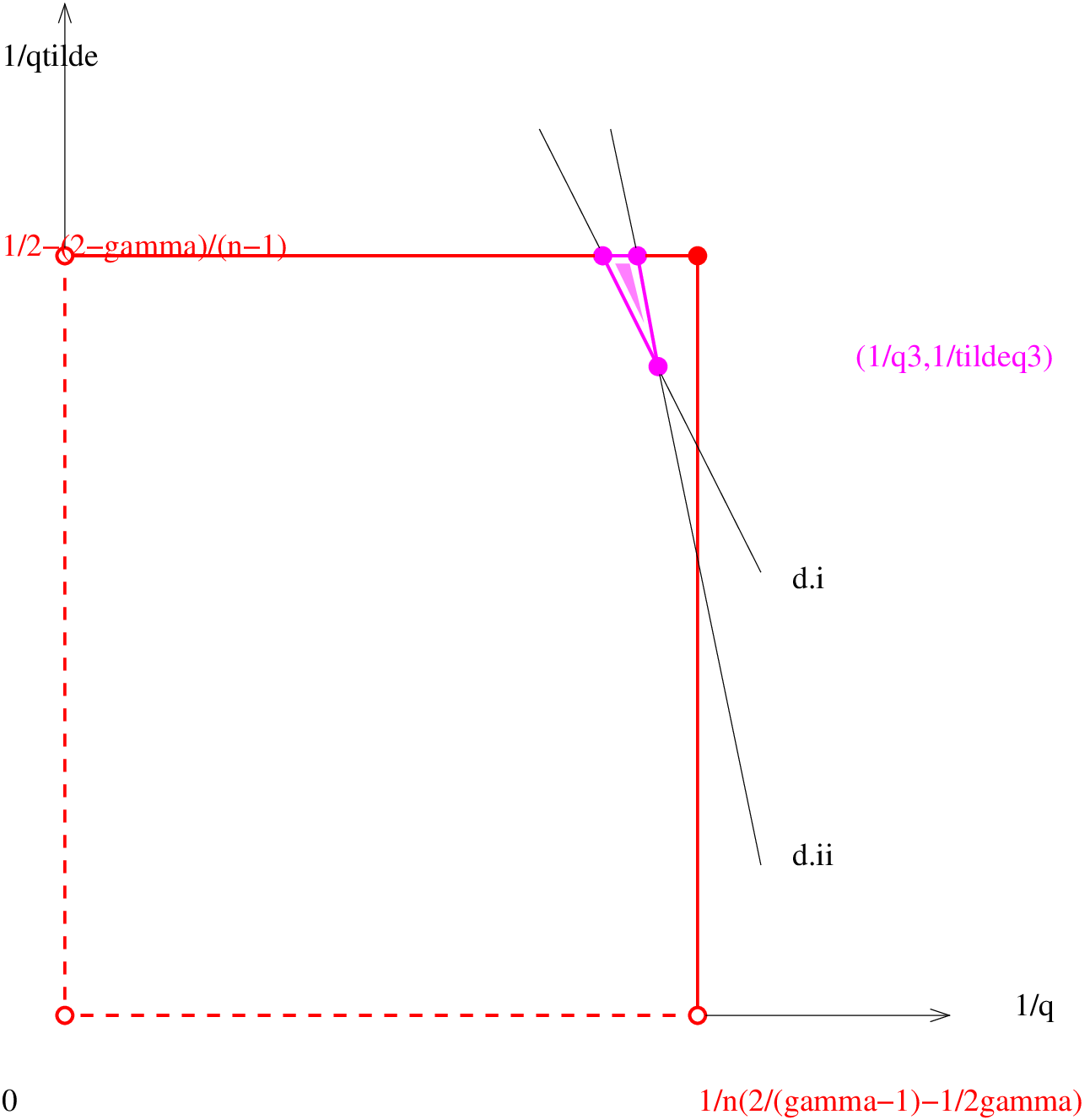}
\caption{\textit{Case 5}\,:
\,$\gammathree\ssb\le\!\gamma\ssb\le\ssb\gammafour$}
\label{Case5}
\end{center}
\end{figure}

\noindent
In both cases, the minimal regularity
\ssf$\sigma\ssb=\ssb\sigma_3(\gamma)$
\ssf is reached when $\bigl(\frac1p,\frac1q\bigr)$
and $\bigl(\frac1q,\frac1{\tilde{q}}\bigr)$
lie on the edges (\ref{2lines}.d.ii) and (\ref{edges1}.d.ii).
See Figures \ref{FIGURE3} and \ref{FIGURE4}.
This concludes the proof of Theorem \ref{GWPND}
for \ssf$\gammaconf\ssb<\ssb\gamma\ssb\le\ssb\gammafour$
\ssf and \ssf$n\ssb\ge\ssb6$\ssf.
\medskip

\noindent$\blacktriangleright$
\,Assume that \ssf$3\ssb\le\ssb n\ssb\le\ssb5$\ssf.
\smallskip

Then \ssf$\gamma\ssb\ge\ssb\gammaconf\ssb\ge\ssb2$
\ssf and the four quadrant figures become

\newpage

\begin{figure}[ht]
\begin{center}
\psfrag{1/p}[c]{$\frac1p$}
\psfrag{1/ptilde}[c]{$\frac1{\tilde{p}}$}
\psfrag{1/q}[c]{$\frac1q$}
\psfrag{1/qtilde}[c]{$\frac1{\tilde{q}}$}
\psfrag{(1/p2,1/q2)}[c]{\color{red}$\bigl(\frac1{\ptwo},\frac1{\qtwo}\bigr)$}
\psfrag{(1/tildep2,1/tildeq2)}[c]{\color{red}$\bigl(\frac1{\tilde{p}_2},\frac1{\tilde{q}_2}\bigr)$}
\psfrag{(1/p2,1/tildep2)}[c]{\color{red}$\bigl(\frac1{\ptwo},\frac1{\tilde{p}_2}\bigr)$}
\psfrag{(1/q2,1/tildeq2)}[c]{\color{red}$\bigl(\frac1{\qtwo},\frac1{\tilde{q}_2}\bigr)$}
\psfrag{1/2}[c]{$\frac12$}
\psfrag{1/2right}[r]{$\frac12$}
\psfrag{1/2-1/(n-1)}[c]{$\frac12\!-\!\frac1{n-1}$}
\psfrag{1/gamma}[c]{$\frac1\gamma$}
\psfrag{1/2gamma}[c]{$\frac1{2\ssf\gamma}$}
\psfrag{1/2-2/(n-1)1/gamma}[c]
{$\frac12\ssb-\ssb\frac2{n\ssf-1}\ssf\frac1\gamma$}
\psfrag{1/n(2/(gamma-1)-1/2gamma)}[c]{\color{red}$\frac1n\ssf\bigl(\frac2{\gamma-1}\!-\!\frac1{2\ssf\gamma}\bigr)$}
\psfrag{conditions4.c}[c]{(\ref{conditions4}.c)}
\psfrag{2lines.a}[c]{(\ref{2lines}.a)}
\psfrag{2lines.d.ii}[c]{(\ref{2lines}.d.ii)}
\psfrag{edges1.a}[c]{(\ref{edges1}.a)}
\psfrag{edges1.d.ii}[c]{(\ref{edges1}.d.ii)}
\includegraphics[height=110mm]{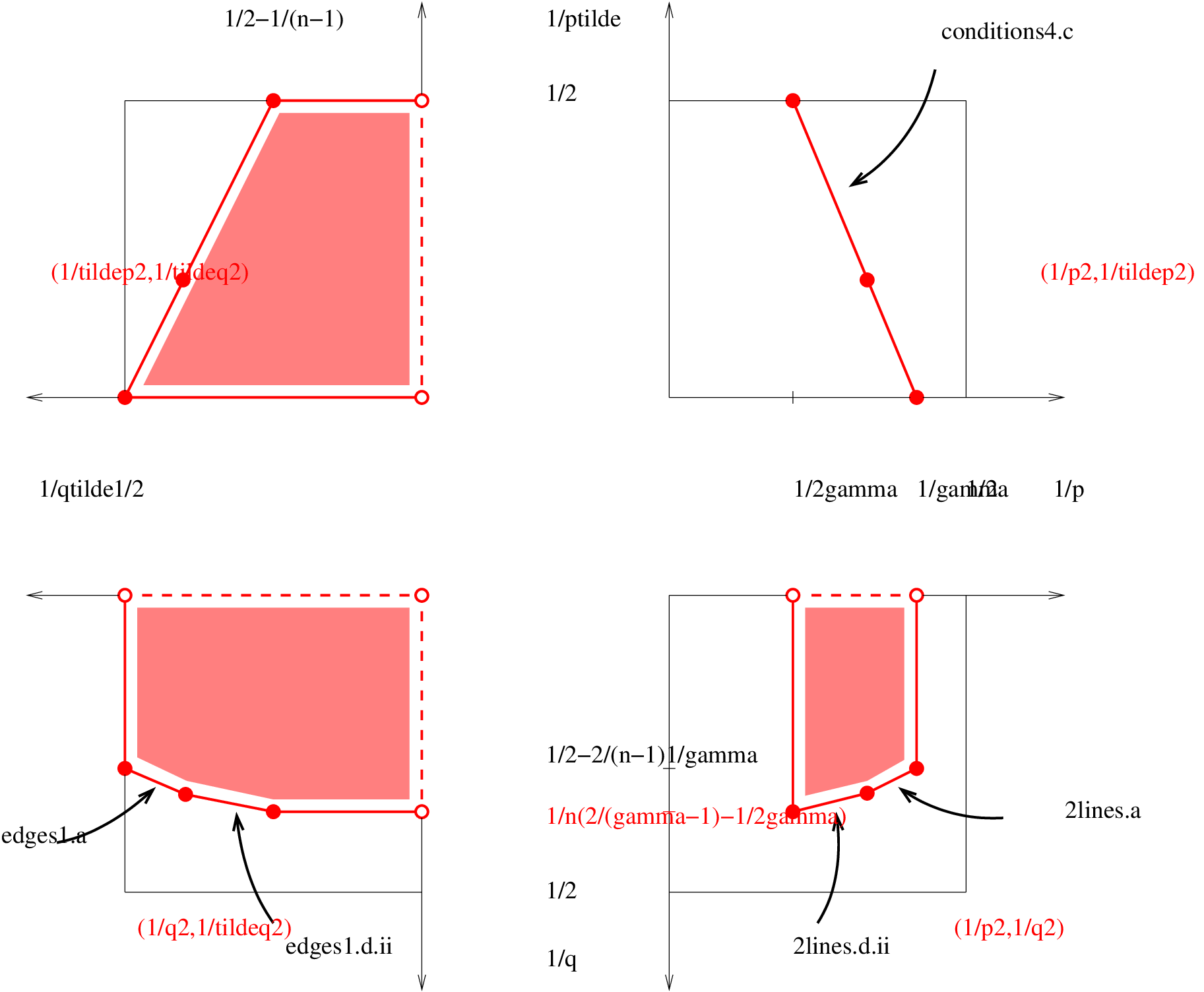}
\caption{\textit{Case 4}\,:
\,$\gammaconf\ssb\le\ssb\gamma\ssb\le\ssb\gammathree$}
\label{FIGURE5}
\end{center}
\end{figure}

\newpage

\begin{figure}[ht]
\begin{center}
\psfrag{1/p}[c]{$\frac1p$}
\psfrag{1/ptilde}[c]{$\frac1{\tilde{p}}$}
\psfrag{1/q}[c]{$\frac1q$}
\psfrag{1/qtilde}[c]{$\frac1{\tilde{q}}$}
\psfrag{1/2}[c]{$\frac{1\vphantom{\gamma}}2$}
\psfrag{1/2-1/(n-1)}[c]{$\frac12\!-\!\frac1{n-1}$}
\psfrag{1/gamma}[c]{$\frac1\gamma$}
\psfrag{1/2gamma}[c]{$\frac1{2\ssf\gamma}$}
\psfrag{1/n(2/(gamma-1)-1/gamma)}[c]
{$\frac1n\ssf\bigl(\frac2{\gamma-1}\!-\!\frac1\gamma\bigr)$}
\psfrag{1/n(2/(gamma-1)-1/2gamma)}[c]
{\color{red}$\frac1n\ssf\bigl(\frac2{\gamma-1}\!-\!\frac1{2\ssf\gamma}\bigr)$}
\psfrag{conditions4.c}[c]{(\ref{conditions4}.c)}
\psfrag{2lines.d.ii}[c]{(\ref{2lines}.d.ii)}
\psfrag{edges1.d.ii}[c]{(\ref{edges1}.d.ii)}
\includegraphics[height=110mm]{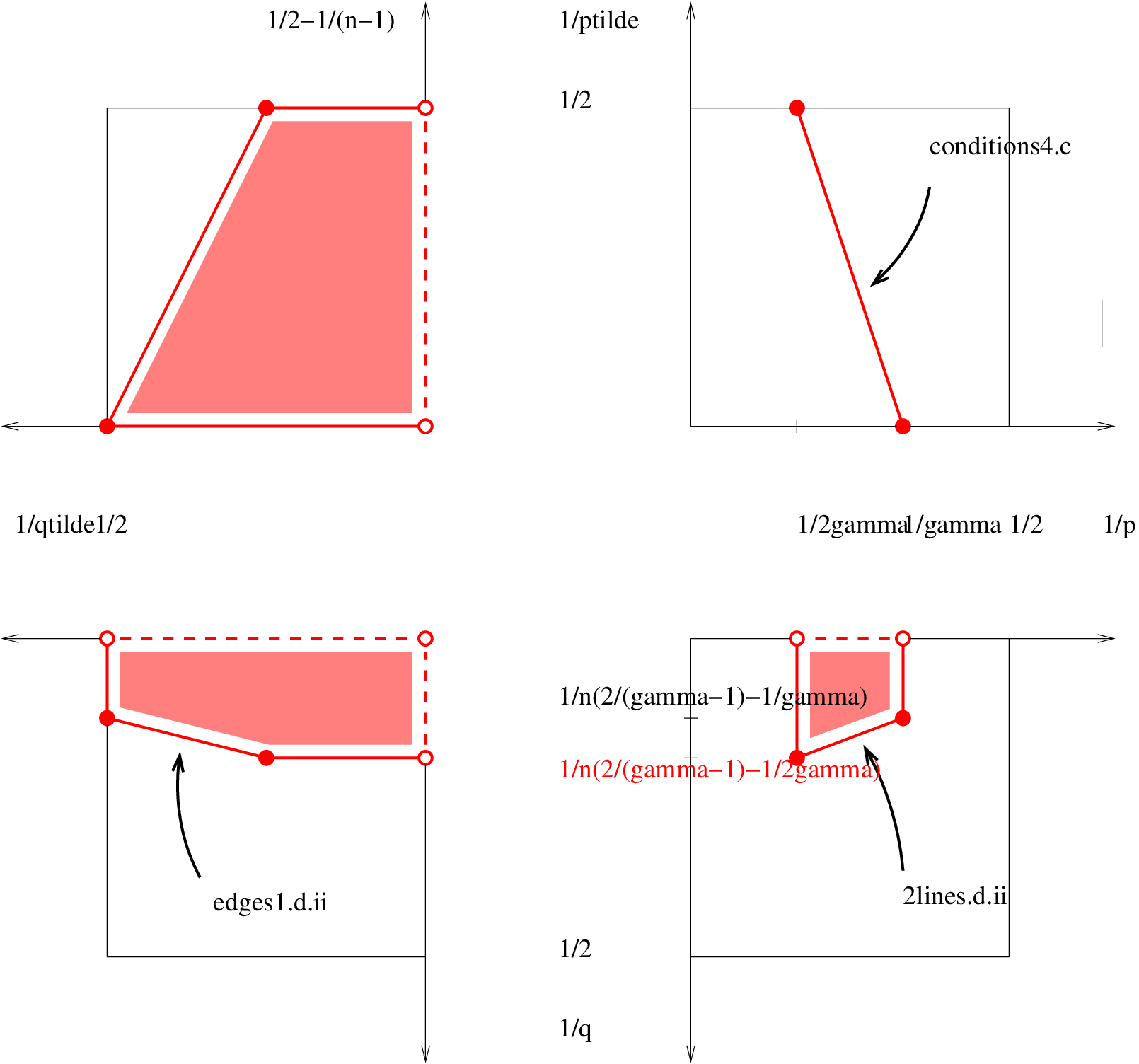}
\caption{\textit{Case 5}\,:
\,$\gammathree\ssb\le\ssb\gamma\ssb\le\ssb\gammafour$}
\label{FIGURE6}
\end{center}
\end{figure}

\newpage

Consequently the four conditions
(\ref{conditions4}.a), (\ref{conditions4}.$\tilde{\text{a}}$),
(\ref{conditions4}.c), (\ref{conditions4}.d.ii)
reduce again to the two conditions (\ref{aacd}.a), (\ref{aacd}.d.ii)
if \ssf$\gammaconf\ssb\le\ssb\gamma\ssb\le\ssb\gammathree$\ssf,
and actually to the single condition (\ref{aacd}.d.ii)
if \ssf$\gammathree\ssb\le\ssb\gamma\ssb\le\ssb\gammafour$\ssf,
but this time in the rectangle
\begin{equation}\label{SQUARE4}\textstyle
R\ssf=\bigl(\ssf0\ssf,
\frac1n\ssf\bigl(\frac2{\gamma-1}\!-\!\frac1{2\ssf\gamma}\bigr)\ssf\bigr]
\ssb\times\ssb\bigl(\ssf0\ssf,\frac12\ssf\bigr]\,.
\end{equation}
Moreover (\ref{conditions4}.e) is satisfied,
as \ssf$\frac1q\!
\le\!\frac1n\ssf(\frac2{\gamma\ssf-1}\!-\!\frac1{2\ssf\gamma})\!
<\!\frac1\gamma$\ssf.
\smallskip

We conclude again by examining
the possible intersections \ssf$C\ssb\cap\ssb R$
\ssf of the convex region defined
by (\ref{conditions4}.d.i), (\ref{aacd}.a), (\ref{aacd}.d.ii)
with the rectangle \eqref{SQUARE4}
and by determining in each case the minimal regularity
$\sigma\hspace{-.5mm}=\ssb n\ssf\bigl(\frac12\!-\!\frac1q)\!-\!\frac1p$\ssf.
\vspace{2mm}

\noindent$\bullet$
\textit{\,Case 4}\,:
\,$\gammaconf\ssb\le\ssb\gamma\ssb\le\ssb\gammathree$
\vspace{-2mm}

\begin{figure}[ht]
\begin{center}
\psfrag{0}[r]{$0$}
\psfrag{1/2}[c]{\color{red}$\frac12$}
\psfrag{1/q}[c]{$\frac1q$}
\psfrag{1/qtilde}[r]{$\frac1{\tilde{q}}$}
\psfrag{(1/q2,1/tildeq2)}[c]{\color{violet}$(\frac1{q_2},\frac1{\tilde q_2})$}
\psfrag{(1/q3,1/tildeq3)}[c]{\color{violet}$(\frac1{q_3},\frac1{\tilde q_3})$}
\psfrag{1/n(2/(gamma-1)-1/2gamma)}[c]
{\color{red}$\frac1n\ssf(\frac2{\gamma\ssf-1}\!-\!\frac1{2\ssf\gamma})$}
\psfrag{1/2-(2-gamma)/(n-1)}[r]
{\color{red}$\frac12\!-\!\frac{2\ssf-\ssf\gamma}{n\ssf-1}$}
\psfrag{a}[l]{(\ref{edges1}.a)}
\psfrag{d.i}[l]{(\ref{edges1}.d.i)}
\psfrag{d.ii}[l]{(\ref{edges1}.d.ii)}
\includegraphics[height=60mm]{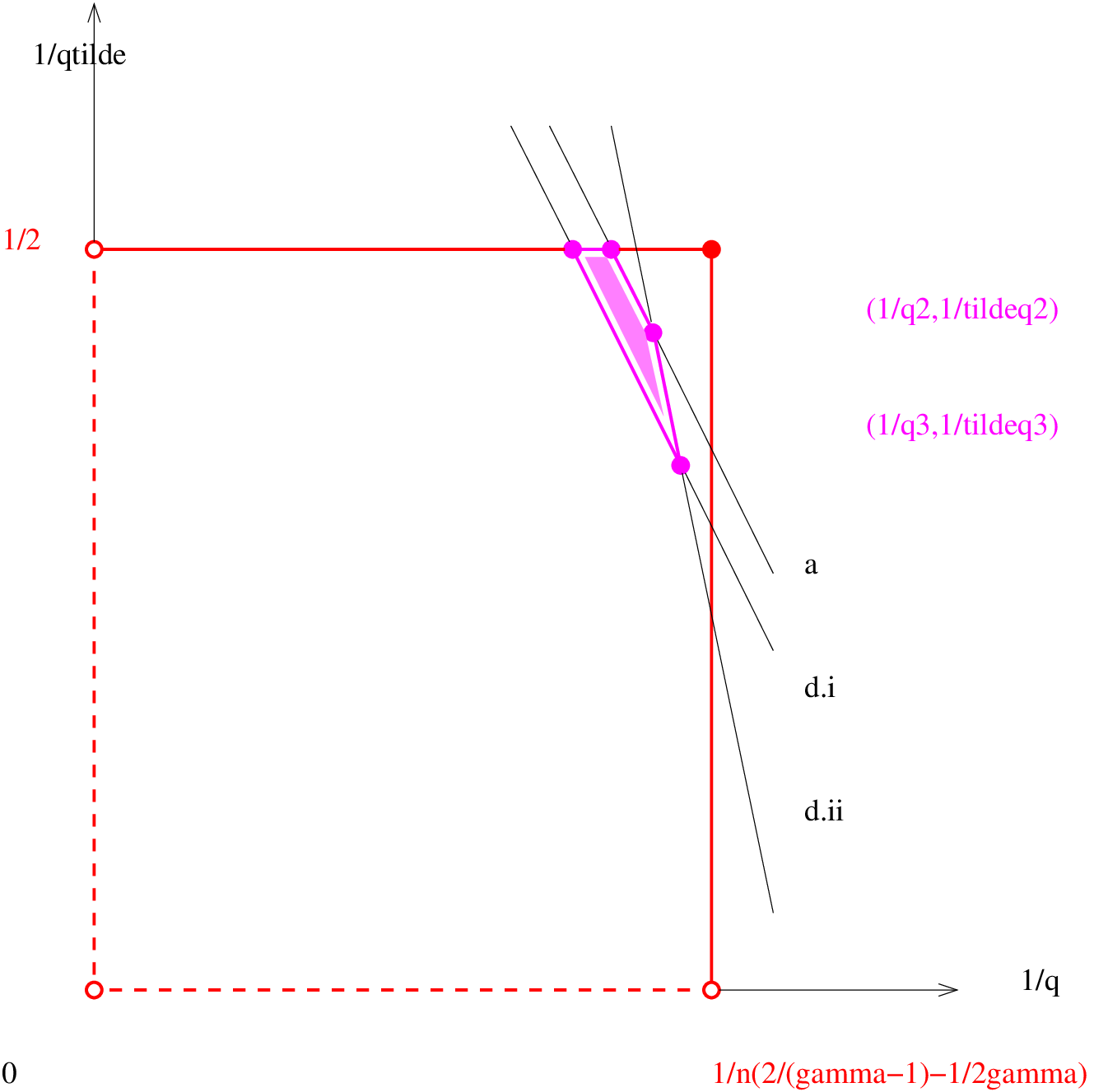}
\caption{\textit{Case 4}\,:
\,$\gammaconf\ssb\le\!\gamma\ssb\le\ssb\gammathree$}
\label{Case4bis}
\end{center}
\end{figure}

\noindent$\bullet$
\textit{\,Case 5}\,:
\,$\gammathree\ssb\le\ssb\gamma\ssb\le\ssb\gammafour$
\vspace{-2mm}

\begin{figure}[ht]
\begin{center}
\psfrag{0}[r]{$0$}
\psfrag{1/2}[c]{\color{red}$\frac12$}
\psfrag{1/q}[c]{$\frac1q$}
\psfrag{1/qtilde}[r]{$\frac1{\tilde{q}}$}
\psfrag{(1/q2,1/tildeq2)}[c]{\color{violet}$(\frac1{q_2},\frac1{\tilde q_2})$}
\psfrag{(1/q3,1/tildeq3)}[c]{\color{violet}$(\frac1{q_3},\frac1{\tilde q_3})$}
\psfrag{1/n(2/(gamma-1)-1/2gamma)}[c]
{\color{red}$\frac1n\ssf(\frac2{\gamma\ssf-1}\!-\!\frac1{2\ssf\gamma})$}
\psfrag{1/2-(2-gamma)/(n-1)}[r]
{\color{red}$\frac12\!-\!\frac{2\ssf-\ssf\gamma}{n\ssf-1}$}
\psfrag{a}[l]{(\ref{edges1}.a)}
\psfrag{d.i}[l]{(\ref{edges1}.d.i)}
\psfrag{d.ii}[l]{(\ref{edges1}.d.ii)}
\includegraphics[height=60mm]{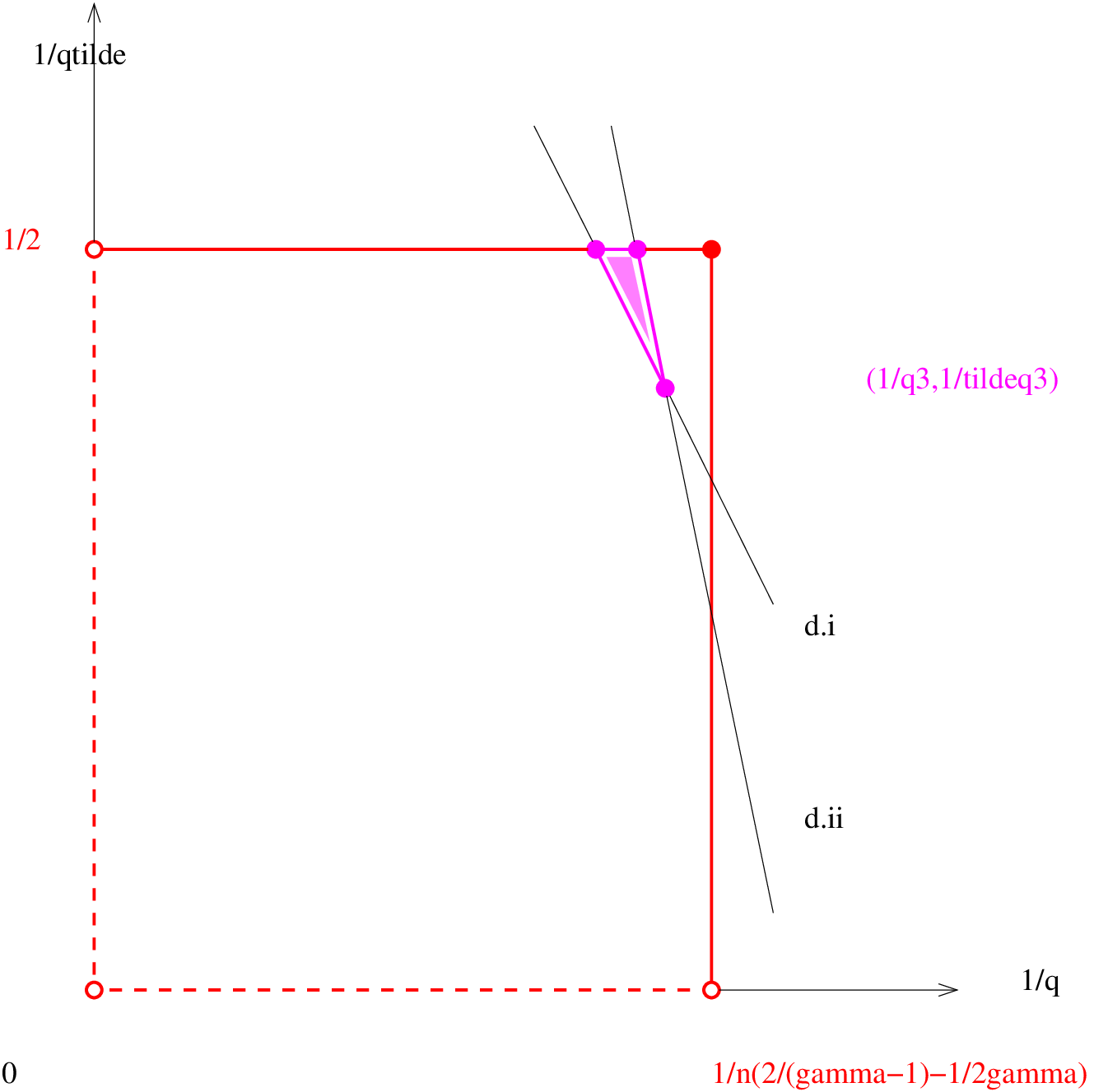}
\caption{\textit{Case 5}\,:
\,$\gammathree\ssb\le\!\gamma\ssb\le\ssb\gammafour$}
\label{Case5bis}
\end{center}
\end{figure}

\noindent
In both cases, the minimal regularity
\ssf$\sigma\ssb=\ssb\sigma_3(\gamma)$
\ssf is reached again when $\bigl(\frac1p,\frac1q\bigr)$
and $\bigl(\frac1q,\frac1{\tilde{q}}\bigr)$
lie on the edges (\ref{2lines}.d.ii) and (\ref{edges1}.d.ii).
See Figures \ref{FIGURE5} and \ref{FIGURE6}.
This concludes the proof of Theorem \ref{GWPND}
for \ssf$\gammaconf\ssb<\ssb\gamma\ssb\le\ssb\gammafour$
\ssf and \ssf$3\ssb\le\ssb n\ssb\le\ssb5$\ssf.
\end{proof}

\begin{remark}
In dimension \,$n\!=\!3$\ssf,
Metcalfe and Taylor \cite{MT} obtain a global existence result
beyond \,$\gamma\!=\!\gammafour$\ssf.
In a forthcoming work,
we shall deal with this case in higher dimensions.
\end{remark}

\begin{remark}
In dimension \,$n\ssb=\ssb2$\ssf,
the statement of Theorem \eqref{GWPND} holds true
with \eqref{RegularityND} replaced by
\begin{equation}\label{Regularity2D}\begin{cases}
\;\sigma\ssb=\ssb0^+
&\text{if \;}1\ssb<\ssb\gamma\ssb\le\ssb2\ssf,\\
\;\sigma\ssb=\ssb\tilde{\sigma}_1(\gamma)^+
&\text{if \;}2\ssb\le\ssb\gamma\ssb\le\ssb3\ssf,\\
\;\sigma\ssb=\ssb\sigma_2(\gamma)
&\text{if \;}3\ssb<\ssb\gamma\ssb<\ssb5\,,\\
\;\sigma\ssb=\ssb\sigma_3(\gamma)^+
&\text{if \;}5\ssb\le\ssb\gamma\ssb<\ssb\infty\ssf.\\
\end{cases}\end{equation}
\end{remark}

\noindent
where \ssf$\tilde{\sigma}_1(\gamma)\ssb
=\ssb\frac34\ssb-\ssb\frac32\ssf\frac1\gamma$\ssf.
Notice that the condition \ssf$q\!>\!\gamma$
\ssf is not redundant if \ssf$2\!<\!\gamma\!<\!3$
\ssf and
\linebreak
\vspace{-4.5mm}

\noindent
that it is actually responsible for the curve \ssf$\tilde{C}_1$.

\begin{figure}[ht]
\begin{center}
\psfrag{0}[c]{$0$}
\psfrag{1}[c]{$1$}
\psfrag{2}[c]{$2$}
\psfrag{3}[c]{$3$}
\psfrag{1/2}[c]{$\frac12$}
\psfrag{1/4}[c]{$\frac14$}
\psfrag{Ctilde1}[c]{\color{red}$\tilde{C}_1$}
\psfrag{C2}[c]{\color{red}$C_2$}
\psfrag{C3}[c]{\color{red}$C_3$}
\psfrag{gamma}[c]{$\gamma$}
\psfrag{gammaconf=5}[c]{$\gammaconf\ssb=\ssb5$}
\psfrag{sigma}[c]{$\sigma$}
\includegraphics[width=125mm]{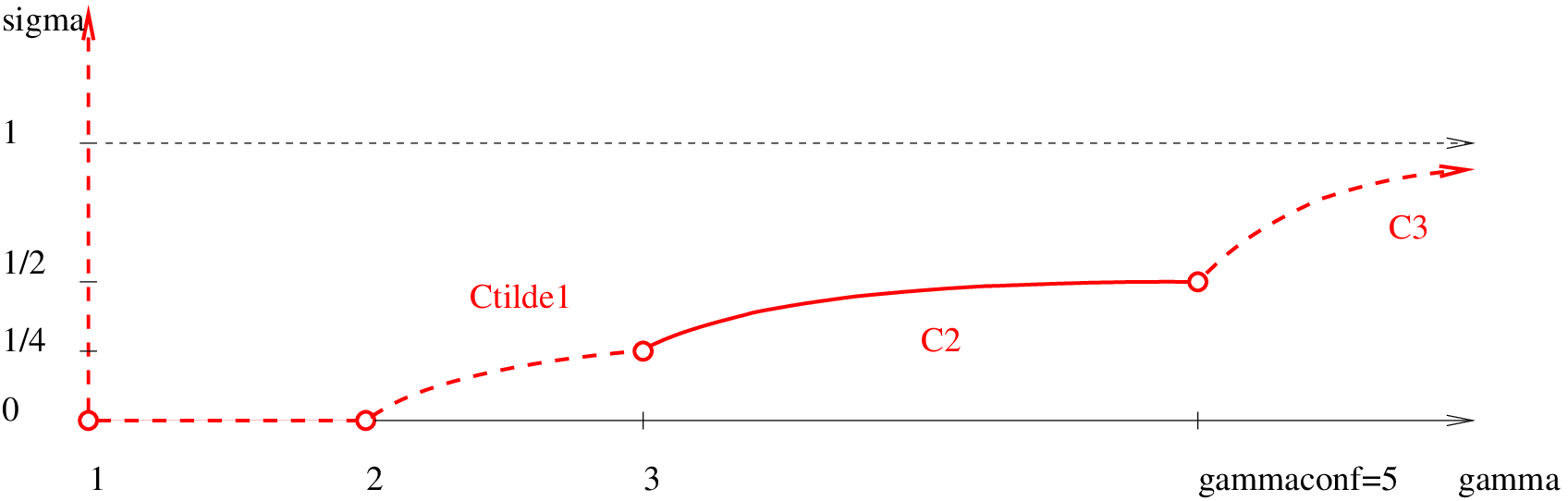}
\caption{Regularity for GWP on $\mathbb{H}^2$}
\label{GlobalRegularity2}
\end{center}
\end{figure}

\section*{Appendix A}
\label{AppendixA}

In this appendix we collect some lemmata in Fourier analysis on \ssf$\R$\ssf,
which are used in the kernel analysis carried out in Section \ref{Kernel}.
\medskip

\noindent
\textbf{Lemma A.1.}
\textit{Consider the oscillatory integral
\begin{equation*}
I(t,x)=\!\int_{-\infty}^{+\infty}\hspace{-1mm}d\lambda\;
a(\lambda)\,e^{\ssf i\ssf t\ssf\phi(\lambda)}
\end{equation*}
where the phase is given by
\begin{equation*}\textstyle
\phi(\lambda)=\sqrt{\ssf\lambda^2\ssb+\kappa^2\ssf}-\frac{x\,\lambda}t
\end{equation*}
{\rm(}recall that \,$\kappa$ is a fixed constant $>\ssb0$\ssf{\rm)}
and the amplitude \,$a\!\in\!\mathcal{C}_c^\infty(\R)$
has the following behavior at the origin
\begin{equation}\label{assumptionA1}
a(\lambda)=\text{O}\ssf(\lambda^2)\,.
\end{equation}
Then}
\begin{equation*}\textstyle
|\ssf I(t,x)\ssf|\,\lesssim\,\frac{1\ssf+\,|x|}{(\ssf1\ssf+\,|t|\ssf)^{3/2}}
\qquad\forall\;|x|\ssb\le\ssb\frac{|t|}2\,.
\end{equation*}

\begin{proof}
Let us compute the first two derivatives
\begin{equation}\label{derivativesphi}\textstyle
\phi'(\lambda)=\frac{\lambda}{\sqrt{\lambda^2\ssf+\,\kappa^2}}-\frac xt
\quad\text{and}\quad
\phi''(\lambda)=\kappa^2\,(\lambda^2\!+\ssb\kappa^2)^{-\frac32}\,.
\end{equation}
The phase \ssf$\phi$ \ssf has a single stationary point\,:
\begin{equation}\label{stationarypoint}\textstyle
\lambda_0=\ssf\kappa\,\frac xt\ssf
\bigl(\ssf1\!-\ssb\frac{x^2}{t^2}\bigr)^{-\frac12}\,,
\end{equation}
which remains bounded under our assumption \ssf$|x|\!\le\!\frac{|t|}2$\,:
\begin{equation}\label{estimatestationarypoint}\textstyle
|\lambda_0|\le\frac\kappa{\sqrt{3\ssf}}\le\kappa\,.
\end{equation}
For later use, let us compute
\begin{equation*}\textstyle
\phi(\lambda_0)=\kappa\,
\bigl(\ssf1\!-\ssb\frac{x^2}{t^2}\bigr)^{\frac12}
\quad\text{and}\quad
\phi''(\lambda_0)=\kappa^{-1}\ssf
\bigl(\ssf1\!-\ssb\frac{x^2}{t^2}\bigr)^{\frac32}\,.
\end{equation*}
%According to Taylor's formula, we have
%\begin{equation*}
%\phi(\lambda)-\ssf\phi(\lambda_0)\ssf=
%\int_{\,0}^{\ssf1}\!ds\,(1\!-\!s)\,
%\phi''\bigl((1\!-\!s)\lambda_0\ssb+\ssb s\ssf\lambda\ssf\bigr)\,.
%\end{equation*}
Since \ssf$\phi''\!>\!0$\ssf,
we can perform a global change of variables
\,$\lambda\ssb\longleftrightarrow\ssb\mu$
\,on \ssf$\R$ \ssf so that
\begin{equation*}
\phi(\lambda)-\phi(\lambda_0)=\mu^2\,.
\end{equation*}
Specifically,
\begin{equation*}
\mu=\epsilon(\lambda)\,(\lambda\!-\!\lambda_0)\,,
\end{equation*}
where
\begin{equation*}
\epsilon(\lambda)=\Bigl\{\ssf\int_{\,0}^{\ssf1}\!ds\,(1\!-\!s)\,
\phi''\bigl((1\!-\!s)\lambda_0\ssb+\ssb s\ssf\lambda\ssf\bigr)
\ssf\Bigr\}^{\frac12}\ssf.
\end{equation*}
This way, our oscillatory integral becomes
\begin{equation}\textstyle
I(t,x)=\ssf e^{\ssf i\ssf t\ssf\phi(\lambda_0)}
{\displaystyle\int_{\ssf\R}}\,d\mu\;
\tilde{a}(\mu)\,e^{\ssf(-1+i\ssf t\ssf)\ssf\mu^2}\,,
\end{equation}
where
\begin{equation*}\textstyle
\tilde{a}(\mu)=\frac{d\lambda}{d\mu}\,a(\lambda(\mu))\,e^{\ssf\mu^2}
\end{equation*}
is again a smooth function with compact support,
whose derivatives are controlled uniformly in \ssf$t$ \ssf and \ssf$x$\ssf,
as long as \ssf$|x|\!\le\!\frac{|t|}2$\ssf.
Using Taylor's formula, let us expand
\begin{equation*}
\tilde{a}(\mu)
=\ssf\sum\nolimits_{\ssf j=0}^{\,3}\tilde{a}_j\,\mu^{\ssf j}
+\,\tilde{a}_4(\mu)\,\mu^4\ssf,
\end{equation*}
where
\begin{equation*}\textstyle
\tilde{a}_0
=\bigl(\frac2{\phi''(\lambda_0)}\bigr)^{\frac12}\,a(\lambda_0\ssf)
=\text{O}\ssf(\lambda_0^2\ssf)
=\text{O}\ssf(\frac{x^2}{t^2})\,,
\end{equation*}
the other constants \ssf$\tilde{a}_1$, $\tilde{a}_2$\ssf, $\tilde{a}_3$
and the function \ssf$\tilde{a}_4(\mu)$\ssf, as well as its derivatives,
are bounded uniformly in \ssf$t$ \ssf and \ssf$x$\ssf.
Let us split up accordingly
\begin{equation*}
I(t,x)=\ssf\sum\nolimits_{\ssf j=0}^{\,4}I_j(t,x)\,,
\end{equation*}
where
\begin{equation*}
I_j(t,x)=\ssf\tilde{a}_j\;e^{\ssf i\ssf t\ssf\phi(\lambda_0)}\ssb
\int_{\ssf\R}d\mu\;\mu^{\ssf j}\,e^{\ssf(-1+\ssf i\ssf t)\ssf\mu^2}
\qquad(\ssf j\ssb=\ssb0,1,2,3\ssf)
\end{equation*}
and
\begin{equation*}
I_4(t,x)=
e^{\ssf i\ssf t\ssf\phi(\lambda_0)}\ssb
\int_{\ssf\R}d\mu\;\tilde{a}_4(\mu)\,\mu^4\,e^{\ssf(-1+\ssf i\ssf t)\ssf\mu^2}.
\end{equation*}
The first and third expressions are handled by elementary complex integration\,:
\begin{equation*}\begin{aligned}
I_0(t,x)&\textstyle
=\tilde{a}_0\,\sqrt{\pi\ssf}\,e^{\ssf i\ssf t\ssf\phi(\lambda_0)}\ssf
(1\!-\ssb i\ssf t\ssf)^{-\frac12}
=\text{O}\ssf\bigl(\ssf\frac{x^2}{t^2}\ssf
(1\hspace{-.75mm}+\!|t|)^{-\frac12}\bigr)\ssb
=\text{O}\ssf\bigl(\frac{1\,+\,|x|}{(\ssf1\ssf+\ssf|t|\ssf)^{3/2}}\bigr)\,,\\
I_2(t,x)&\textstyle
=\tilde{a}_2\,\frac{\sqrt{\pi\ssf}}2\,e^{\ssf i\ssf t\ssf\phi(\lambda_0)}\ssf
(1\!-\ssb i\ssf t\ssf)^{-\frac32}
=\text{O}\ssf\bigl((1\!+\ssb|t|\ssf)^{-3/2}\bigr)\,.\\
\end{aligned}\end{equation*}
The expressions $I_1(t,x)$ and $I_3(t,x)$ vanish by oddness.
The expression $I_4(t,x)$ is obviously bounded by the finite integral
\begin{equation*}
\int_{\ssf\R}d\mu\;\mu^4\,e^{-\mu^2}\,.
\end{equation*}
In order to improve this estimate when \ssf$|t|$ \ssf is large,
let us split up
\begin{equation*}
\int_{\R}\,d\mu\;
=\int_{\ssf|\mu|\ssf\le\ssf|t|^{-1/2}}\hspace{-1mm}d\mu\;
+\int_{\ssf|\mu|\ssf>\ssf|t|^{-1/2}}\hspace{-1mm}d\mu\;.
\end{equation*}
The first integral is easily estimated,
using the uniform boundedness of \ssf$\tilde{a}_4(\mu)$\,:
\begin{equation*}
\Bigl|\,\int_{\ssf|\mu|\ssf\le\ssf|t|^{-1/2}}\hspace{-1mm}d\mu\;
\tilde{a}_4(\mu)\,\mu^4\,e^{\ssf(-1+\ssf i\ssf t)\ssf\mu^2}\ssf\Bigr|\,
\lesssim\int_{\ssf|\mu|\ssf\le\ssf|t|^{-1/2}}\hspace{-1mm}d\mu\;\mu^4\ssf
\lesssim\,|t|^{-\frac52}\,.
\end{equation*}
After two integration by parts, using
\,$\mu\,e^{\ssf(-1+\ssf i\ssf t)\ssf\mu^2}\hspace{-1mm}
=\!-\,\frac1{2\ssf(1-\ssf i\ssf t)}\,
\frac\partial{\partial\mu}\,e^{\ssf(-1+\ssf i\ssf t)\ssf\mu^2}$\ssb,
the second integral is estimated by
\begin{equation*}
|t|^{-\frac52}
+\ssf|t|^{-2}\ssb\int_{\ssf\R}d\mu\,
(\ssf1\!+\ssb|\mu|\ssf)^2\,e^{-\mu^2}\,.
\end{equation*}
Altogether
\begin{equation*}
I_4(t,x)=\text{O}\ssf\bigl(\ssf(1\!+\ssb|t|\ssf)^{-2}\ssf\bigr)
\end{equation*}
and this concludes the proof of Lemma A.1.
\end{proof}

\noindent
\textbf{Lemma A.2.}
\textit{Consider the oscillatory integral
\begin{equation*}
J(t,x)=\!\int_{-\infty}^{+\infty}\hspace{-1mm}d\lambda\;
a(\lambda)\,e^{\ssf i\ssf t\ssf\phi(\lambda)}
\end{equation*}
where the phase is given again by
\begin{equation*}\textstyle
\phi(\lambda)=\sqrt{\ssf\lambda^2\ssb+\kappa^2\ssf}-\frac{x\,\lambda}t
\end{equation*}
and the amplitude \,$a(\lambda)$
is now a symbol \,{\rm(}of any order\,{\rm)} on \,$\R$\ssf,
which vanishes on the interval \;$[-\kappa,\kappa\ssf]$\ssf.
Then}
\begin{equation*}\textstyle
J(t,x)=\text{O}\ssf(\ssf|t|^{-\infty})
\qquad\forall\;0\ssb\le\ssb|x|\ssb\le\ssb\frac{|t|}2\,.
\end{equation*}

\begin{proof}
According to \eqref{derivativesphi}, \eqref{stationarypoint} and
\eqref{estimatestationarypoint},
\begin{itemize}
\item[$\bullet$]
$\,\phi$ \ssf has a single stationary point
\ssf$\lambda_0\!\in\!\bigl[-\frac{\vphantom{|}\kappa}{\sqrt{3\ssf}},
\frac{\vphantom{|}\kappa}{\sqrt{3\ssf}}\ssf\bigr]$\ssf,
which remains away from the support of \ssf$a$\ssf,
\item[$\bullet$]
$\ssf|\ssf\phi'(\lambda)|\ssb
=\ssb\bigl|\frac{\vphantom{|}\lambda}{\sqrt{\lambda^2+\ssf\kappa^2}}\ssb
-\ssb\frac{\vphantom{|}x}t\ssf\bigr|\ssb
\ge\ssb\frac{\vphantom{|}1}{\sqrt{2\ssf}}\ssb
-\ssb\frac{\vphantom{|}1}{\vphantom{\sqrt{2}}2}\ssb>\ssb0$
\,on \ssf$\supp a$\ssf,
\item[$\bullet$]
$\,\phi''$ is a symbol of order $-3\vphantom{\big|}$\ssf.
\end{itemize}
\vspace{1mm}
\noindent
These facts allow us to perform several integrations by parts based on
\begin{equation*}\textstyle
e^{\ssf i\ssf t\ssf\phi(\lambda)}
=\frac1{i\ssf t\ssf\phi'(\lambda)}\,
\frac\partial{\partial\lambda}\,
e^{\ssf i\ssf t\ssf\phi(\lambda)}
\end{equation*}
and to reach the conclusion.
\end{proof}

\section*{Appendix B}
\label{AppendixB}

In this appendix, we recall the definition of Sobolev spaces on \ssf$\Hn$
and some related inequalities.
We refer to \cite{Tr2} for more details
about function spaces on Riemannian manifolds.

Let  \ssf$\sigma\!\in\!\R$ \ssf and \ssf$1\!<\!q\!<\!\infty$\ssf.
Then \ssf$H^{\sigma,\ssf q}(\Hn)$ \ssf denotes the image of \ssf$L^q(\Hn)$
\ssf under $(-\Delta)^{-\frac\sigma2}$
(in the space of distributions on \ssf$\Hn$),
equipped with the norm
\begin{equation*}
\|\ssf f\ssf\|_{H^{\sigma,q}}
=\ssf\|\ssf(-\Delta)^{-\frac\sigma2}\ssb f\ssf\|_{L^q}\,.
\end{equation*}
In this definition, we may replace \ssf$(-\Delta)^{-\frac\sigma2}$
\ssf by \ssf$D^{-\sigma}\!
=\ssb(-\ssf\Delta\!-\!\rho^{\ssf2}\!+\ssb\kappa^2\ssf)^{-\frac\sigma2}$\ssf,
\ssf as long as
\ssf$\kappa\ssb>\ssb2\,\bigl|\frac12\ssb-\ssb\frac1q\bigr|\,\rho$\,,
\ssf in particular by \ssf$\tilde{D}^{-\sigma}\!
=\ssb(-\ssf\Delta\!-\!\rho^{\ssf2}\!+\ssb\tilde{\kappa}^2\ssf)^{-\frac\sigma2}$\ssf,
\ssf since \ssf$\tilde{\kappa}\!>\!\rho$\ssf.
If \ssf$\sigma\!=\!N$ \ssf is a nonnegative integer,
then \ssf$H^{\sigma,\ssf q}(\Hn)$ co\"{\i}ncides with
the Sobolev space
\begin{equation*}
W^{N,\ssf q}(\Hn)
=\{\,f\!\in\!L^q(\Hn)\mid
\nabla^j\ssb f\!\in\! L^q(\Hn)
\hspace{2mm}\forall\;1\!\le\!j\!\le\!N\,\}
\end{equation*}
defined in terms of covariant derivatives.
In the $L^2$ setting, we write \ssf$H^{\sigma}(\Hn)$
\ssf instead of \ssf$H^{\sigma,\ssf2}(\Hn)$\ssf.
\medskip

\noindent
\textbf{Proposition B.1 (Sobolev embedding theorem).}
\textit{Let \,$1\ssb<\ssb q_1,\qtwo\ssb<\ssb\infty$
\ssf and \,$\sigma_1,\sigma_2\ssb\in\ssb\R$
such that \,$\sigma_1\!-\ssb\sigma_2\ssb
\ge\ssb\frac n{q_1}\!-\!\frac n{\qtwo}\!\ge\ssb0$\ssf.
Then
\begin{equation*}
H^{\sigma_1,\ssf q_1}(\Hn)\subset H^{\sigma_2,\ssf \qtwo}(\Hn)\,.
\end{equation*}
By this inclusion, we mean that
there exists a constant \,$C\!\ge\!0$ such that}
\begin{equation*}
\|f\|_{H^{\sigma_2,\ssf \qtwo}}\le C\;\|f\|_{H^{\sigma_1,\ssf q_1}}
\qquad\forall\;f\!\in\ssb C_c^\infty(\Hn)\,.
\end{equation*}


\medskip

\begin{thebibliography}{9999}

%\bibitem{A}
%J.--Ph. Anker,
%\textit{$L_p$ Fourier multipliers
%on Riemannian symmetric spaces of the noncompact type\/},
%Ann. Math. (2) 132 (1990), no. 3, 597--628

%\bibitem{ADY}
%J.-Ph. Anker, E. Damek, C. Yacoub, 
%\textit{Spherical analysis on harmonic $AN$ groups\/},
%Ann. Scuola Norm. Sup. Pisa 33 (1996), 643--679

\bibitem{AP}
J.-Ph. Anker, V. Pierfelice, 
\textit{Nonlinear Schr\"odinger equation on real hyperbolic spaces\/},
Ann. Inst. H. Poincar\'e (C) {\it Non Linear Analysis\/} 26 (2009), 1853--1869

\bibitem{APV1}
J.-Ph. Anker, V. Pierfelice, M. Vallarino, 
\textit{The Schr\"odinger equation on Damek--Ricci spaces\/}, 
preprint [hal--00525155, arXiv:1010.2137],
Comm. Part. Diff. Eq. (to appear)

\bibitem{APV2}
J.-Ph. Anker, V. Pierfelice, M. Vallarino,
\textit{The wave equation on real hyperbolic spaces}, 
preprint [hal--00525251, arXiv:1010.2372]

%\bibitem{A1} 
%F. Astengo, 
%\textit{The maximal ideal space of a heat algebra
%on solvable extensions of H--type groups\/},
%Boll. Un. Mat. Ital. A (7) 9 (1995), 157--165

%\bibitem{A2} 
%F. Astengo, 
%\textit{Multipliers for a distinguished Laplacean
%on solvable extensions of H--type groups\/},
%Monatsh. Math. 120 (1995), 179--188

\bibitem{BG}
H. Bahouri, P. G\'erard, 
\textit{High frequency approximation of critical nonlinear wave equations \/},
Amer. J. Math. 121 (1999), 131--175

%\bibitem{Ba}
%V. Banica, 
%\textit{The nonlinear Schr\"odinger equation on the hyperbolic space\/},
%Comm. P.D.E. 32 (2007), 1643--1677

%\bibitem{BCS}
%V. Banica, R. Carles, G. Staffilani, 
%\textit{Scattering theory
%for radial nonlinear Schr\"odinger equations on hyperbolic space\/},
%Geom. Funct. Anal. 18 (2008), 367--399

%\bibitem{B}
%P. B\'erard, 
%\textit{On the wave equation
%on a compact Riemannian manifolds without conjugate points\/},
%Math. Z. 155 (1977), 249--276

%\bibitem{BL} 
%J. Bergh, J. L\"ofstrom, 
%\textit{Interpolation spaces\/} (\textit{an introduction\/}),
%Grundlehren Math. Wissenschaften 223, Springer--Verlag (1976)

%\bibitem{Bo}
%J. Bourgain, 
%\textit{Fourier transformation restriction phenomena for certain lattice subsets
%and application to the nonlinear evolution equations I -- Schr\"odinger equations\/},
%Geom. Funct. Anal. 3 (1993), 107--156

%\bibitem{BGT}
%N. Burq, P. G\'erard, N. Tzvetkov, 
%\textit{Strichartz inequalities and the nonlinear Schr\"odinger equation
%on compact manifolds\/},
%Amer. J. Math. 126 (2004), 569--605

\bibitem{CK}
M. Christ, A. Kiselev,
\textit{Maximal functions associated to filtrations\/},
J. Funct. Anal. 179 (2001), 409-425

%\bibitem{CW}
%M. Christ, M. Weinstein\,:
%\textit{Dispersion of small amplitude solutions
%of the generalized Korteweg--de Vries equations\/},
%J. Funct. Anal. 100 (1991), 87--109

%\bibitem{CS}
%J.--L. Clerc, E. M. Stein, 
%\textit{$L^p$-multipliers for noncompact symmetric spaces\/},
%Proc. Nat. Acad. Sci. U.S.A. 71 (1974), 3911--3912

%\bibitem{Co1}
%M.G. Cowling, 
%\textit{The Kunze--Stein phenomenon\/},
%Ann. Math 107 (1978), 209--234

%\bibitem{Co2}
%M.G. Cowling, 
%\textit{Herz's``principe de majoration'' and the Kunze-Stein phenomenon\/},
%in \textit{Harmonic analysis and number theory\/} (\textit{Montreal, 1996\/}),
%CMS Conf. Proc. 21, Amer. Math. Soc. (1997), 73--88

%\bibitem{CF}
%M.G. Cowling, J.J.F. Fournier, 
%\textit{Inclusions and noninclusion of spaces of convolution operators\/},
%Trans. Amer. Math. Soc.  221  (1976), 59--95

%\bibitem{CGM1}
%M. Cowling, S. Giulini, S. Meda, 
%\textit{$L^p-L^q$ estimates for functions of the Laplace-Beltrami operator
%on noncompact symmetric spaces I\/},
%Duke Math. J. 72 (1993), 109--150

%\bibitem{CGM2} M. Cowling, S. Giulini, S. Meda, 
%\textit{$L^p-L^q$ estimates for functions of the Laplace-Beltrami operator
%on noncompact symmetric spaces II\/},
%J. Lie. Theory 5 (1995), 1--14

%\bibitem{CGM3} M. Cowling, S. Giulini, S. Meda, 
%\textit{$L^p-L^q$ estimates for functions of the Laplace-Beltrami operator
%on noncompact symmetric spaces III\/},
%Ann. Inst. Fourier 51 (2001), 1047--1069

%\bibitem{D1} 
%E. Damek,
%\textit{Geometry of a semidirect extension of a Heisenberg type nilpotent group\/},
%Colloq. Math. 53 (1987), 255--268

%\bibitem{D2} 
%E. Damek,
%\textit{Curvature of a semidirect extension of a Heisenberg type nilpotent group\/},
%Colloq. Math. 53 (1987), 249--253

%\bibitem{DR1} 
%E. Damek, F. Ricci,
%\textit{A class of nonsymmetric harmonic Riemannian spaces},
%Bull. Amer. Math. Soc. 27 (1992), 139--142

%\bibitem{DR2} 
%E. Damek, F. Ricci,
%\textit{Harmonic analysis on solvable extensions of $H$--type groups},
%J. Geom. Anal. 2 (1992), 213--248

\bibitem{DGK}
P. D'Ancona, V. Georgiev, H. Kubo,
\textit{Weighted decay estimates for the wave equation\/},
J. Diff. Eq. 177 (2001), 146--208

\bibitem{F1}
J. Fontaine,
\textit{Une \'equation semi--lin\'eaire des ondes sur \ssf$\mathbb{H}^3$},
C. R. Acad. Sci. Paris S\'er. I \textit{Math\'ematiques\/} 319 (1994), 935--948

\bibitem{F2}
J. Fontaine,
\textit{A semilinear wave equation on hyperbolic spaces\/},
Comm. Partial Diff. Eq. 22 (1997), 633--659

%\bibitem{GP}
%P. G\'erard, V. Pierfelice,
%\textit{Nonlinear Schr\"odinger equation
%on four--dimensional compact manifolds\/},
%Bull. Soc. Math. France 137 (2009), 503--535

\bibitem{G}
V. Georgiev,
\textit{Semilinear hyperbolic equations\/},
Mem. Math. Soc. Japan 7 (2000)

\bibitem{GLS}
V. Georgiev, H. Lindblad, C. Sogge,
\textit{Weighted Strichartz estimates and global existence
for semilinear wave equations\/},
Amer. J. Math. 119 (1997), 1291--1319

\bibitem{GV}
J. Ginibre, G. Velo, 
\textit{Generalized Strichartz inequalities for the wave equation\/},
J. Funct. Anal. 133 (1995), 50--68

\bibitem{GV1}
J. Ginibre, G. Velo, 
\textit{The global Cauchy problem for the non linear Klein--Gordon equation\/},
Math. Z. 189 (1985), no. 4, 487--505

%\bibitem{GMM} S. Giulini, G. Mauceri, S. Meda, 
%\textit{$L^p$ multipliers on noncompact symmetric spaces\/},
%J. Reine angew. Math. 482 (1997), 151--175

%\bibitem{HTW} A. Hassell, T. Tao, J. Wunsch,
%\textit{Sharp Strichartz estimates on non-trapping asymptotically conic manifolds\/},
%Amer. J. Math.  128  (2006), 963--1024
 
%\bibitem{HS} 
%W. Hebisch, T. Steger,
%\textit{Multipliers and singular integrals on exponential growth groups},
%Math. Z. 245 (2003), 37--61 

%\bibitem{GC}
%I.M. Guelfand, G.E. Chilov, 
%\textit{Les distributions\/}, tome I,
%Dunod (1962)

\bibitem{Ha}
A. Hassani, 
\textit{Equation des ondes
sur les espaces sym\'etriques riemanniens de type non compact\/},
Ph.D. thesis, Universit\'e Paris--Ouest\,/\,Nanterre (2011).

\bibitem{He1}
S. Helgason, 
\textit{Differential geometry, Lie groups, and symmetric spaces\/},
Academic Press (1978), Amer. Math. Soc. (2001)

\bibitem{He2}
S. Helgason, 
\textit{Groups and geometric analysis}
(\textit{integral geometry, invariant differential operators,
and spherical functions\/}),
Academic Press (1984), Amer. Math. Soc. (2002) 

\bibitem{He3}
S. Helgason,
\textit{Geometric analysis on symmetric spaces\/},
Amer. Math. Soc. (1994) 

%\bibitem{Ho1}
%L.V. H\"ormander,
%{\it The analysis of linear partial differential operators I\/}
%({\it distribution theory and Fourier analysis\/}),
%Springer--Verlag (1983, 1990, 2003)

\bibitem{Ho3}
L.V. H\"ormander,
{\it The analysis of linear partial differential operators III\/}
({\it pseudo--differential operators\/}),
Springer--Verlag (1985, 1994, 2007)

\bibitem{I1}
A.D. Ionescu,
\textit{Fourier integral operators
on noncompact symmetric spaces of real rank one\/},
J. Funct. Anal.  174 (2000), 274--300

%\bibitem{I2}
%A.D. Ionescu,
%\textit{An endpoint estimate for the Kunze--Stein phenomenon
%and related maximal operators\/},
%Ann. of Math. (2) 152 (2000), 259--275

\bibitem{IS}
A.D. Ionescu, G. Staffilani,
\textit{Semilinear Schr\"odinger flows on hyperbolic spaces}\,:
\textit{Scattering in $H^1$\/},
Math. Ann. 345 (2009), 133--158

\bibitem{J}
F. John, 
\textit{Blow--up of solutions of nonlinear wave equations in three space dimensions\/},
Manuscripta Math. 28 (1979), 235--265

\bibitem{Kap}
L. Kapitanski,
\textit{Weak and yet weaker solutions of semilinear wave equations\/},
Comm. Partial Diff. Eq. 19 (1994), 1629--1676

\bibitem{KT}
M. Keel, T. Tao,
\textit{Endpoint Strichartz estimates\/}
Amer. J. Math. 120 (1998), 955--980

\bibitem{KP}
S. Klainerman, G. Ponce,
\textit{Global, small amplitude solutions to nonlinear evolution equations\/},
Comm. Pure Appl. Math. 36 (1983), 133--141

\bibitem{Ko}
T.H. Koornwinder,
\textit{Jacobi functions and analysis on noncompact semisimple Lie groups\/},
in \textit{Special functions\/} (\textit{group theoretical aspects and applications\/}),
R.A. Askey \& al. (eds.), Reidel (1984), 1--85

%\bibitem{L1}
%H. Lindblad,
%\textit{A sharp counterexample to the local existence
%of low--regularity solutions to nonlinear wave equations\/},
%Duke Math. J.  72  (1993), 503--539

%\bibitem{L2}
%H. Lindblad,
%\textit{Counterexamples to local existence for quasilinear wave equations\/},
%Math. Res. Lett. 5 (1998), no. 5, 605--622

%\bibitem{L3}
%H. Lindblad,
%\textit{Minimal regularity solutions of nonlinear wave equations\/},
%Proc. Internat. Congress Math. Berlin 1998, vol. III, Doc. Math. (1998), 39--48

\bibitem{LS}
H. Lindblad, C. Sogge,
\textit{On existence and scattering with minimal regularity
for semilinear wave equations\/}, 
J. Funct. Anal.  130 (1995), 357--426

%\bibitem{Lip} 
%R.L. Lipsman, 
%\textit{An indicator diagram II\/}
%(\textit{the $L_p$ convolution theorem for connected unimodular groups\/},
%Duke Math. J. 37 (1970), 459--466

%\bibitem{Lo}
%N. Lohou\'e, 
%\textit{In\'egalit\'es de Sobolev pour les sous--laplaciens
%de certains groupes unimodulaires\/},
%Geom. Funct. Anal. 2 (1992), 394--420

\bibitem{MNO}
S. Machihara, M. Nakamura, T. Ozawa,
\textit{Small global solutions for nonlinear Dirac equations\/},
Differential Integral Equations 17 (2004), no. 5--6, 623--636

\bibitem{MT}
J. Metcalfe, M.E. Taylor, 
\textit{Nonlinear waves on 3D hyperbolic space\/},
Trans. Amer. Math. Soc. (to appear)

%\bibitem{MT} D. M\"uller, C. Thiele, 
%\textit{ Wave equation and multiplier estimates on $ax+b$ groups}, 
%Studia Math. 179 (2007), 117--148

%\bibitem{MV} D. M\"uller, M. Vallarino, 
%\textit{Wave equation and multiplier estimates on Damek--Ricci spaces\/}, 
%to appear in J. Fourier Anal. Appl.

\bibitem{N}
K. Nakanishi,
\textit{Scattering theory
for the nonlinear Klein--Gordon equation
with Sobolev critical power\/},
Internat. Math. Res. Notices 1 (1999), 31--60

%\bibitem{P1}
%V. Pierfelice,
%\textit{Weighted Strichartz estimates
%for the radial perturbed Schr\"odinger equation
%on the hyperbolic space\/},
%Manuscripta Math. 120 (2006), 377--389

\bibitem{P}
V. Pierfelice,
\textit{Weighted Strichartz estimates
for the Schr\"odinger and wave equations
on Damek--Ricci spaces\/},
Math. Z.  260 (2008), 377--392

%\bibitem{RS}
%Th. Runst, W. Sickel,
%\textit{Sobolev spaces of fractional order, Nemytskij operators,
%and nonlinear partial differential equations\/},
%de Gruyter (1996)

\bibitem{Si}
T. Sideris, 
\textit{Nonexistence of global solutions
to semilinear wave equations in high dimensions\/}, 
J. Diff. Eq. 52 (1984), 378--406

%\bibitem{So}
%C. Sogge,
%\textit{Lectures on non-linear wave equations\/}
%(2nd ed.), International Press (2008)

%\bibitem{St}
%G. Staffilani\,:
%\textit{On the generalized Korteweg--de Vries--type equations\/},
%Differential Integral Equations 10 (1997), 777--796

%\bibitem{ST}
%R.J. Stanton, P.A. Tomas,
%\textit{Expansions for spherical functions on noncompact symmetric spaces\/},
%Acta Math 140 (1978), no. 3--4, 251--276

\bibitem{Stra}
W. Strauss,
\textit{Nonlinear wave equations\/}, 
CBMS Regional Conf. Ser. Math. 73, Amer. Math. Soc. (1989) 

%\bibitem{Sz} Z.I. Szabo,
%\textit{The Lichnerowicz conjecture on harmonic manifolds\/},
%J. Diff. Geom. 31 (1990), 1--28

%\bibitem{T}
%T. Tao,
%\textit{Nonlinear dispersive equations\/} ({\it Local and global analysis\/}),
%CBMS Regional Conf. Ser. Math. 106, Amer. Math. Soc. (2006)

%\bibitem{TVZ}
%T. Tao, M. Visan, X. Zhang,
%\textit{Global well--posedness and scattering
%for the defocusing mass--critical nonlinear Schr\"odinger equation 
%for radial data in high dimensions\/},
%Duke Math. J. 140 (2007), 165--202

\bibitem{Ta}
D. Tataru,
\textit{Strichartz estimates in the hyperbolic space
and global existence for the semilinear wave equation\/},
Trans. Amer. Math. Soc. 353 (2001), 795--807

%\bibitem{Ta1}
%M.\,E. Taylor, 
%\textit{$L^p$ estimates on functions of the Laplace operator\/},
%Duke Math. J. 58 (1989), no. 3, 773--793

%\bibitem{Ta2}
%M.\,E. Taylor,
%\textit{Tools for PDE\/} (\textit{Pseudodifferential operators,
%paradifferential operators, and layer potentials\/}),
%Math. Surveys Monographs 81, Amer. Math. Soc. (2000)

\bibitem{Tr2}
H. Triebel,
\textit{Theory of function spaces II\/},
Monographs Math. 84, Birkh\"auser (1992)

\end{thebibliography}
\end{document}